\documentclass[11pt,reqno]{amsart}

\usepackage{
  mathabx,
  amsmath,
  amsfonts,
  amssymb,
  amsthm,
  amscd,
  graphicx,
  mathtools,
  etoolbox,
  enumitem,
  mathdots,
  txfonts,
  fancybox,
  stackrel
}

\usepackage[nomargin,inline]{fixme}
\makeatletter
\renewcommand*\FXLayoutInline[3]{%
  {\@fxuseface{inline}
  \ignorespaces\noindent \ovalbox{\hspace{.01\textwidth} \begin{minipage}{.95\textwidth}
  	#3 \fxnotename{#1}: #2
  \end{minipage}\hspace{.01\textwidth}}}
  \newline}
\makeatother

\usepackage[all]{xy}
\usepackage[usenames,dvipsnames]{xcolor}

\makeatletter
\@namedef{subjclassname@2020}{%
  \textup{2020} Mathematics Subject Classification}
\makeatother

\usepackage{bbm}

\DeclareFontFamily{OT1}{pzc}{}
\DeclareFontShape{OT1}{pzc}{m}{it}{<-> s * [1.10] pzcmi7t}{}
\DeclareMathAlphabet{\mathpzc}{OT1}{pzc}{m}{it}

\usepackage[colorlinks=true, linkcolor=blue, citecolor=purple, urlcolor=blue, breaklinks=true]{hyperref}

\usepackage{tikz}
\usetikzlibrary{decorations.markings}
\usetikzlibrary{cd}
\usetikzlibrary{patterns}
\usetikzlibrary{shapes}

\tikzset{anchorbase/.style={outer sep=auto,baseline={([yshift=-0.5ex]current bounding box.center)}}}
\tikzset{wipe/.style={white,line width=4pt}}

\tikzset{every picture/.style=semithick}

\colorlet{objcolor}{OliveGreen}

\newcommand{\textupdot}{\begin{tikzpicture}[anchorbase,scale=1.2]
\draw[thin,-to] (0,-.2) to (0,.2);
\opendot{0,-.01};
\end{tikzpicture}\!}

\newcommand{\textdownstar}{\begin{tikzpicture}[anchorbase,scale=1.2]
\draw[thin,to-] (0,-.2) to (0,.2);
\node at (0,.01) {$\star$};
\end{tikzpicture}\!}

\newcommand{\textdowndot}{\begin{tikzpicture}[anchorbase,scale=1.2]
\draw[thin,to-] (0,-.2) to (0,.2);
\opendot{0,.01};
\end{tikzpicture}\!}

\newcommand{\textupcrossing}{\begin{tikzpicture}[anchorbase,scale=1.2]
\draw[thin,-to] (-.15,-.2) to (.15,.2);
\draw[thin,-to] (.15,-.2) to (-.15,.2);
\end{tikzpicture}}

\newcommand{\textdowncrossing}{\begin{tikzpicture}[anchorbase,scale=1.2]
\draw[thin,to-] (-.15,-.2) to (.15,.2);
\draw[thin,to-] (.15,-.2) to (-.15,.2);
\end{tikzpicture}}

\newcommand\textclockwisebubble{\begin{tikzpicture}[anchorbase,scale=.85]
\clockwisebubble{(0,0)};
\end{tikzpicture}}

\newcommand\textanticlockwisebubble{\begin{tikzpicture}[anchorbase,scale=.85]
\anticlockwisebubble{(0,0)};
\end{tikzpicture}}

\newcommand{\textrightcap}{\begin{tikzpicture}[anchorbase,scale=1.2]
\draw[thin,-to] (-.15,-.15) to [out=90,in=180] (0,.15) to [out=0,in=90] (.15,-.15);
\end{tikzpicture}}

\newcommand{\textrightcup}{\begin{tikzpicture}[anchorbase,scale=1.2]
\draw[thin,-to] (-.15,.15) to [out=-90,in=180] (0,-.15) to [out=0,in=-90] (.15,.15);
\end{tikzpicture}}

\newcommand{\textdot}{\begin{tikzpicture}[anchorbase,scale=1.2]
\draw[-] (0,-.2) to (0,.2);
\closeddot{0,0};
\end{tikzpicture}\!\!}

\newcommand{\textcrossing}{\begin{tikzpicture}[anchorbase,scale=1.2]
\draw[-] (-.15,-.2) to (.15,.2);
\draw[-] (.15,-.2) to (-.15,.2);
\end{tikzpicture}}

\newcommand{\txtcap}{\begin{tikzpicture}[anchorbase,scale=1.2]
\draw[-] (-.15,-.15) to [out=90,in=180] (0,.15) to [out=0,in=90] (.15,-.15);
\end{tikzpicture}}

\newcommand{\txtcup}{\begin{tikzpicture}[anchorbase,scale=1.2]
\draw[-] (-.15,.15) to [out=-90,in=180] (0,-.15) to [out=0,in=-90] (.15,.15);
\end{tikzpicture}}

\newcommand{\limitteleporterOO}[5][black]{
\path #2 node[inner sep=0pt] (x) {${\color{white}\bullet}\hspace{-2mm}\color{#1}\circ$};
\path #5 node[inner sep=0pt] (z) {${\color{white}\bullet}\hspace{-2mm}\color{#1}\circ$};
\draw[Triangle Cap-Triangle Cap,thick,#1!60!white] (x)--node[above,inner sep=1pt,very near start]{$\color{#1!60!white}\scriptstyle#3$} node[above,inner sep=1pt,very near end]{$\color{#1!60!white}\scriptstyle#4$}
(z);
}

\newcommand{\bentlimitteleporterOO}[6][black]{
\path #2 node[inner sep=0pt] (x) {${\color{white}\bullet}\hspace{-2mm}\color{#1}\circ$};
\path #5 node[inner sep=0pt] (z) {${\color{white}\bullet}\hspace{-2mm}\color{#1}\circ$};
\draw[Triangle Cap-Triangle Cap,thick,#1!60!white] (x)..controls#6..node[above,inner sep=1pt,very near start]{$\color{#1!60!white}\scriptstyle#3$} node[above,inner sep=1pt,very near end]{$\color{#1!60!white}\scriptstyle#4$}
(z);
}

\newcommand{\limitteleporteroo}[5][black]{
\path #2 node[inner sep=0pt] (x) {${\color{white}\scriptstyle\bullet}\hspace{-1.42mm}\color{#1}\scriptstyle\circ$};
\path #5 node[inner sep=0pt] (z) {${\color{white}\scriptstyle\bullet}\hspace{-1.42mm}\color{#1}\scriptstyle\circ$};
\draw[Triangle Cap-Triangle Cap,thick,#1!60!white] (x)--node[above,inner sep=1pt,very near start]{$\color{#1!60!white}\scriptstyle#3$} node[above,inner sep=1pt,very near end]{$\color{#1!60!white}\scriptstyle#4$}
(z);
}

\newcommand{\bentlimitteleporteroo}[6][black]{
\path #2 node[inner sep=0pt] (x) {${\color{white}\scriptstyle\bullet}\hspace{-1.42mm}\color{#1}\scriptstyle\circ$};
\path #5 node[inner sep=0pt] (z) {${\color{white}\scriptstyle\bullet}\hspace{-1.42mm}\color{#1}\scriptstyle\circ$};
\draw[Triangle Cap-Triangle Cap,thick,#1!60!white] (x)..controls#6..node[above,inner sep=1pt,very near start]{$\color{#1!60!white}\scriptstyle#3$} node[above,inner sep=1pt,very near end]{$\color{#1!60!white}\scriptstyle#4$}
(z);
}

\newcommand{\limitbandOO}[5][black]{
\path #2 node[inner sep=0pt] (x) {${\color{white}\bullet}\hspace{-2mm}\color{#1}\circ$};
\path #5 node[inner sep=0pt] (z) {${\color{white}\bullet}\hspace{-2mm}\color{#1}\circ$};
\draw[-,densely dotted,#1!60!white] (x)-- node[above,very near start,inner sep=1pt]{$\color{#1!60!white}\scriptstyle#3$} node[above,inner sep=1pt,very near end]{$\color{#1!60!white}\scriptstyle#4$} (z);
}

\newcommand{\bentlimitbandOO}[6][black]{
\path #2 node[inner sep=0pt] (x) {${\color{white}\bullet}\hspace{-2mm}\color{#1}\circ$};
\path #5 node[inner sep=0pt] (z) {${\color{white}\bullet}\hspace{-2mm}\color{#1}\circ$};
\draw[-,densely dotted,#1!60!white] (x)..controls#6..node[above,inner sep=1pt,very near start]{$\color{#1!60!white}\scriptstyle#3$} node[above,inner sep=1pt,very near end]{$\color{#1!60!white}\scriptstyle#4$} (z);
}

\newcommand{\limitbandoo}[5][black]{
\path #2 node[inner sep=0pt] (x) 
{${\color{white}\scriptstyle\bullet}\hspace{-1.4mm}\!\color{#1}\scriptstyle\circ$};
\path #5 node[inner sep=0pt] (z) 
{${\color{white}\scriptstyle\bullet}\hspace{-1.4mm}\!\color{#1}\scriptstyle\circ$};
\draw[-,densely dotted,#1!60!white] (x)-- node[above,very near start,inner sep=1pt]{$\color{#1!60!white}\scriptstyle#3$} node[above,inner sep=1pt,very near end]{$\color{#1!60!white}\scriptstyle#4$} (z);
}

\newcommand{\bentlimitbandoo}[6][black]{
\path #2 node[inner sep=0pt] (x) 
{${\color{white}\scriptstyle\bullet}\hspace{-1.42mm}\color{#1}\scriptstyle\circ$};
\path #5 node[inner sep=0pt] (z) 
{${\color{white}\scriptstyle\bullet}\hspace{-1.42mm}\color{#1}\scriptstyle\circ$};
\draw[-,densely dotted,#1!60!white] (x)..controls#6..node[above,inner sep=1pt,very near start]{$\color{#1!60!white}\scriptstyle#3$} node[above,inner sep=1pt,very near end]{$\color{#1!60!white}\scriptstyle#4$} (z);
}

\newcommand{\pinO}[5][black]{
\path #2 node[inner sep=0pt] (x) {${\color{white}\bullet}\hspace{-2mm}\color{#1}\circ$};
\path #4 node[rectangle,rounded corners,draw,#1!60!white,fill=#1!10!white,inner sep=2.5pt](y) {$\color{#1}\scriptstyle#5$};
\draw[Triangle Cap-,thick,#1!60!white] (x)-- node[above,inner sep=1pt]{$\color{#1!60!white}\scriptstyle#3$}(y);
}

\newcommand{\pino}[5][black]{
\path #2 node[inner sep=0pt] (x) {${\color{white}\scriptstyle\bullet}\hspace{-1.42mm}\color{#1}\scriptstyle\circ$};
\path #4 node[rectangle,rounded corners,draw,#1!60!white,fill=#1!10!white,inner sep=2.5pt](y) {$\color{#1}\scriptstyle#5$};
\draw[Triangle Cap-,thick,#1!60!white] (x)-- node[above,inner sep=1pt]{$\color{#1!60!white}\scriptstyle#3$}(y);
}

\newcommand{\pinX}[5][black]{
\path #2 node[inner sep=0pt] (x) {$\color{#1}\bullet$};
\path #4 node[rectangle,rounded corners,draw,#1!60!white,fill=#1!10!white,inner sep=2.5pt](y) {$\color{#1}\scriptstyle#5$};
\draw[Triangle Cap-,thick,#1!60!white] (x)-- node[above,inner sep=1pt]{$\color{#1!60!white}\scriptstyle#3$}(y);
}

\newcommand{\bentpinX}[6][black]{
\path #2 node[inner sep=0pt] (x) {$\color{#1}\bullet$};
\path #4 node[rectangle,rounded corners,draw,#1!60!white,fill=#1!10!white,inner sep=2.5pt](y) {$\color{#1}\scriptstyle#5$};
\draw[Triangle Cap-,thick,#1!60!white] (x)..controls#6.. node[above,inner sep=1pt]{$\color{#1!60!white}\scriptstyle#3$}(y);
}

\newcommand\clockwisebubble[2][black]{%
\draw[-to,thin,#1] #2++(0,.2) arc(90:-270:0.2);
}
\newcommand\anticlockwisebubble[2][black]{%
  \draw[-to,thin,#1] #2++(0,.2) arc(90:450:0.2);
}
\newcommand\anticlockwisebubbleflip[2][black]{%
  \draw[-to,thin,#1] #2++(0,-.2) arc(-90:270:0.2);
}
\newcommand\filledclockwisebubble[2][black]{%
\draw[-to,#1,thin,fill=#1!10!white] #2++(0,.2) arc(90:-270:0.2);
}
\newcommand\filledanticlockwisebubble[2][black]{%
  \draw[-to,#1,thin,fill=#1!10!white] #2++(0,.2) arc(90:450:0.2);
}
\newcommand\anticlockwiseinternalbubble[2][black]{%
  \draw[-to,#1,thin,fill=#1!10!white] #2++(-.2,0) arc(-180:180:0.2);
}
\newcommand\clockwiseinternalbubble[2][black]{%
  \draw[-to,#1,thin,fill=#1!10!white] #2++(.2,0) arc(0:-360:0.2);
}

\colorlet{catcolor}{OliveGreen}
\def\catlabel#1{{\scriptstyle\color{catcolor}#1}}
\def\kmlabel#1{{\scriptstyle#1}}

\newcommand\closeddot[1]{\node at (#1) {$\bullet$}}
\newcommand\opendot[1]{\node at (#1)   {${\color{white}\bullet}\hspace{-2mm}\circ$}} 
\newcommand\opendotsmall[1]{\node at (#1)   {${\color{white}\scriptstyle\bullet}\hspace{-1.42mm}\scriptstyle\circ$}} 

\usepackage{cleveref}

\crefname{definition}{Definition}{Definitions}
\crefname{example}{Example}{Examples}
\crefname{lemma}{Lemma}{Lemmas}
\crefname{conjecture}{Conjecture}{Conjectures}
\crefname{corollary}{Corollary}{Corollaries}
\crefname{theorem}{Theorem}{Theorems}
\crefname{remark}{Remark}{Remarks}
\crefname{equation}{}{}
\crefname{enumi}{}{}
\crefname{section}{Section}{Section}
\crefname{subsection}{Subsection}{Section}

\newcommand\Z{\mathbb{Z}}

\newcommand\kk{\Bbbk}

\def\O{\mathbb{O}}

\def\pmod#1{\;(\operatorname{mod}\,#1)}\def\wt{\operatorname{wt}}
\def\op{{\operatorname{op}}}
\def\rev{{\operatorname{rev}}}
\newcommand\one{\mathbbm{1}}

\def\sl{\mathfrak{sl}}
\def\RSD{\mathrm{D}}

\def\loc{{\begin{tikzpicture}[baseline=.2mm]
\path (0,0) node[inner sep=0pt] (x) {$\scriptstyle\circ$};
\path (0.32,0) node[inner sep=0pt] (y) {$\scriptstyle\circ$};
\draw[-,densely dotted,black!60!white] (x)--(y);
\end{tikzpicture}}}

\newcommand\cNB{\mathpzc{NB}}

\newcommand\cU{\mathpzc{U}}
\newcommand\fU{\mathfrak{U}}
\newcommand\cUc{\widehat{\mathpzc{U}}}
\newcommand\fUc{\widehat{\mathfrak{U}}}

\newcommand\tT{{\scriptstyle\mathrm{T}}}
\newcommand\tR{{\scriptstyle\mathrm{R}}}
\newcommand\tD{{\scriptstyle\mathrm{D}}}
\newcommand{\ndots}{\bullet}

\def\lround{(\!(}
\def\rround{)\!)}

\DeclareMathOperator{\End}{End}

\DeclareMathOperator{\Hom}{Hom}

\def\Add{\operatorname{Add}}

\def\X{\mathrm{X}}

\newcommand{\U}{{\mathrm U}}

\newcommand{\UAi}{{_\Z}{\mathrm U}_t^{\imath}}

\newtheorem{theorem}{Theorem}[section]
\newtheorem{lemma}[theorem]{Lemma}
\newtheorem*{lemma*}{Lemma}
\newtheorem{corollary}[theorem]{Corollary}

\theoremstyle{definition}
\newtheorem{definition}[theorem]{Definition}
\newtheorem{remark}[theorem]{Remark}

\numberwithin{equation}{section}
\allowdisplaybreaks

\begin{document}
%===============

\title{The nil-Brauer category}

\author{Jonathan Brundan}
\address[J.B.]{
  Department of Mathematics,
  University of Oregon,
  Eugene, OR, USA
}
\email{brundan@uoregon.edu}

\author{Weiqiang Wang}
\address[W.W.]{Department of Mathematics, University of Virginia, Charlottesville, VA, USA}
\email{ww9c@virginia.edu}

\author{Ben Webster}
\address[B.W.]{Department of Pure Mathematics, University of Waterloo \& Perimeter Institute for Theoretical Physics,
Waterloo, ON, Canada}
\email{ben.webster@uwaterloo.ca}

\thanks{J.B. supported in part by NSF grant
  DMS-2101783. W.W. supported in part by the NSF grant DMS-2001351. B.W. supported by Discovery Grant RGPIN-2018-03974 from the
  Natural Sciences and Engineering Research Council of Canada. This research was also supported by Perimeter Institute for Theoretical Physics. Research at Perimeter Institute is supported by the Government of Canada through the Department of Innovation, Science and Economic Development and by the Province of Ontario through the Ministry of Research and Innovation.
}


\begin{abstract}
We introduce the nil-Brauer category and prove a basis theorem for its morphism spaces. This basis theorem is an essential ingredient required to prove that nil-Brauer categorifies the split iquantum group of rank one. As this iquantum group is a basic building block for iquantum groups of higher rank, we expect that the nil-Brauer category will play a central role in future developments related to the categorification of quantum symmetric pairs.
\end{abstract}

\maketitle

%=====================
% Introduction
%=====================

\section{Introduction}\label{sec1}

Throughout the article, we work over an integral domain $\kk$ in which 2 is invertible, and all categories, functors, etc. will be assumed to be $\kk$-linear without further mention. The aim of the article is to introduce a new strict graded monoidal category, the {\em nil-Brauer category}, denoted $\cNB_t$ for a parameter $t \in \kk$.
It turns out that $\cNB_t$ is non-trivial only for $t=0$ or $t=1$; we assume this is the case for the remainder of the introduction.

The importance of the nil-Brauer category 
stems from results established in our subsequent work \cite{BWW} relating $\cNB_t$ to $\mathrm{U}_q^\imath(\sl_2)$,
the split iquantum group of rank one
corresponding to the symmetric pair $(\mathrm{SL}_2,\mathrm{SO}_2)$.
Roughly speaking, 
the monoidal categories $\cNB_t\:(t=0,1)$ play the same role
for this, the most basic of all iquantum groups, as 
the strict graded
2-category $\fU(\sl_2)$ introduced in \cite{Rou,Lauda} 
plays for the ordinary quantum group 
$\mathrm{U}_q(\sl_2)$.
To make a more precise statement,
let $\UAi$ be the modified $\Z[q,q^{-1}]$-form
of 
$\mathrm{U}_q^\imath(\sl_2)$
associated to weights of parity $t \in \{0,1\}$
arising as a special case of the constructions in \cite{BW18QSP,BW18KL}. 
The algebra $\mathrm{U}_q^\imath(\sl_2)$ is simply a polynomial
algebra in one variable $B$, but the integral form $\UAi$ is not at all obvious; it has a basis as a free $\Z[q,q^{-1}]$-module given by the icanonical basis which was computed explicitly in \cite{BeW18}.
We will show
in \cite{BWW} that $\cNB_t$
categorifies $\UAi$ in the sense that the split Grothendieck ring of its graded Karoubi envelope, viewed as a $\Z[q,q^{-1}]$-algebra with the action of $q$ coming from the grading shift functor, is isomorphic to $\UAi$, with the icanonical basis arising from isomorphism classes of indecomposable objects.

The main theorem about $\cNB_t$ proved in the present article gives explicit bases for morphism spaces in $\cNB_t$, a result which is used in an essential way in \cite{BWW}.
Our proof follows a similar strategy to the approach developed for Khovanov's Heisenberg category in \cite{K0}, exploiting a remarkable monoidal functor 
$$
\Omega_t:
\cNB_t \longrightarrow \Add\!\big(\cU(\sl_2;t)_\loc\big)
$$
from $\cNB_t$ to the additive envelope of a monoidal category $\cU(\sl_2;t)_\loc$ obtained by localizing the 2-category $\fU(\sl_2)$ at certain morphisms. This functor takes the generating object $B$ of $\cNB_t$ to the direct sum $E \oplus F$ of the generating objects of $\cU(\sl_2;t)_\loc$. It can be interpreted (in a weak sense due to the need to localize)
as a categorification of the natural embedding of $\UAi$ into the completion 
$\prod_{\lambda,\mu \equiv t\pmod{2}}
1_\mu\;{_\Z}\dot \U\, 1_\lambda$
of Lusztig's modified $\Z[q,q^{-1}]$-form of the quantum group $\mathrm{U}_q(\sl_2)$. Although the construction of $\Omega_t$ is elementary, it depends on an astonishingly delicate computation with generators and relations expressed in terms of string calculus.

The rest of the article is organized as follows. 
\begin{itemize}
\item
We begin in \cref{sec2} by defining $\cNB_t$ by generators and relations; see \cref{NBdef}. Then we show that it is trivial unless $t \in \{0,1\}$. For these values of $t$, we construct a homomorphism
$$
\gamma_t:\Gamma \rightarrow \End_{\cNB_t}(\one)
$$
where $\Gamma$ is the subalgebra of the algebra of
symmetric functions over $\kk$ generated by
Schur's $Q$-functions. 
\item
In \cref{sec3}, we recall the definition of the 2-category $\fU(\sl_2)$, and then introduce a localized version of it, denoted $\fU(\sl_2)_\loc$. 
In particular, this localization adjoins inverses of the 2-morphisms usually denoted by dots.
We then derive some remarkable relations in $\fU(\sl_2)_\loc$ involving certain 2-morphisms which we call {\em internal bubbles} which are analogous to, but more complicated than, corresponding morphisms for the Heisenberg category introduced in \cite{K0}. 
\item
Then in \cref{sec4} we pass from the 2-category $\fU(\sl_2)_\loc$ to the monoidal category $\cU(\sl_2;t)_\loc$. The objects in the latter are words $X$ in $E$ and $F$, corresponding to formal sums of the 1-morphisms $X 1_\lambda$ in $\fU(\sl_2)_\loc$ for the same word $X$ and all weights $\lambda\equiv t \pmod{2}$, and its morphisms are represented by sequences of 2-morphisms in $\fU(\sl_2)_\loc$; see \cref{hearts} for the complete definition. The monoidal functor $\Omega_t$ is finally constructed in \cref{psit}.
\item
In \cref{sec5}, we use this functor to prove our main basis theorem, \cref{basisthm}. This result is analogous to the
basis theorem establishing non-degeneracy of the Kac-Moody 2-category $\fU(\sl_n)$ in \cite{KL3}.
It shows that the homomorphism 
$\gamma_t$ mentioned earlier is actually  an 
{\em isomorphism}, i.e.,
$\End_{\cNB_t}(\one)\cong \Gamma$.
Moreover, each morphism space 
$\Hom_{\cNB_t}(B^{\star n}, B^{\star m})$
is free as a graded $\Gamma$-module with an explicit homogeneous basis. 
\end{itemize}

Now we can give some further justification of the importance of the basis theorem proved here:
in the sequel \cite{BWW}, we show
that the graded rank of 
$\Hom_{\cNB_t}(B^{\star n}, B^{\star m})$
as a free $\Gamma$-module is
equal to $(B^n,B^m)^\iota$, 
where $(\cdot,\cdot)^\iota$ is the 
non-degenerate
symmetric bilinear form on the iquantum group
$\UAi$ from 
\cite[Lem.~6.25]{BW18QSP}. We regard this as
the first indication that
$\cNB_t$ categorifies
$\UAi$, indeed, it provided us with some initial clues
as to the precise form of the generators and relations for $\cNB_t$ in \cref{NBdef}.

\vspace{2mm}
\noindent{\em Acknowledgements.} We thank Peng Shan for her interest and encouragement in the early stages of this project.

%=====================
% Section 2
%=====================

\section{Definition and first properties of the nil-Brauer category}\label{sec2}

The definition of $\cNB_t$ is by generators and relations.
We use the string calculus to denote morphisms in
 strict monoidal categories (or 2-morphisms in strict 2-categories), our general convention being that composition $f \circ g$ (``vertical composition")
 is depicted by placing $f$ on top of $g$
 and tensor product $f \star g$ (``horizontal composition")
 is depicted by placing $f$ to the left of $g$.

\begin{definition}\label{NBdef}
For $t \in \kk$, the {\em nil-Brauer category}
 $\cNB_t$
is the strict monoidal category
with one generating object $B$, whose identity endomorphism will be
represented diagrammatically by the unlabeled string
$\;\begin{tikzpicture}[anchorbase]\draw[-] (0,-.2) to (0,.2);
\end{tikzpicture}\;$,
and four generating morphisms
\begin{align}\label{gens}
\begin{tikzpicture}[anchorbase]
	\draw[-] (0.08,-.3) to (0.08,.4);
    \closeddot{0.08,.05};
\end{tikzpicture}
&:B\rightarrow B,&
\begin{tikzpicture}[anchorbase]
\draw[-] (0.28,-.3) to (-0.28,.4);
	\draw[-] (-0.28,-.3) to (0.28,.4);
\end{tikzpicture}
&:B\star B \rightarrow B \star B,
&
\begin{tikzpicture}[anchorbase]
	\draw[-] (0.4,0) to[out=90, in=0] (0.1,0.4);
	\draw[-] (0.1,0.4) to[out = 180, in = 90] (-0.2,0);
\end{tikzpicture}
&:B \star B\rightarrow\one,&
\begin{tikzpicture}[anchorbase]
	\draw[-] (0.4,0.4) to[out=-90, in=0] (0.1,0);
	\draw[-] (0.1,0) to[out = 180, in = -90] (-0.2,0.4);
\end{tikzpicture}
&:\one\rightarrow B\star B,
\\\notag
&\text{(degree $2$)}&
&\text{(degree $-2$)}&
&\text{(degree $0$)}&
&\text{(degree $0$)}
\end{align}
subject to the following relations:
\begin{align}\label{rels1}
\begin{tikzpicture}[anchorbase]
	\draw[-] (0.28,0) to[out=90,in=-90] (-0.28,.6);
	\draw[-] (-0.28,0) to[out=90,in=-90] (0.28,.6);
	\draw[-] (0.28,-.6) to[out=90,in=-90] (-0.28,0);
	\draw[-] (-0.28,-.6) to[out=90,in=-90] (0.28,0);
\end{tikzpicture}
&=0,
&\begin{tikzpicture}[anchorbase]
	\draw[-] (0.45,.6) to (-0.45,-.6);
	\draw[-] (0.45,-.6) to (-0.45,.6);
        \draw[-] (0,-.6) to[out=90,in=-90] (-.45,0);
        \draw[-] (-0.45,0) to[out=90,in=-90] (0,0.6);
\end{tikzpicture}
&=
\begin{tikzpicture}[anchorbase]
	\draw[-] (0.45,.6) to (-0.45,-.6);
	\draw[-] (0.45,-.6) to (-0.45,.6);
        \draw[-] (0,-.6) to[out=90,in=-90] (.45,0);
        \draw[-] (0.45,0) to[out=90,in=-90] (0,0.6);
\end{tikzpicture}\:,\\
\label{rels2}
\begin{tikzpicture}[baseline=-2.5mm]
\draw (0,-.15) circle (.3);
\end{tikzpicture}
&= t 1_\one,
&\begin{tikzpicture}[anchorbase]
  \draw[-] (0.3,0) to (0.3,.4);
	\draw[-] (0.3,0) to[out=-90, in=0] (0.1,-0.4);
	\draw[-] (0.1,-0.4) to[out = 180, in = -90] (-0.1,0);
	\draw[-] (-0.1,0) to[out=90, in=0] (-0.3,0.4);
	\draw[-] (-0.3,0.4) to[out = 180, in =90] (-0.5,0);
  \draw[-] (-0.5,0) to (-0.5,-.4);
\end{tikzpicture}
&=
\begin{tikzpicture}[anchorbase]
  \draw[-] (0,-0.4) to (0,.4);
\end{tikzpicture}
=\begin{tikzpicture}[anchorbase]
  \draw[-] (0.3,0) to (0.3,-.4);
	\draw[-] (0.3,0) to[out=90, in=0] (0.1,0.4);
	\draw[-] (0.1,0.4) to[out = 180, in = 90] (-0.1,0);
	\draw[-] (-0.1,0) to[out=-90, in=0] (-0.3,-0.4);
	\draw[-] (-0.3,-0.4) to[out = 180, in =-90] (-0.5,0);
  \draw[-] (-0.5,0) to (-0.5,.4);
\end{tikzpicture}
\:,\\\label{rels3}
\begin{tikzpicture}[anchorbase,scale=1.1]
	\draw[-] (0.35,.3)  to [out=90,in=-90] (-.1,.9) to [out=90,in=180] (.1,1.1);
 \draw[-] (-.15,.3)  to [out=90,in=-90] (.3,.9) to [out=90,in=0] (.1,1.1);
\end{tikzpicture}
&=0
\:,&
\begin{tikzpicture}[anchorbase,scale=1.1]
	\draw[-] (0.5,0) to[out=90, in=0] (0.1,0.5);
	\draw[-] (0.1,0.5) to[out = 180, in = 90] (-0.3,0);
 \draw[-] (0.1,0) to[out=90,in=-90] (-.3,.7);
\end{tikzpicture}
&=
\begin{tikzpicture}[anchorbase,scale=1.1]
	\draw[-] (0.5,0) to[out=90, in=0] (0.1,0.5);
	\draw[-] (0.1,0.5) to[out = 180, in = 90] (-0.3,0);
 \draw[-] (0.1,0) to[out=90,in=-90] (.5,.7);
\end{tikzpicture}\:,
\\
\label{rels4}
\begin{tikzpicture}[anchorbase,scale=1.4]
	\draw[-] (0.25,.3) to (-0.25,-.3);
	\draw[-] (0.25,-.3) to (-0.25,.3);
 \closeddot{-0.12,0.145};
\end{tikzpicture}
-
\begin{tikzpicture}[anchorbase,scale=1.4]
	\draw[-] (0.25,.3) to (-0.25,-.3);
	\draw[-] (0.25,-.3) to (-0.25,.3);
     \closeddot{.12,-0.145};
\end{tikzpicture}
&=
\begin{tikzpicture}[anchorbase,scale=1.4]
 	\draw[-] (0,-.3) to (0,.3);
	\draw[-] (-.3,-.3) to (-0.3,.3);
\end{tikzpicture}
-
\begin{tikzpicture}[anchorbase,scale=1.4]
 	\draw[-] (-0.15,-.3) to[out=90,in=180] (0,-.1) to[out=0,in=90] (0.15,-.3);
 	\draw[-] (-0.15,.3) to[out=-90,in=180] (0,.1) to[out=0,in=-90] (0.15,.3);
\end{tikzpicture}
\:,&
\begin{tikzpicture}[anchorbase,scale=1.1]
	\draw[-] (0.4,0) to[out=90, in=0] (0.1,0.5);
	\draw[-] (0.1,0.5) to[out = 180, in = 90] (-0.2,0);
 \closeddot{.38,.2};
\end{tikzpicture}
&=
-\begin{tikzpicture}[anchorbase,scale=1.1]
	\draw[-] (0.4,0) to[out=90, in=0] (0.1,0.5);
	\draw[-] (0.1,0.5) to[out = 180, in = 90] (-0.2,0);
 \closeddot{-.18,.2};
\end{tikzpicture}
\:.
\end{align}
\end{definition}

\begin{remark} \label{NBgrading}
Although it will only play a secondary role in this article, it is important 
for the sequel to note that $\cNB_t$
can be viewed as a {\em graded} monoidal category, i.e., a monoidal category enriched in
the closed symmetric monoidal category of graded
$\kk$-modules, by declaring that the generators
are of the degrees indicated in parentheses in \cref{gens}.
\end{remark}

\begin{remark}
The defining relations just recorded are quite familiar in related settings. The first two
relations
\cref{rels1} are the same as defining relations in the nil-Hecke algebras associated to symmetric groups, but the polynomial generator
of the nil-Hecke algebra (often depicted also by a dot) satisfies slightly different relations to \cref{rels4}.
These relations come instead from the defining relations for the {\em affine Brauer category} of \cite{RS}, which is 
the strict monoidal category defined in the same way as in 
\cref{NBdef} but replacing the $0$ on the right-hand side of the first relation in \cref{rels1} by the identity (so that the crossing is an involution) 
and the $0$ on the right-hand side of the first relation in \cref{rels3} 
by $\:\txtcap\:$. The
endomorphism algebras of
objects in the affine Brauer category were
introduced earlier by Nazarov \cite[Sec.~4]{Nazarov}, and are known as
Nazarov-Wenzl algebras \cite{AMR}, degenerate affine BMW algebras \cite{DRV} or affine VW algebras \cite{ES}.
The critical sign in the final relation from \cref{rels4} emerged in that setting from considerations involving orthogonal groups (and is unrelated to superalgebra).
\end{remark}

As usual with definitions by generators and relations, the first task is to derive further relations as consequences of the defining ones.
To start with, we have the following:
\begin{align}\label{rels5}
\begin{tikzpicture}[anchorbase,scale=1.1]
	\draw[-] (0.5,0) to[out=-90, in=0] (0.1,-0.5);
	\draw[-] (0.1,-0.5) to[out = 180, in = -90] (-0.3,0);
 \draw[-] (0.1,0) to[out=-90,in=90] (-.3,-.7);
\end{tikzpicture}
&=
\begin{tikzpicture}[anchorbase,scale=1.1]
	\draw[-] (0.5,0) to[out=-90, in=0] (0.1,-0.5);
	\draw[-] (0.1,-0.5) to[out = 180, in = -90] (-0.3,0);
 \draw[-] (0.1,0) to[out=-90,in=90] (.5,-.7);
\end{tikzpicture}\:,&
\begin{tikzpicture}[anchorbase]
	\draw[-] (0,-.3)  to (0,-.1) to [out=90,in=180] (.3,.3) to [out=0,in=90] (.5,.1) to[out=-90,in=0] (.3,-.1) to [out=180,in=-90] (0,.3) to (0,.5);
\end{tikzpicture}&=
0=\begin{tikzpicture}[anchorbase]
	\draw[-] (0,-.3)  to (0,-.1) to [out=90,in=0] (-.3,.3) to [out=180,in=90] (-.5,.1) to[out=-90,in=180] (-.3,-.1) to [out=0,in=-90] (0,.3) to (0,.5);
\end{tikzpicture}\:,\\
\label{rels6}
\begin{tikzpicture}[anchorbase,scale=1.1]
	\draw[-] (0.35,-.3)  to [out=-90,in=90] (-.1,-.9) to [out=-90,in=180] (.1,-1.1);
 \draw[-] (-.15,-.3)  to [out=-90,in=90] (.3,-.9) to [out=-90,in=0] (.1,-1.1);
\end{tikzpicture}
&=0,&
\begin{tikzpicture}[anchorbase,scale=1.4]
 	\draw[-] (-0.25,-.3) to[out=90,in=180] (0,.1) to[out=0,in=90] (0.25,-.3);
 	\draw[-] (-0.25,.3) to[out=-90,in=180] (0,-.1) to[out=0,in=-90] (0.25,.3);
\end{tikzpicture}&=0,
\\
\label{rels7}
\begin{tikzpicture}[anchorbase,scale=1.4]
	\draw[-] (0.25,.3) to (-0.25,-.3);
	\draw[-] (0.25,-.3) to (-0.25,.3);
 \closeddot{-0.12,-0.145};
\end{tikzpicture}
-
\begin{tikzpicture}[anchorbase,scale=1.4]
	\draw[-] (0.25,.3) to (-0.25,-.3);
	\draw[-] (0.25,-.3) to (-0.25,.3);
     \closeddot{.12,0.145};
\end{tikzpicture}
&=
\begin{tikzpicture}[anchorbase,scale=1.4]
 	\draw[-] (0,-.3) to (0,.3);
	\draw[-] (-.3,-.3) to (-0.3,.3);
\end{tikzpicture}
-
\begin{tikzpicture}[anchorbase,scale=1.4]
 	\draw[-] (-0.15,-.3) to[out=90,in=180] (0,-.1) to[out=0,in=90] (0.15,-.3);
 	\draw[-] (-0.15,.3) to[out=-90,in=180] (0,.1) to[out=0,in=-90] (0.15,.3);
\end{tikzpicture}
\:,&
\begin{tikzpicture}[anchorbase,scale=1.1]
	\draw[-] (0.4,0) to[out=-90, in=0] (0.1,-0.5);
	\draw[-] (0.1,-0.5) to[out = 180, in = -90] (-0.2,0);
 \closeddot{.38,-.2};
\end{tikzpicture}
&=
-\begin{tikzpicture}[anchorbase,scale=1.1]
	\draw[-] (0.4,0) to[out=-90, in=0] (0.1,-0.5);
	\draw[-] (0.1,-0.5) to[out = 180, in = -90] (-0.2,0);
 \closeddot{-.18,-.2};
\end{tikzpicture}
\:.
\end{align}
For example, to prove the first of these, we attach cups to the bottom left and bottom right of the second relation in \cref{rels3} to obtain
$$
\begin{tikzpicture}[anchorbase,scale=.8]
\draw[-] (-1.1,.7) to (-1.1,0) to [out=-90,in=180] (-.7,-.4) to [out=0,in=-90] (-.3,0);
\draw[-] (1.3,.7) to (1.3,0) to [out=-90,in=0] (.9,-.4) to [out=180,in=-90] (.5,0);
	\draw[-] (0.5,0) to[out=90, in=0] (0.1,0.5);
	\draw[-] (0.1,0.5) to[out = 180, in = 90] (-0.3,0);
 \draw[-] (0.1,-.5) to (0.1,-.1) to[out=90,in=-90] (-.3,.7);
\end{tikzpicture}
=
\begin{tikzpicture}[anchorbase,scale=.8]
\draw[-] (-1.1,.7) to (-1.1,0) to [out=-90,in=180] (-.7,-.4) to [out=0,in=-90] (-.3,0);
\draw[-] (1.3,.7) to (1.3,0) to [out=-90,in=0] (.9,-.4) to [out=180,in=-90] (.5,0);
	\draw[-] (0.5,0) to[out=90, in=0] (0.1,0.5);
	\draw[-] (0.1,0.5) to[out = 180, in = 90] (-0.3,0);
 \draw[-] (0.1,-.5) to (0.1,-.1) to[out=90,in=-90] (.5,.7);
\end{tikzpicture}\;.
$$
This can then be simplified using the zigzag relations from \cref{rels2} to obtain the desired relation.
The other relations in \cref{rels5,rels6,rels7} are proved
similarly by attaching cups and/or caps to the ends of the strings in other defining relations then simplifying in obvious ways.

Now we take the first relation from \cref{rels4} and close at the top by attaching a cap and at the bottom by attaching a cup.
The left-hand side is 0 due to the first relations from \cref{rels3,rels6}. After replacing the bubbles on the right-hand side with $t 1_\one$, we deduce from this that
\begin{equation}\label{rels8}
t^2 1_\one = t 1_\one.
\end{equation}
Thus, for $\cNB_t$ to be non-trivial, one must 
have that $t \in \{0,1\}$. This will be assumed from now on.
In fact, henceforth, we will treat $t$ as though it is an element of $\{0,1\} \subset \Z$ (rather than being the image of that integer in $\kk$) so that we can use convenient
expressions like $(-1)^t$.

The relations established so far imply that 
there are strict monoidal functors
\begin{align}
\tR:\cNB_t &\rightarrow \cNB_t^{\rev}, &B &\mapsto B, & s &\mapsto (-1)^{\ndots(s)} s^{\leftrightarrow},
\label{R}\\\label{T}
\tT:\cNB_t &\rightarrow \cNB_t^{\op}, &B &\mapsto B, & s &\mapsto
s^{\updownarrow}.
\end{align}
Here, for a string diagram $s$ 
we use $s^{\updownarrow}$ 
and $s^{\leftrightarrow}$ to denote
its reflection in a horizontal or vertical axis,
and $\ndots(s)$ denotes the total number of dots
in the diagram.
The category $\cNB_t$ is {\em strictly pivotal} with duality functor $\tD := \tR \circ \tT = \tT \circ \tR$; this rotates a string diagram $s$ through $180^\circ$ then scales by $(-1)^{\ndots(s)}$. 
Also by the relations established so far, a string diagram with no dots can be deformed freely under planar isotopy without changing the morphism that it represents. This is not true in the presence of dots due to the sign in the last relations from \cref{rels4,rels7}---there is a sign change whenever a dot slides across the critical point of a cup or cap.

To establish additional useful relations, we adopt a generating function formalism which is a slight refinement of the setup introduced in \cite{HKM}. We denote the $r$th power of 
$\textdot$ under vertical composition
simply by labeling the dot with $r$. More generally,
given a polynomial $f(x) = \sum_{r\geq 0} c_r x^r \in \kk[x]$
and a dot in some string diagram $s$,
we denote
$$
\sum_{r \geq 0} c_r \times 
(\text{the morphism obtained from $s$ by labeling the dot by }r)
$$
by attaching what we call a {\em pin} to the dot, labeling 
the node at the head of the pin by $f(x)$:
\begin{equation}\label{labelledpin}
\begin{tikzpicture}[anchorbase]
	\draw[-] (0,-.4) to (0,.4);
    \pinX{(0,0)}{}{(.8,0)}{f(x)};
\end{tikzpicture}\;
:=
\sum_{r \geq 0} c_r 
\:\begin{tikzpicture}[anchorbase]
	\draw[-] (0,-.4) to (0,.4);
    \closeddot{0,0};
    \node at (.2,0) {$\scriptstyle{r}$};
\end{tikzpicture}
\in \End_{\cNB_t}(B).
\end{equation}
In the drawing of a pin, the arm and the head of the pin can be moved freely around larger diagrams so long as the point stays put---these are not part of the string calculus.
More generally, $f(x)$ here could be a polynomial
with coefficients in the algebra $\kk\lround u^{-1} \rround$
of formal Laurent series in an indeterminate $u^{-1}$;
then the string $s$ decorated with a pin labeled $f(x)$ defines a morphism in the base-changed monoidal category 
$\cNB_t\llbracket u^{-1}\rrbracket$. We think of this as being a generating function for a family of morphisms.
Pins labeled by the power series
\begin{equation}
(u+ax)^{-1} = u^{-1} - ax u^{-2}+ a^2 x^2 u^{-3} -a^3 x^3 u^{-4}+\cdots \in \kk[x]\llbracket u^{-1} \rrbracket
\end{equation}
for $a \in \{+,-\}$ appear so often that we denote them by a special shorthand,
putting the variable $u$ into the node at the head of the pin and in addition we label the arm of the pin by the sign $a$:
\begin{align}\label{ping}
\begin{tikzpicture}[anchorbase]
	\draw[-] (0,-.4) to (0,.4);
    \pinX{(0,0)}{-}{(.8,0)}{u};
\end{tikzpicture}\;
&:=
\begin{tikzpicture}[anchorbase]
	\draw[-] (0,-.4) to (0,.4);
    \closeddot{0,0};
    \pinX{(0,0)}{}{(1.3,0)}{(u-x)^{-1}};
\end{tikzpicture}\;
= 
u^{-1}
\: \begin{tikzpicture}[anchorbase]
	\draw[-] (0.08,-.4) to (0.08,.4);
\end{tikzpicture}
+
u^{-2}\: \begin{tikzpicture}[anchorbase]
	\draw[-] (0.08,-.4) to (0.08,.4);
    \closeddot{0.08,0};
\end{tikzpicture}
+
u^{-3}
\: \begin{tikzpicture}[anchorbase]
	\draw[-] (0.08,-.4) to (0.08,.4);
    \closeddot{0.08,.1};
    \closeddot{0.08,-.1};
\end{tikzpicture}
+
u^{-4}
\: \begin{tikzpicture}[anchorbase]
	\draw[-] (0.08,-.4) to (0.08,.4);
    \closeddot{0.08,.2};
    \closeddot{0.08,0};
    \closeddot{0.08,-.2};
\end{tikzpicture}
+\cdots \in \End_{\cNB_t}(B)\llbracket u^{-1}\rrbracket,\\\label{pong}
\begin{tikzpicture}[anchorbase]
	\draw[-] (0,-.4) to (0,.4);
    \pinX{(0,0)}{+}{(.8,0)}{u};
\end{tikzpicture}\;
&:=
\begin{tikzpicture}[anchorbase]
	\draw (0,-.4) to (0,.4);
    \closeddot{0,0};
    \pinX{(0,0)}{}{(1.3,0)}{(u+x)^{-1}};
\end{tikzpicture}\;
= 
u^{-1}\: \begin{tikzpicture}[anchorbase]
	\draw[-] (0.08,-.4) to (0.08,.4);
\end{tikzpicture}
-
u^{-2}\: \begin{tikzpicture}[anchorbase]
	\draw[-] (0.08,-.4) to (0.08,.4);
    \closeddot{0.08,0};
\end{tikzpicture}
+
u^{-3}
\: \begin{tikzpicture}[anchorbase]
	\draw[-] (0.08,-.4) to (0.08,.4);
    \closeddot{0.08,.1};
    \closeddot{0.08,-.1};
\end{tikzpicture}
-
u^{-4}\: \begin{tikzpicture}[anchorbase]
	\draw[-] (0.08,-.4) to (0.08,.4);
    \closeddot{0.08,.2};
    \closeddot{0.08,0};
    \closeddot{0.08,-.2};
\end{tikzpicture}
+\cdots \in \End_{\cNB_t}(B)\llbracket u^{-1}\rrbracket.
\end{align}
These shorthand symbols behave well under $\tR$ and $\tT$ \cref{R,T}:
\begin{align}\label{RTdots}
\tR\left(
\begin{tikzpicture}[anchorbase]
	\draw[-] (0.08,-.3) to (0.08,.4);
\pinX{(0.08,0.05)}{a}{(.88,0.05)}{u};
\end{tikzpicture}\right) &=
\begin{tikzpicture}[anchorbase]
	\draw[-] (0.08,-.3) to (0.08,.4);
\pinX{(0.08,0.05)}{-a}{(-.72,0.05)}{u};
\end{tikzpicture}\:,&
\tT\left(
\begin{tikzpicture}[anchorbase]
	\draw[-] (0.08,-.3) to (0.08,.4);
\pinX{(0.08,0.05)}{a}{(.88,0.05)}{u};
\end{tikzpicture}\right) &=
\begin{tikzpicture}[anchorbase]
	\draw[-] (0.08,-.3) to (0.08,.4);
\pinX{(0.08,0.05)}{a}{(.88,0.05)}{u};
\end{tikzpicture}\:.
\end{align}

\begin{lemma}
The following relations hold in $\cNB_t$:
\begin{align}
\label{rels9}
\begin{tikzpicture}[anchorbase,scale=1.1]
	\draw[-] (0.4,0) to[out=90, in=0] (0.1,0.5);
	\draw[-] (0.1,0.5) to[out = 180, in = 90] (-0.2,0);
\pinX{(0.4,.2)}{a}{(1.2,.2)}{u};
\end{tikzpicture}
&
=
\begin{tikzpicture}[anchorbase,scale=1.1]
	\draw[-] (0.4,0) to[out=90, in=0] (0.1,0.5);
	\draw[-] (0.1,0.5) to[out = 180, in = 90] (-0.2,0);
\pinX{(-0.2,.2)}{-a}{(-1,.2)}{u};
\end{tikzpicture}\:,\\\label{rels9b}
\begin{tikzpicture}[anchorbase,scale=1.1]
	\draw[-] (0.4,0) to[out=-90, in=0] (0.1,-0.5);
	\draw[-] (0.1,-0.5) to[out = 180, in = -90] (-0.2,0);
\pinX{(0.4,-.2)}{a}{(1.2,-.2)}{u};
\end{tikzpicture}
&=
\begin{tikzpicture}[anchorbase,scale=1.1]
	\draw[-] (0.4,0) to[out=-90, in=0] (0.1,-0.5);
	\draw[-] (0.1,-0.5) to[out = 180, in = -90] (-0.2,0);
\pinX{(-0.2,-.2)}{-a}{(-1,-.2)}{u};
\end{tikzpicture}\:,
\\
\begin{tikzpicture}[anchorbase,scale=1.4]
	\draw[-] (0.25,.3) to (-0.25,-.3);
	\draw[-] (0.25,-.3) to (-0.25,.3);
	\pinX{(.12,-.145)}{a}{(.7,-.145)}{u};
\end{tikzpicture}
-\begin{tikzpicture}[anchorbase,scale=1.4]
	\draw[-] (0.25,.3) to (-0.25,-.3);
	\draw[-] (0.25,-.3) to (-0.25,.3);
	\pinX{(-.12,.145)}{a}{(-.7,.145)}{u};
\end{tikzpicture}
&=
a\:\begin{tikzpicture}[anchorbase,scale=1.4]
 	\draw[-] (0,-.3) to (0,.3);
	\draw[-] (-.3,-.3) to (-0.3,.3);
	\pinX{(0,-.15)}{a}{(.4,-.15)}{u};
	\pinX{(-.3,.15)}{a}{(-.7,.15)}{u};
\end{tikzpicture}\:-\:a\:\begin{tikzpicture}[anchorbase,scale=1.4]
 	\draw[-] (-0.15,-.3) to[out=90,in=180] (0,-.05) to[out=0,in=90] (0.15,-.3);
 	\draw[-] (-0.15,.3) to[out=-90,in=180] (0,.05) to[out=0,in=-90] (0.15,.3);
	\pinX{(-.12,.15)}{a}{(-.55,.15)}{u};
	\pinX{(.12,-.15)}{a}{(.55,-.15)}{u};
	\end{tikzpicture}
\:,
\label{rels10}\\
\label{rels10b}
\begin{tikzpicture}[anchorbase,scale=1.4]
	\draw[-] (0.25,.3) to (-0.25,-.3);
	\draw[-] (0.25,-.3) to (-0.25,.3);
	\pinX{(.12,.145)}{a}{(.7,.145)}{u};
\end{tikzpicture}
-\begin{tikzpicture}[anchorbase,scale=1.4]
	\draw[-] (0.25,.3) to (-0.25,-.3);
	\draw[-] (0.25,-.3) to (-0.25,.3);
	\pinX{(-.12,-.145)}{a}{(-.7,-.145)}{u};
\end{tikzpicture}
&=
a\:\begin{tikzpicture}[anchorbase,scale=1.4]
 	\draw[-] (0,-.3) to (0,.3);
	\draw[-] (-.3,-.3) to (-0.3,.3);
	\pinX{(0,.15)}{a}{(.4,.15)}{u};
	\pinX{(-.3,-.15)}{a}{(-.7,-.15)}{u};
\end{tikzpicture}\:-\:a\:\begin{tikzpicture}[anchorbase,scale=1.4]
 	\draw[-] (-0.15,-.3) to[out=90,in=180] (0,-.05) to[out=0,in=90] (0.15,-.3);
 	\draw[-] (-0.15,.3) to[out=-90,in=180] (0,.05) to[out=0,in=-90] (0.15,.3);
	\pinX{(-.12,-.15)}{a}{(-.55,-.15)}{u};
	\pinX{(.12,.15)}{a}{(.55,.15)}{u};
	\end{tikzpicture}
\:.
\end{align}
\end{lemma}

\begin{proof}
It is clear from the last relation in \cref{rels4} that
$\begin{tikzpicture}[anchorbase,scale=.9]
	\draw[-] (0.4,0) to[out=90, in=0] (0.1,0.5);
	\draw[-] (0.1,0.5) to[out = 180, in = 90] (-0.2,0);
\pinX{(.4,.2)}{}{(1.2,.2)}{f(x)};
\end{tikzpicture}
=
\begin{tikzpicture}[anchorbase,scale=.9]
	\draw[-] (0.4,0) to[out=90, in=0] (0.1,0.5);
	\draw[-] (0.1,0.5) to[out = 180, in = 90] (-0.2,0);
\pinX{(-.2,.2)}{}{(-1.2,.2)}{f(-x)};
\end{tikzpicture}\:$ and similarly for cups. The relations 
\cref{rels9,rels9b} follow. 
To prove \cref{rels10}, it suffices to establish the equivalent relation obtained by vertically composing on top with 
$
\begin{tikzpicture}[anchorbase,scale=.8]
\draw[-] (-.2,.3) to (-.2,-.3);
\draw[-] (.2,.3) to (.2,-.3);
\pinX{(-.2,0)}{}{(-1.2,0)}{u+ax};
\end{tikzpicture}
$
and on the bottom with
$\begin{tikzpicture}[anchorbase,scale=.8]
\draw[-] (-.2,.3) to (-.2,-.3);
\draw[-] (.2,.3) to (.2,-.3);
\pinX{(.2,0)}{}{(1.4,0)}{u+ax};
\end{tikzpicture}\:
$;
this equivalent relation follows immediately from the first relation from \cref{rels4}. Finally, the relation \cref{rels10b} follows on
applying $\tT$. 
\end{proof}

The pin notation gives the generating function
\begin{equation}
\begin{tikzpicture}[anchorbase]
\draw (-.2,0) circle (.2);
\pinX{(0,0)}{-}{(.7,0)}{u};
\end{tikzpicture}
= 
\sum_{r \geq 0} 
\begin{tikzpicture}[anchorbase]
\draw (0,0) circle (.2);
\closeddot{0.2,0};
\node at (0.4,0) {$\scriptstyle r$};
\end{tikzpicture}\:
u^{-r-1}
\in \End_{\cNB_t}(\one)\llbracket u^{-1}\rrbracket
\end{equation}
for ``dotted bubbles".
This is often useful, but even more important 
is 
the renormalization
\begin{equation}\label{deltadef}
\O(u)
 = \sum_{r \geq 0} \O_r u^{-r} := (-1)^{t}\left(1_\one - 2u \;
\begin{tikzpicture}[anchorbase]
\draw (0,0) circle (.2);
\pinX{(.2,0)}{-}{(.8,0)}{u};
\end{tikzpicture}
\right)
\in 1_\one + u^{-1} \End_{\cNB_t}(\one)\llbracket u^{-1}\rrbracket.
\end{equation}
Its coefficients are given explicitly by
\begin{align}
\O_0 &:= 1_\one,&
\O_r &:= -2(-1)^{t} \:\begin{tikzpicture}[anchorbase]
\draw (0,0) circle (.2);
\closeddot{.2,0};
\node at (.5,0) {$\scriptstyle r$};
\end{tikzpicture}
\end{align}
for $r \geq 1$. Note also by \cref{RTdots,rels9} that $\O(u)$ is invariant under both of the symmetries $\tR$ and $\tT$.

The following derives some further relations involving these generating functions. Yet more can then be obtained by applying $\tR$ and $\tT$.

\begin{theorem}\label{nbrelations}
The following relations hold in $\cNB_t$:
\begin{align}
\label{rels11}
2u \:\begin{tikzpicture}[anchorbase]
	\draw[-] (0,-.3)  to (0,-.1) to [out=90,in=180] (.3,.3) to [out=0,in=90] (.5,.1) to[out=-90,in=0] (.3,-.1) to [out=180,in=-90] (0,.3) to (0,.5);
	\pinX{(.5,.1)}{-}{(1.1,.1)}{u};
\end{tikzpicture}&=
2u\:
\begin{tikzpicture}[anchorbase]
	\draw[-] (0,-.3) to (0,.5);
 \draw (1.2,0.1) circle (.2);
\pinX{(1.4,.1)}{-}{(2,.1)}{u};
\pinX{(0,.1)}{+}{(.6,.1)}{u};
\end{tikzpicture}
-\:\begin{tikzpicture}[anchorbase]
	\draw[-] (0,-.3) to (0,.5);
	\pinX{(0,.1)}{-}{(.6,.1)}{u};
\end{tikzpicture}
\:-\:
\begin{tikzpicture}[anchorbase]
	\draw[-] (0,-.3) to (0,.5);
	\pinX{(0,.1)}{+}{(.6,.1)}{u};
\end{tikzpicture}
\:,\\\label{rels11a}
\begin{tikzpicture}[anchorbase]
 \draw (-.2,0.1) circle (.2);
	\pinX{(0,.1)}{-}{(.6,.1)}{u};
\end{tikzpicture}
\:+\:
\begin{tikzpicture}[anchorbase]
 \draw (-.2,0.1) circle (.2);
	\pinX{(0,.1)}{+}{(.6,.1)}{u};
\end{tikzpicture}
&=
2u\: \begin{tikzpicture}[anchorbase]
 \draw (-.2,0.1) circle (.2);
	\pinX{(0,.1)}{+}{(.6,.1)}{u};
\end{tikzpicture}
\:\:\begin{tikzpicture}[anchorbase]
 \draw (-.2,0.1) circle (.2);
	\pinX{(0,.1)}{-}{(.6,.1)}{u};
\end{tikzpicture}
\:,\\
\label{rels12}
\O(u)\O(-u)&=1_\one,\\\label{rels13}
\O(u)\;
\begin{tikzpicture}[anchorbase]
	\draw[-] (0,-.3) to (0,.5);
\end{tikzpicture}&=
\begin{tikzpicture}[anchorbase]
	\draw[-] (0,-.3) to (0,.5);
	\pinX{(0,.1)}{}{(-1,.1)}{\left(\frac{u-x}{u+x}\right)^2};
\end{tikzpicture}\;
\O(u)\:.
\end{align}
\end{theorem}

\begin{proof}
To prove \cref{rels11}, we have that
$$
2u \:\begin{tikzpicture}[anchorbase]
	\draw[-] (0,-.3)  to (0,-.1) to [out=90,in=180] (.3,.3) to [out=0,in=90] (.5,.1) to[out=-90,in=0] (.3,-.1) to [out=180,in=-90] (0,.3) to (0,.5);
	\pinX{(.5,.1)}{-}{(1.1,.1)}{u};
\end{tikzpicture}
\stackrel{\cref{rels9b}}{=}
2u\: \begin{tikzpicture}[anchorbase,scale=1.2]
	\draw[-] (0,-.4)  to (0,-.1) to [out=90,in=180] (.3,.3) to [out=0,in=90] (.5,.1) to[out=-90,in=0] (.3,-.1) to [out=180,in=-90] (0,.3) to (0,.6);
\bentpinX{(.14,-.05)}{+}{(.7,-.3)}{u}{(0.2,-.4)};
\end{tikzpicture}
\!\!\stackrel{\cref{rels10}}{=}
2u\:\begin{tikzpicture}[anchorbase,scale=1.2]
	\draw[-] (0,-.4) to (0,.6);
 \pinX{(0,0.4)}{+}{(-0.6,.4)}{u};
  \draw (.8,0.1) circle (.2);
 \bentpinX{(.6,.1)}{+}{(.3,-.3)}{u}{(.3,-.1)};
\end{tikzpicture}
\:-
2u\: \begin{tikzpicture}[anchorbase,scale=1.47]
 	\draw[-] (-0.15,-.4) to[out=90,in=180] (0,-.05) to[out=0,in=180] (0.15,-.2) to [out=0,in = -90] (.3,0);
 	\draw[-] (-0.15,.4) to[out=-90,in=180] (0,.05) to[out=0,in=180] (0.15,.2) to [out=0,in=90] (.3,0);
\pinX{(-.13,.1)}{+}{(-.6,.1)}{u};
\bentpinX{(.09,-.14)}{+}{(.6,-.3)}{u}{(0.2,-.4)};
\end{tikzpicture}\!\!
=
2u\:
\begin{tikzpicture}[anchorbase]
	\draw[-] (0,-.3) to (0,.5);
 \draw (1.2,0.1) circle (.2);
\pinX{(1.4,.1)}{-}{(2,.1)}{u};
\pinX{(0,.1)}{+}{(.6,.1)}{u};
\end{tikzpicture}\:-\:
\begin{tikzpicture}[anchorbase]
	\draw[-] (0,-.3) to (0,.5);
 \draw (1.2,0.1) circle (.2);
\pinX{(0,.1)}{}{(1.2,.1)}{\frac{2u}{(u+x)(u-x)}};
\end{tikzpicture}\:.
$$
The desired relation \cref{rels11} follows easily from this using $\frac{2u}{(u+x)(u-x)} = \frac{1}{u-x}+\frac{1}{u+x}$. 

Closing the free strings on the left in \cref{rels11} gives a new relation whose left-hand side is zero. This gives \cref{rels11a} immediately. Then \cref{rels12} follows
by substituting the definition of $\O(u)$ and $\O(-u)$ into the product $\O(u)\O(-u)$, multiplying out the brackets, and simplifying using \cref{rels11a}.

Finally, for the ``bubble slide" relation \cref{rels13},
we have that
\begin{align*}
0 &=
\begin{tikzpicture}[anchorbase,scale=1.5]
\draw(-.3,-.4) to[out=90,in = -90] (0,0) to [in=-90,out=90] (-.3,.4);
\draw (0,0) circle (.2);
\pinX{(.2,0)}{-}{(.7,0)}{u};
\end{tikzpicture}
=
\begin{tikzpicture}[anchorbase,scale=1.5]
\draw(.3,-.4) to[out=90,in = -90] (0,0) to [in=-90,out=90] (.3,.4);
\draw (0,0) circle (.2);
\pinX{(.2,0)}{-}{(.7,0)}{u};
\end{tikzpicture}
=
\begin{tikzpicture}[anchorbase,scale=1.5]
\draw(0,.35) to [out=0,in=90] (.3,.2) to [out=-90,in=0] (.2,0.05) to [out=180,in=-90] (0.1,0.15)
to [out=90,in=-90] (.3,.4);
\draw(.3,-.4) to [out=90,in=0] (.2,-.05) to[out=180,in=0] (0,-.25) to [out=180,in=-90] (-.2,0.1) to [out=90,in=180] (0,.35);
\bentpinX{(.12,-.15)}{-}{(.7,-.3)}{u}{(.4,-.2)};
\pinX{(.3,.15)}{-}{(.8,.15)}{u};
\end{tikzpicture}
\:-\:\begin{tikzpicture}[anchorbase,scale=1.5]
\draw(.3,-.4) to[out=90,in = 0] (-.1,.2) to [out=180,in = 90] (-.3,0) to [out=-90,in=180] (-.2,-.2) to [out=0,in=-90] (0,0)
to [out=90,in=-90] (.3,.4);
\bentpinX{(0,0)}{-}{(.1,-.33)}{u}{(.15,-.1)};
\pinX{(.24,-.05)}{-}{(.74,-.05)}{u};
\end{tikzpicture}\:.
\end{align*}
This shows that
$$
\begin{tikzpicture}[anchorbase,scale=1.2]
	\draw[-] (0,-.5)  to (0,-.1) to [out=90,in=180] (.3,.3) to [out=0,in=90] (.5,.1) to[out=-90,in=0] (.3,-.1) to [out=180,in=-90] (0,.3) to (0,.5);
 \pinX{(.5,0.1)}{-}{(1,0.1)}{u};
 \pinX{(0,-0.3)}{+}{(.5,-.3)}{u};
\end{tikzpicture}=
\begin{tikzpicture}[anchorbase,scale=1.2]
	\draw[-] (0,-.5)  to (0,-.1) to [out=90,in=0] (-.3,.3) to [out=180,in=90] (-.5,.1) to[out=-90,in=180] (-.3,-.1) to [out=0,in=-90] (0,.3) to (0,.5);
  \pinX{(-.5,0.1)}{+}{(-1,0.1)}{u};
 \pinX{(0,-0.3)}{-}{(-.5,-.3)}{u};
\end{tikzpicture}\:.
$$
Let $f$ denote the morphism on the left-hand side of this equation. The equation shows that $f$ is fixed by the symmetry $\tR$. Expanding the curl using \cref{rels11} gives that
\begin{equation*}
2uf
=
2u\begin{tikzpicture}[anchorbase]
	\draw[-] (0,-.7) to (0,.5);
 \draw (1.1,-0.1) circle (.2);
\pinX{(1.3,-.1)}{-}{(2,-.1)}{u};
\pinX{(0,.2)}{+}{(.6,.2)}{u};
\pinX{(0,-.4)}{+}{(.6,-.4)}{u};
\end{tikzpicture}
\:-\:
\begin{tikzpicture}[anchorbase]
	\draw[-] (0,-.7) to (0,.5);
\pinX{(0,.2)}{-}{(.6,.2)}{u};
\pinX{(0,-.4)}{+}{(.6,-.4)}{u};
\end{tikzpicture}\:-\:
\begin{tikzpicture}[anchorbase]
	\draw[-] (0,-.7) to (0,.5);
\pinX{(0,.2)}{+}{(.6,.2)}{u};
\pinX{(0,-.4)}{+}{(.6,-.4)}{u};
\end{tikzpicture}
\:.
\end{equation*}
So, using the definition of $\O(u)$,
we have that
$$
-(-1)^{t}\left(2uf+
\begin{tikzpicture}[anchorbase]
	\draw[-] (0,-.7) to (0,.5);
\pinX{(0,.2)}{-}{(.6,.2)}{u};
\pinX{(0,-.4)}{+}{(.6,-.4)}{u};
\end{tikzpicture}
\right)
=
\begin{tikzpicture}[anchorbase]
	\draw[-] (0,-.7) to (0,.5);
\pinX{(0,.2)}{+}{(.6,.2)}{u};
\pinX{(0,-.4)}{+}{(.6,-.4)}{u};
\end{tikzpicture}
\times
(-1)^t
\left(1_\one-2u\:
\begin{tikzpicture}[anchorbase,scale=1]
\draw (0,0) circle (.2);
\pinX{(.2,0)}{-}{(.7,0)}{u};
\end{tikzpicture}
\right)
=
\begin{tikzpicture}[anchorbase]
	\draw[-] (0,-.3) to (0,.5);
\pinX{(0,.1)}{}{(-1.2,.1)}{(u+x)^{-2}};
\end{tikzpicture}\;
\O(u).
$$
Since $f=\tR(f)$, the expression on the left-hand side of this equation is fixed by $\tR$, hence, so is the expression on the right-hand side.
Since
$\tR(\O(u)) = \O(u)$, this implies 
that
$$
\begin{tikzpicture}[anchorbase]
	\draw[-] (0,-.3) to (0,.5);
\pinX{(0,.1)}{}{(-1.2,.1)}{(u+x)^{-2}};
\end{tikzpicture}\;
\O(u)=
\O(u)\;
\begin{tikzpicture}[anchorbase]
	\draw[-] (0,-.3) to (0,.5);
\pinX{(0,.1)}{}{(1.2,.1)}{(u-x)^{-2}};
\end{tikzpicture}\:.
$$
The relation
\cref{rels13} follows immediately from this.
\end{proof}

Let $\Lambda$ be the algebra of 
symmetric functions over the ground ring $\kk$. 
Adopting standard notations, this is freely generated
either by the elementary symmetric functions $e_r\:(r > 0)$
or by the complete symmetric functions $h_r\:(r > 0)$.
The two families of generators are related by the identity
\begin{equation}\label{grassmannian}
e(-u) h(u) = 1
\end{equation}
where 
\begin{align}\label{genfuncs}
e(u) &= \sum_{r \geq 0} e_r u^{-r},&
h(u) &= \sum_{r \geq 0} h_r u^{-r}
\end{align}
are the corresponding generating functions, and $e_0 = h_0 = 1$ by convention.
%Later on, we will need to view $\Lambda$ as a graded algebra, which we do by declaring that the generators $e_r$ and $h_r$ are of degree $2r$.

Following \cite[Ch.~III, Sec.~8]{Mac}, we define
$q(u) \in \Lambda \llbracket u^{-1}\rrbracket$ 
and elements $q_r(r \geq 0)$ of $\Lambda$ so that
\begin{equation}\label{qdef}
q(u) = \sum_{r\geq 0} q_r u^{-r} := e(u)h(u).
\end{equation}
By \cref{grassmannian},
we have that 
\begin{equation}\label{qgrassmannian}
q(u) q(-u) = 1
\end{equation}
Equivalently, $q_0 = 1$ and
\begin{align}\label{muscles}
q_{2r} &= (-1)^{r-1}{\textstyle\frac{1}{2}}q_r^2 +
\sum_{s=1}^{r-1}(-1)^{s-1}q_s q_{2r-s}
\end{align}
for $r \geq 1$; cf. \cite[(III.8.2$'$)]{Mac}.
The subalgebra of $\Lambda$ generated by all $q_r\:(r \geq 0)$ is denoted
$\Gamma$. As explained in \cite{Mac}, 
$\Gamma$ is freely generated by $q_1,q_3,q_5,\dots$
(and it has a distinguished basis given by the {\em Schur $Q$-functions} $Q_\lambda$ indexed by all strict partitions).
It follows that $\Gamma$ is generated by 
the elements $q_r\:(r \geq 0)$ subject just to the relations
\cref{qgrassmannian}.
Hence, the relation \cref{rels12} from \cref{nbrelations} implies the following:

\begin{corollary}\label{gammacor}
There is a unique algebra homomorphism
$\gamma_t:\Gamma \rightarrow \End_{\cNB_t}(\one)$ such that $q_r
\mapsto \O_r$ for all $r \geq 0$.
\end{corollary}

We will show in \cref{bubblealgebra} below that the homomorphism $\gamma_t$ just constructed is actually an {\em isomorphism}.

%=====================
% Section 3
%=====================

\section{Relations in the 2-category \texorpdfstring{$\fU(\sl_2)$}{}}\label{sec3}

Next we recall the definition of the 
2-category
$\fU(\sl_2)$ following the approach of Lauda  \cite{Lauda}.
Working still over the ground ring $\kk$, 
$\fU(\sl_2)$ is the strict 2-category
with object set $\Z$,
generating
1-morphisms
$E 1_\lambda:\lambda \rightarrow \lambda+2$ and
$1_\lambda F:\lambda+2 \rightarrow \lambda$, whose identity $2$-morphisms will be represented by
the oriented strings $\,\begin{tikzpicture}[anchorbase]\draw[-to,thin] (0,-.2) to (0,.2);\node at (.2,0) {$\catlabel{\lambda}$};\end{tikzpicture}\,$
 and
$\,\begin{tikzpicture}[anchorbase]\draw[to-,thin] (0,-.2) to (0,.2);\node at (-.2,0) {$\catlabel{\lambda}$};\end{tikzpicture}\;$,
and generating 2-morphisms
\begin{align}\label{KMgens}
\begin{tikzpicture}[anchorbase,scale=1.4]
	\draw[-to,thin] (0.08,-.15) to (0.08,.3);
      \opendot{0.08,0.05};
      \node at (0.3,0.05) {$\catlabel{\lambda}$};
\end{tikzpicture}
&:E 1_\lambda \Rightarrow E 1_\lambda,&
\begin{tikzpicture}[anchorbase,scale=1.4]
	\draw[to-,thin] (0.3,-0.1) to[out=90, in=0] (0.1,0.2);
	\draw[-,thin] (0.1,0.2) to[out = 180, in = 90] (-0.1,-0.1);
      \node at (0.5,0.05) {$\catlabel{\lambda}$};
\end{tikzpicture}
&:E F 1_\lambda \Rightarrow 1_\lambda,
&\begin{tikzpicture}[anchorbase,scale=1.4]
	\draw[to-,thin] (0.3,0.2) to[out=-90, in=0] (0.1,-0.1);
	\draw[-,thin] (0.1,-0.1) to[out = 180, in = -90] (-0.1,0.2);
      \node at (0.5,0.05) {$\catlabel{\lambda}$};
\end{tikzpicture}&:1_\lambda \Rightarrow F E 1_\lambda,\\
&\text{(degree $2$)}&
&\text{(degree $1-\lambda$)}&
&\text{(degree $\lambda+1$)}\notag\\
\label{KMgens2}
\begin{tikzpicture}[anchorbase,scale=1.4]
	\draw[-to,thin] (0.18,-.15) to (-0.18,.3);
	\draw[-to,thin] (-0.18,-.15) to (0.18,.3);
      \node at (0.3,0.1) {$\catlabel{\lambda}$};
\end{tikzpicture}
&:E E 1_\lambda \Rightarrow E E 1_\lambda,&
\begin{tikzpicture}[anchorbase,scale=1.4]
	\draw[-,thin] (0.3,-0.1) to[out=90, in=0] (0.1,0.2);
	\draw[-to,thin] (0.1,0.2) to[out = 180, in = 90] (-0.1,-0.1);
      \node at (0.5,0.05) {$\catlabel{\lambda}$};
\end{tikzpicture}
&:F E 1_\lambda \Rightarrow 1_\lambda,&
\begin{tikzpicture}[anchorbase,scale=1.4]
	\draw[-,thin] (0.3,0.2) to[out=-90, in=0] (0.1,-0.1);
	\draw[-to,thin] (0.1,-0.1) to[out = 180, in = -90] (-0.1,0.2);
      \node at (0.5,0.1) {$\catlabel{\lambda}$};
\end{tikzpicture}
&:1_\lambda \Rightarrow E F 1_\lambda,
\\\notag
&\text{(degree $-2$)}&
&\text{(degree $\lambda+1$)}&
&\text{(degree $1-\lambda$)}
\end{align}
for all $\lambda \in \Z$, subject to certain relations below. We refer to the dots in this setting as {\em open dots} to distinguish them from the (closed) dots in the nil-Brauer category in the previous section. 

\begin{remark}\label{KMgrading}
The 2-category $\fU(\sl_2)$ also admits
a grading making it into a strict graded 2-category. This is defined by declaring that the generating 2-morphisms are of the degrees specified in parentheses. This grading will usually be ignored, but it crops up again in \cref{grrem},
and it will be needed in the final section 
since we need there to pass to the completion of $\fU(\sl_2)$ with respect to the grading.
\end{remark}
To write down the relations, we denote the $r$th power of the dot under vertical composition simply by labeling it with the natural number $r$.
We introduce rightward and leftward crossings
by setting
\begin{align}\label{KMsideways}
\begin{tikzpicture}[anchorbase]
	\draw[to-,thin] (0.28,-.3) to (-0.28,.4);
	\draw[-to,thin] (-0.28,-.3) to (0.28,.4);
   \node at (0.5,0.05) {$\catlabel{\lambda}$};
\end{tikzpicture}
&:=
\begin{tikzpicture}[anchorbase]
	\draw[-to,thin] (0.3,-.5) to (-0.3,.5);
	\draw[-,thin] (-0.2,-.2) to (0.2,.3);
        \draw[-,thin] (0.2,.3) to[out=50,in=180] (0.5,.5);
        \draw[-to,thin] (0.5,.5) to[out=0,in=90] (0.9,-.5);
        \draw[-,thin] (-0.2,-.2) to[out=230,in=0] (-0.6,-.5);
        \draw[-,thin] (-0.6,-.5) to[out=180,in=-90] (-0.9,.5);
   \node at (1.1,0) {$\catlabel{\lambda}$};
\end{tikzpicture}\:,&
\begin{tikzpicture}[anchorbase]
	\draw[-to,thin] (0.28,-.3) to (-0.28,.4);
	\draw[to-,thin] (-0.28,-.3) to (0.28,.4);
   \node at (0.5,0.05) {$\catlabel{\lambda}$};
\end{tikzpicture}
&:=
\begin{tikzpicture}[anchorbase]
	\draw[to-,thin] (0.3,.5) to (-0.3,-.5);
	\draw[-,thin] (-0.2,.2) to (0.2,-.3);
        \draw[-,thin] (0.2,-.3) to[out=130,in=180] (0.5,-.5);
        \draw[-,thin] (0.5,-.5) to[out=0,in=270] (0.9,.5);
        \draw[-,thin] (-0.2,.2) to[out=130,in=0] (-0.6,.5);
        \draw[-to,thin] (-0.6,.5) to[out=180,in=-270] (-0.9,-.5);
   \node at (1.1,0) {$\catlabel{\lambda}$};
\end{tikzpicture}\:,
\\&\text{(degree $0$)}&&\text{(degree $0$)}\notag
\end{align}
and use the following shorthands:
\begin{align}
\begin{tikzpicture}[anchorbase]
  \draw[to-,thin] (0,0.4) to[out=180,in=90] (-.2,0.2);
  \draw[-,thin] (0.2,0.2) to[out=90,in=0] (0,.4);
 \draw[-,thin] (-.2,0.2) to[out=-90,in=180] (0,0);
  \draw[-,thin] (0,0) to[out=0,in=-90] (0.2,0.2);
      \node at (0,0.2) {$\kmlabel{n}$};
   \node at (.38,0.2) {$\catlabel{\lambda}$};
\end{tikzpicture}
&:=
\left\{\begin{array}{ll}
(-1)^n\det\left(
\:\begin{tikzpicture}[anchorbase]
  \draw[-to,thin] (0.2,0.2) to[out=90,in=0] (0,.4);
  \draw[-,thin] (0,0.4) to[out=180,in=90] (-.2,.2);
\draw[-,thin] (-.2,0.2) to[out=-90,in=180] (0,0);
  \draw[-,thin] (0,0) to[out=0,in=-90] (0.2,0.2);
   \opendot{.2,.2};
     \node at (.75,0.2) {$\kmlabel{r-s-\lambda}$};
   \node at (-.38,0.2) {$\catlabel{\lambda}$};
\end{tikzpicture}
\right)_{r,s=1,\dots,n}&\text{if $0 \leq n \leq -\lambda$}\\
0&\text{if $n < 0$ or $n > -\lambda$},
\end{array}\right.
\label{KMbubbles2}
\\\notag &\text{(degree $2n$)}\\
\label{KMbubbles1}
\begin{tikzpicture}[anchorbase]
  \draw[-to,thin] (0.2,0.2) to[out=90,in=0] (0,.4);
  \draw[-,thin] (0,0.4) to[out=180,in=90] (-.2,0.2);
\draw[-,thin] (-.2,0.2) to[out=-90,in=180] (0,0);
  \draw[-,thin] (0,0) to[out=0,in=-90] (0.2,0.2);
     \node at (0,0.2) {$\kmlabel{n}$};
   \node at (-.38,0.2) {$\catlabel{\lambda}$};
\end{tikzpicture}
&:=\left\{
\begin{array}{ll}
(-1)^n \det\left(
\:
\begin{tikzpicture}[anchorbase]
  \draw[to-,thin] (0,0.4) to[out=180,in=90] (-.2,0.2);
  \draw[-,thin] (0.2,0.2) to[out=90,in=0] (0,.4);
 \draw[-,thin] (-.2,0.2) to[out=-90,in=180] (0,0);
  \draw[-,thin] (0,0) to[out=0,in=-90] (0.2,0.2);
   \opendot{-.2,.2};
     \node at (-.75,0.2) {$\kmlabel{r-s+\lambda}$};
   \node at (.38,0.2) {$\catlabel{\lambda}$};
\end{tikzpicture}\:
\right)_{r,s=1,\dots,n}
&\text{if $0 \leq n \leq \lambda$}\\
0&\text{if $n < 0$ or $n > \lambda$.}
\end{array}\right.\\
\notag &\text{(degree $2n$)}
\end{align}
Then the defining relations are as follows:
\begin{align}\label{KMrels1}
\begin{tikzpicture}[anchorbase,scale=.8]
	\draw[-to,thin] (0.28,.4) to[out=90,in=-90] (-0.28,1);
	\draw[-to,thin] (-0.28,.4) to[out=90,in=-90] (0.28,1);
	\draw[-,thin] (0.28,-.2) to[out=90,in=-90] (-0.28,.4);
	\draw[-,thin] (-0.28,-.2) to[out=90,in=-90] (0.28,.4);
   \node at (.43,.4) {$\catlabel{\lambda}$};
\end{tikzpicture}
&=
0,&
\begin{tikzpicture}[anchorbase,scale=.8]
	\draw[to-,thin] (0.45,.8) to (-0.45,-.4);
	\draw[-to,thin] (0.45,-.4) to (-0.45,.8);
        \draw[-,thin] (0,-.4) to[out=90,in=-90] (-.45,0.2);
        \draw[-to,thin] (-0.45,0.2) to[out=90,in=-90] (0,0.8);
   \node at (.5,.2) {$\catlabel{\lambda}$};
\end{tikzpicture}
&=
\begin{tikzpicture}[anchorbase,scale=.8]
	\draw[to-,thin] (0.45,.8) to (-0.45,-.4);
	\draw[-to,thin] (0.45,-.4) to (-0.45,.8);
        \draw[-,thin] (0,-.4) to[out=90,in=-90] (.45,0.2);
        \draw[-to,thin] (0.45,0.2) to[out=90,in=-90] (0,0.8);
   \node at (.7,.2) {$\catlabel{\lambda}$};
\end{tikzpicture}\:,
&
\begin{tikzpicture}[anchorbase,scale=1.2]
 	\draw[to-,thin] (0.25,.55) to (-0.25,-.15);
	\draw[-to,thin] (0.25,-.15) to (-0.25,.55);
  \node at (.3,.25) {$\catlabel{\lambda}$};
      \opendot{-0.13,0.38};
      \end{tikzpicture}
-
\begin{tikzpicture}[anchorbase,scale=1.2]
	\draw[to-,thin] (0.25,.55) to (-0.25,-.15);
	\draw[-to,thin] (0.25,-.15) to (-0.25,.55);
  \node at (.3,.25) {$\catlabel{\lambda}$};
      \opendot{0.13,0.01};
      \end{tikzpicture}
&=
\begin{tikzpicture}[anchorbase,scale=1.2]
	\draw[to-,thin] (0.25,.55) to (-0.25,-.15);
	\draw[-to,thin] (0.25,-.15) to (-0.25,.55);
  \node at (.3,.25) {$\catlabel{\lambda}$};
      \opendot{-0.13,0.01};
\end{tikzpicture}
-\begin{tikzpicture}[anchorbase,scale=1.2]
	\draw[to-,thin] (0.25,.55) to (-0.25,-.15);
	\draw[-to,thin] (0.25,-.15) to (-0.25,.55);
  \node at (.3,.25) {$\catlabel{\lambda}$};
      \opendot{0.13,0.38};
\end{tikzpicture}=
\begin{tikzpicture}[anchorbase,scale=1.2]
 	\draw[-to,thin] (0.08,-.3) to (0.08,.4);
	\draw[-to,thin] (-0.28,-.3) to (-0.28,.4);
 \node at (.28,.06) {$\catlabel{\lambda}$};
\end{tikzpicture}\:,\\
\label{KMrels2}
&&
\begin{tikzpicture}[anchorbase]
  \draw[-to,thin] (0.3,0) to (0.3,.4);
	\draw[-,thin] (0.3,0) to[out=-90, in=0] (0.1,-0.4);
	\draw[-,thin] (0.1,-0.4) to[out = 180, in = -90] (-0.1,0);
	\draw[-,thin] (-0.1,0) to[out=90, in=0] (-0.3,0.4);
	\draw[-,thin] (-0.3,0.4) to[out = 180, in =90] (-0.5,0);
  \draw[-,thin] (-0.5,0) to (-0.5,-.4);
   \node at (0.5,0) {$\catlabel{\lambda}$};
\end{tikzpicture}
&=
\begin{tikzpicture}[anchorbase]
  \draw[-to,thin] (0,-0.4) to (0,.4);
   \node at (0.2,0) {$\catlabel{\lambda}$};
\end{tikzpicture}\:,&
\begin{tikzpicture}[anchorbase]
  \draw[-to,thin] (0.3,0) to (0.3,-.4);
	\draw[-,thin] (0.3,0) to[out=90, in=0] (0.1,0.4);
	\draw[-,thin] (0.1,0.4) to[out = 180, in = 90] (-0.1,0);
	\draw[-,thin] (-0.1,0) to[out=-90, in=0] (-0.3,-0.4);
	\draw[-,thin] (-0.3,-0.4) to[out = 180, in =-90] (-0.5,0);
  \draw[-,thin] (-0.5,0) to (-0.5,.4);
   \node at (0.5,0) {$\catlabel{\lambda}$};
\end{tikzpicture}
&=
\begin{tikzpicture}[anchorbase]
  \draw[to-,thin] (0,-0.4) to (0,.4);
   \node at (0.2,0) {$\catlabel{\lambda}$};
\end{tikzpicture}\:,
\end{align}
\begin{align}\label{KMrels3}
\begin{tikzpicture}[anchorbase,scale=.95]
	\draw[-,thin] (0.28,0) to[out=90,in=-90] (-0.28,.7);
	\draw[-to,thin] (-0.28,0) to[out=90,in=-90] (0.28,.7);
	\draw[-,thin] (0.28,-.7) to[out=90,in=-90] (-0.28,0);
	\draw[to-,thin] (-0.28,-.7) to[out=90,in=-90] (0.28,0);
  \node at (.5,0) {$\catlabel{\lambda}$};
\end{tikzpicture}
&=
-
\begin{tikzpicture}[anchorbase,scale=.95]
	\draw[-to,thin] (0.2,-.7) to (0.2,.7);
	\draw[to-,thin] (-0.3,-.7) to (-0.3,.7);
   \node at (.35,0) {$\catlabel{\lambda}$};
\end{tikzpicture}
+\sum_{n=0}^{-\lambda-1}\!\!\!
\sum_{\substack{r,s\geq 0\\r+s=-\lambda-1-n}}\!\!
\begin{tikzpicture}[anchorbase,scale=.95]
	\draw[-,thin] (0.3,-0.7) to[out=90, in=0] (0,-0.1);
	\draw[-to,thin] (0,-0.1) to[out = 180, in = 90] (-0.3,-.7);
   \opendot{0.28,-0.5};
      \node at (.5,-.5) {$\kmlabel{s}$};
      \clockwisebubble{(.53,0)};
      \node at (.53,0) {$\kmlabel{n}$};
   \node at (-0.45,0) {$\catlabel{\lambda}$};
	\draw[to-,thin] (0.3,.7) to[out=-90, in=0] (0,0.1);
	\draw[-,thin] (0,0.1) to[out = -180, in = -90] (-0.3,.7);
   \opendot{.28,0.5};
      \node at (0.5,0.5) {$\kmlabel{r}$};
\end{tikzpicture}\:,
&\begin{tikzpicture}[anchorbase,scale=.95]
	\draw[-to,thin] (0.28,0) to[out=90,in=-90] (-0.28,.7);
	\draw[-,thin] (-0.28,0) to[out=90,in=-90] (0.28,.7);
	\draw[to-,thin] (0.28,-.7) to[out=90,in=-90] (-0.28,0);
	\draw[-,thin] (-0.28,-.7) to[out=90,in=-90] (0.28,0);
  \node at (.5,0) {$\catlabel{\lambda}$};
\end{tikzpicture}
&=
-
\begin{tikzpicture}[anchorbase,scale=.95]
	\draw[to-,thin] (0.2,-.7) to (0.2,.7);
	\draw[-to,thin] (-0.3,-.7) to (-0.3,.7);
   \node at (.35,.05) {$\catlabel{\lambda}$};
\end{tikzpicture}
+
\sum_{n=0}^{\lambda-1}\!\!\!
\sum_{\substack{r,s\geq 0\\r+s=\lambda-1-n}}\!\!
\begin{tikzpicture}[anchorbase,scale=.95]
	\draw[-,thin] (0.3,0.7) to[out=-90, in=0] (0,0.1);
	\draw[-to,thin] (0,0.1) to[out = 180, in = -90] (-0.3,0.7);
    \node at (0.4,0) {$\catlabel{\lambda}$};
    \anticlockwisebubble{(-.7,0)};
      \node at (-.7,0) {$\kmlabel{n}$};
   \opendot{-0.28,0.45};
      \node at (-0.5,0.45) {$\kmlabel{r}$};
	\draw[to-,thin] (0.3,-.7) to[out=90, in=0] (0,-0.1);
	\draw[-,thin] (0,-0.1) to[out = 180, in = 90] (-0.3,-.7);
   \opendot{-0.28,-0.5};
   \node at (-.5,-.5) {$\kmlabel{s}$};
\end{tikzpicture}\:,
\end{align}
\begin{align}\label{KMrels4}
\begin{tikzpicture}[anchorbase,scale=.8]
	\draw[to-,thin] (0,0.6) to (0,0.3);
	\draw[-,thin] (0,0.3) to [out=-90,in=180] (.3,-0.2);
	\draw[-,thin] (0.3,-0.2) to [out=0,in=-90](.5,0);
	\draw[-,thin] (0.5,0) to [out=90,in=0](.3,0.2);
	\draw[-,thin] (0.3,.2) to [out=180,in=90](0,-0.3);
	\draw[-,thin] (0,-0.3) to (0,-0.6);
   \node at (0.9,0) {$\catlabel{\lambda}$};
\end{tikzpicture}
&=
-\delta_{\lambda,0}
\begin{tikzpicture}[anchorbase,scale=.8]
	\draw[to-,thin] (0,0.6) to (0,-0.6);
     \node at (0.2,0) {$\catlabel{\lambda}$};
\end{tikzpicture}\:
\text{ if $\lambda \geq 0$,}&
\begin{tikzpicture}[anchorbase,scale=.8]
	\draw[to-,thin] (0,0.6) to (0,0.3);
	\draw[-,thin] (0,0.3) to [out=-90,in=0] (-.3,-0.2);
	\draw[-,thin] (-0.3,-0.2) to [out=180,in=-90](-.5,0);
	\draw[-,thin] (-0.5,0) to [out=90,in=180](-.3,0.2);
	\draw[-,thin] (-0.3,.2) to [out=0,in=90](0,-0.3);
	\draw[-,thin] (0,-0.3) to (0,-0.6);
     \node at (-0.9,0) {$\catlabel{\lambda}$};
\end{tikzpicture}
&=
\delta_{\lambda,0}
\:
\begin{tikzpicture}[anchorbase,scale=.8]
	\draw[to-,thin] (0,0.6) to (0,-0.6);
   \node at (-0.2,0) {$\catlabel{\lambda}$};
\end{tikzpicture}\:
\text{ if $\lambda \leq 0$},
\end{align}
\begin{align}\label{KMrels5}
\begin{tikzpicture}[anchorbase]
  \draw[to-,thin] (0,0.4) to[out=180,in=90] (-.2,0.2);
  \draw[-,thin] (0.2,0.2) to[out=90,in=0] (0,.4);
 \draw[-,thin] (-.2,0.2) to[out=-90,in=180] (0,0);
  \draw[-,thin] (0,0) to[out=0,in=-90] (0.2,0.2);
     \node at (0.36,0.2) {$\catlabel{\lambda}$};
\opendot{-0.2,.2};
\node at (-.4,0.2) {$\kmlabel{n}$};
\end{tikzpicture}
&=
\delta_{n,\lambda-1}
 \:1_{1_\lambda}\:
\text{ for all $0 \leq n <\lambda$,}&
\begin{tikzpicture}[anchorbase]
  \draw[-to,thin] (0.2,0.2) to[out=90,in=0] (0,.4);
  \draw[-,thin] (0,0.4) to[out=180,in=90] (-.2,0.2);
\draw[-,thin] (-.2,0.2) to[out=-90,in=180] (0,0);
  \draw[-,thin] (0,0) to[out=0,in=-90] (0.2,0.2);
\opendot{0.2,.2};
\node at (.4,0.2) {$\kmlabel{n}$};
       \node at (-0.4,0.2) {$\catlabel{\lambda}$};
\end{tikzpicture}
&=
\delta_{n,-\lambda-1}
\: 1_{1_\lambda} 
\text{ for all $0 \leq n < -\lambda$.}
\end{align}

A more efficient presentation for $\fU(\sl_2)$ was given by Rouquier in \cite{Rou}.
To formulate this,
we just need the generating morphisms
$\textupdot{\color{catcolor}\scriptstyle\lambda}$,
$\textupcrossing{\scriptstyle\color{catcolor}\lambda}$,
$\textrightcup{\scriptstyle\color{catcolor}\lambda}$ and
$\textrightcap{\scriptstyle\color{catcolor}\lambda}$ 
(hence, we also have the rightward crossings defined via \cref{KMsideways}),
subject to the relations \cref{KMrels1}, the zigzag identities from \cref{KMrels2}, and the new
{\em inversion relation} which asserts that the following matrices are invertible in the additive envelope of $\fU(\sl_2)$:
\begin{align}\label{inversionrel}
\left[\:\,
\begin{tikzpicture}[anchorbase]
	\draw[to-,thin] (0.28,-.3) to (-0.28,.4);
	\draw[-to,thin] (-0.28,-.3) to (0.28,.4);
   \node at (.4,.05) {$\catlabel{\lambda}$};
\end{tikzpicture}\:
\begin{tikzpicture}[anchorbase]
	\draw[to-,thin] (0.4,0.2) to[out=-90, in=0] (0.1,-.2);
	\draw[-,thin] (0.1,-.2) to[out = 180, in = -90] (-0.2,0.2);
  \node at (0.3,-0.25) {$\catlabel{\lambda}$};
      \node at (.85,0) {$\kmlabel{-\lambda-1}$};
      \opendot{0.38,0};
\end{tikzpicture}
\:\cdots
\:\begin{tikzpicture}[anchorbase]
	\draw[to-,thin] (0.4,0.2) to[out=-90, in=0] (0.1,-.2);
	\draw[-,thin] (0.1,-.2) to[out = 180, in = -90] (-0.2,0.2);
      \opendot{0.38,0};
  \node at (0.3,-0.25) {$\catlabel{\lambda}$};
\end{tikzpicture}
\:\:
\begin{tikzpicture}[anchorbase]
	\draw[to-,thin] (0.4,0.2) to[out=-90, in=0] (0.1,-.2);
	\draw[-,thin] (0.1,-.2) to[out = 180, in = -90] (-0.2,0.2);
  \node at (0.3,-0.25) {$\catlabel{\lambda}$};
\end{tikzpicture}
\:\right]
&
:E F 1_\lambda \oplus 
1_\lambda^{\oplus -\lambda}
\Rightarrow
 F E 1_\lambda&&\text{if $\lambda \leq
  0$,}\\
\left[\begin{array}{r}
\begin{tikzpicture}[anchorbase]
	\draw[to-,thin] (0.28,-.3) to (-0.28,.4);
	\draw[-to,thin] (-0.28,-.3) to (0.28,.4);
   \node at (.4,.05) {$\catlabel{\lambda}$};
\end{tikzpicture}
\\
\!\!\!\!\begin{tikzpicture}[anchorbase]
	\draw[to-,thin] (0.4,0) to[out=90, in=0] (0.1,0.4);
	\draw[-,thin] (0.1,0.4) to[out = 180, in = 90] (-0.2,0);
     \node at (0.3,0.5) {$\catlabel{\lambda}$};
      \node at (-0.5,0.2) {$\kmlabel{\lambda-1}$};
      \opendot{-0.15,0.2};
      \end{tikzpicture}
\\\vdots\:\:\:\:\\
\begin{tikzpicture}[anchorbase]
	\draw[to-,thin] (0.4,0) to[out=90, in=0] (0.1,0.4);
	\draw[-,thin] (0.1,0.4) to[out = 180, in = 90] (-0.2,0);
      \opendot{-0.15,0.2};
  \node at (0.3,0.5) {$\catlabel{\lambda}$};
\end{tikzpicture}
\\\begin{tikzpicture}[anchorbase]
	\draw[to-,thin] (0.4,0) to[out=90, in=0] (0.1,0.4);
	\draw[-,thin] (0.1,0.4) to[out = 180, in = 90] (-0.2,0);
  \node at (0.3,0.5) {$\catlabel{\lambda}$};
\end{tikzpicture}
\end{array}
\right]
&:
E F 1_\lambda \Rightarrow
F E 1_\lambda \oplus 1_\lambda^{\oplus \lambda}
&&\text{if $\lambda \geq 0$}.\label{day3}
\end{align}
The equivalence of Rouquier's presentation with Lauda's one from the previous paragraph was established in \cite{Brundan}.
It is shown there that
the two-sided 
inverse of the matrix \cref{inversionrel} is the $(1-\lambda)\times 1$ matrix
with first entry $-\begin{tikzpicture}[anchorbase]
	\draw[-to,thin] (0.28,-.3) to (-0.28,.4);
	\draw[to-,thin] (-0.28,-.3) to (0.28,.4);
   \node at (.4,.05) {$\catlabel{\lambda}$};
\end{tikzpicture}
$ and $(n+2)$th entry equal to
$\:
\begin{tikzpicture}[anchorbase]
	\draw[-,thin] (0.4,0) to[out=90, in=0] (0.1,0.5);
	\draw[-to,thin] (0.1,0.5) to[out = 180, in = 90] (-0.2,0);
     \node at (-0.4,0.2) {$\catlabel{\lambda}$};
      \opendot{.36,.2};      \node at (.55,.2) {$\kmlabel{n}$};
\end{tikzpicture}
      +
\begin{tikzpicture}[anchorbase]
	\draw[-,thin] (0.4,0) to[out=90, in=0] (0.1,0.5);
	\draw[-to,thin] (0.1,0.5) to[out = 180, in = 90] (-0.2,0);
     \node at (-0.4,0.2) {$\catlabel{\lambda}$};
\clockwisebubble{(1.3,.2)};\node at (1.3,.2) {$\kmlabel{1}$};
      \opendot{.36,.2};      \node at (.75,.2) {$\kmlabel{n-1}$};
      \end{tikzpicture}
+      \cdots
      +
\begin{tikzpicture}[anchorbase]
	\draw[-,thin] (0.4,0) to[out=90, in=0] (0.1,0.5);
	\draw[-to,thin] (0.1,0.5) to[out = 180, in = 90] (-0.2,0);
     \node at (-0.4,0.2) {$\catlabel{\lambda}$};
\clockwisebubble{(.8,.2)};\node at (.8,.2) {$\kmlabel{n}$};
      \end{tikzpicture}\:
$
for $0 \leq n \leq -\lambda-1$, and
the two-sided inverse of the matrix \cref{day3} is the $1 \times (\lambda+1)$ matrix
with first entry
$-\begin{tikzpicture}[anchorbase]
	\draw[-to,thin] (0.28,-.3) to (-0.28,.4);
	\draw[to-,thin] (-0.28,-.3) to (0.28,.4);
   \node at (.4,.05) {$\catlabel{\lambda}$};
\end{tikzpicture}$
and $(n+2)$th entry equal to
$\:
\begin{tikzpicture}[anchorbase]
	\draw[-,thin] (0.4,0.3) to[out=-90, in=0] (0.1,-.2);
	\draw[-to,thin] (0.1,-.2) to[out = 180, in = -90] (-0.2,0.3);
     \node at (0.6,0.1) {$\catlabel{\lambda}$};
      \opendot{-.17,.05};      \node at (-.4,.05) {$\kmlabel{n}$};
\end{tikzpicture}
      +
\begin{tikzpicture}[anchorbase]
	\draw[-,thin] (0.4,0.3) to[out=-90, in=0] (0.1,-.2);
	\draw[-to,thin] (0.1,-.2) to[out = 180, in = -90] (-0.2,0.3);
     \node at (0.6,0.1) {$\catlabel{\lambda}$};
\anticlockwisebubble{(-1.2,0.05)};\node at (-1.2,0.05) {$\kmlabel{1}$};
      \opendot{-.17,.05};      \node at (-.55,.05) {$\kmlabel{n-1}$};
      \end{tikzpicture}
+      \cdots
      +
\begin{tikzpicture}[anchorbase]
	\draw[-,thin] (0.4,0.3) to[out=-90, in=0] (0.1,-.2);
	\draw[-to,thin] (0.1,-.2) to[out = 180, in = -90] (-0.2,0.3);
     \node at (0.6,0.1) {$\catlabel{\lambda}$};
\anticlockwisebubble{(-.7,0.05)};\node at (-.7,0.05) {$\kmlabel{n}$};
      \end{tikzpicture}\:$
for $0 \leq n \leq \lambda-1$.

We define the downward open dot ${\scriptstyle\color{catcolor}\lambda}\textdowndot$
 and the downward 
crossing ${\scriptstyle\color{catcolor}\lambda}\textdowncrossing$ to be the right mates
of $\textupdot{\scriptstyle\color{catcolor}\lambda}$ and $\textupcrossing{\scriptstyle\color{catcolor}\lambda}$, respectively.
The defining relations for $\fU(\sl_2)$ imply that
\begin{align}\label{otheradjunction}
\begin{tikzpicture}[anchorbase]
  \draw[-,thin] (0.3,0) to (0.3,-.4);
	\draw[-,thin] (0.3,0) to[out=90, in=0] (0.1,0.4);
	\draw[-,thin] (0.1,0.4) to[out = 180, in = 90] (-0.1,0);
	\draw[-,thin] (-0.1,0) to[out=-90, in=0] (-0.3,-0.4);
	\draw[-,thin] (-0.3,-0.4) to[out = 180, in =-90] (-0.5,0);
  \draw[-to,thin] (-0.5,0) to (-0.5,.4);
     \node at (0.5,0) {$\catlabel{\lambda}$};
\end{tikzpicture}
&=
\begin{tikzpicture}[anchorbase]
  \draw[-to,thin] (0,-0.4) to (0,.4);
   \node at (0.2,0) {$\catlabel{\lambda}$};
\end{tikzpicture},
&
\begin{tikzpicture}[anchorbase]
  \draw[-,thin] (0.3,0) to (0.3,.4);
	\draw[-,thin] (0.3,0) to[out=-90, in=0] (0.1,-0.4);
	\draw[-,thin] (0.1,-0.4) to[out = 180, in = -90] (-0.1,0);
	\draw[-,thin] (-0.1,0) to[out=90, in=0] (-0.3,0.4);
	\draw[-,thin] (-0.3,0.4) to[out = 180, in =90] (-0.5,0);
  \draw[-to,thin] (-0.5,0) to (-0.5,-.4);
   \node at (0.5,0) {$\catlabel{\lambda}$};
\end{tikzpicture}
=
\begin{tikzpicture}[anchorbase]
  \draw[to-,thin] (0,-0.4) to (0,.4);
   \node at (0.2,0) {$\catlabel{\lambda}$};
\end{tikzpicture}\:.
\end{align}
Moreover, ${\scriptstyle\color{catcolor}\lambda}\textdowndot$ and ${\scriptstyle\color{catcolor}\lambda}\textdowncrossing$
are equal to the {\em left} mates
of $\textupdot{\scriptstyle\color{catcolor}\lambda}$ and $\textupcrossing{\scriptstyle\color{catcolor}\lambda}$.
It follows that diagrams for 2-morphisms in $\fU(\sl_2)$
are invariant under boundary-preserving planar isotopy.
In particular, $\fU(\sl_2)$ is strictly pivotal with duality functor $\tD$ 
defined by rotating diagrams through $180^\circ$.
There are a couple more useful symmetries, i.e.,
strict 2-functors
\begin{align}\label{TR}
\tR:\fU(\sl_2) &\rightarrow\fU(\sl_2)^{\rev},&
\tT:\fU(\sl_2) &\rightarrow \fU(\sl_2)^{\op}
\end{align}
defined as follows:
\begin{itemize}
\item
$\tR$ takes the object $\lambda$ to $-\lambda$, switches $E 1_\lambda$ with $1_{-\lambda} E$
and $F 1_\lambda$ with $1_{-\lambda} F$, and 
takes string diagram $s$ to $(-1)^{\times(s)} s^{\leftrightarrow}$,
also negating 
weights $\lambda$ labeling regions;
\item $\tT$ takes $\lambda$ to $-\lambda$,
switches the generating 1-morphisms
$E 1_\lambda$ and $F 1_{-\lambda}$,
and takes string diagram $s$ representing a 2-morphism
to $(-1)^{\times(s)} s^\updownarrow$, also negating 
weights $\lambda$ labeling regions.
\end{itemize}
Here, $\times(s)$ is the total number of crossings in $s$.
The duality functor factorizes as  $\tD=\tR\circ \tT=\tT\circ\tR$.

We use similar conventions for 
pins attached to open dots 
as
in \cref{labelledpin} for pins attached to closed dots in the previous section. 
In particular, as in \cref{ping,pong}, we have
\begin{align}\label{pinminus}
\begin{tikzpicture}[anchorbase]
	\draw[-to,thin] (0,-.4) to (0,.4);
    \pinO{(0,0)}{-}{(.8,0)}{u};
   \node at (1.15,0) {$\catlabel{\lambda}$};
\end{tikzpicture}\;
&:=
\begin{tikzpicture}[anchorbase]
	\draw[-to,thin] (0,-.4) to (0,.4);
    \opendot{0,0};
    \pinO{(0,0)}{}{(1.3,0)}{(u-x)^{-1}};
   \node at (2.1,0) {$\catlabel{\lambda}$};
\end{tikzpicture}\;
= 
u^{-1}
\: \begin{tikzpicture}[anchorbase]
	\draw[-to,thin] (0.08,-.4) to (0.08,.4);
   \node at (0.25,0) {$\catlabel{\lambda}$};
\end{tikzpicture}
+
u^{-2}\: \begin{tikzpicture}[anchorbase]
	\draw[-to,thin] (0.08,-.4) to (0.08,.4);
    \opendot{0.08,.05};
   \node at (0.3,0) {$\catlabel{\lambda}$};
\end{tikzpicture}
+
u^{-3}
\: \begin{tikzpicture}[anchorbase]
	\draw[-to,thin] (0.08,-.4) to (0.08,.4);
    \opendot{0.08,.15};
    \opendot{0.08,-.05};
   \node at (0.3,0) {$\catlabel{\lambda}$};
\end{tikzpicture}
+
+\cdots,\\\label{pinplus}
\begin{tikzpicture}[anchorbase]
	\draw[-to,thin] (0,-.4) to (0,.4);
    \pinO{(0,0)}{+}{(.8,0)}{u};
   \node at (1.15,0) {$\catlabel{\lambda}$};
\end{tikzpicture}\;
&:=
\begin{tikzpicture}[anchorbase]
	\draw[-to,thin] (0,-.4) to (0,.4);
    \opendot{0,0};
    \pinO{(0,0)}{}{(1.3,0)}{(u+x)^{-1}};
   \node at (2.1,0) {$\catlabel{\lambda}$};
\end{tikzpicture}\;
= 
u^{-1}\: \begin{tikzpicture}[anchorbase]
	\draw[-to,thin] (0.08,-.4) to (0.08,.4);
   \node at (0.25,0) {$\catlabel{\lambda}$};
\end{tikzpicture}
-
u^{-2}\: \begin{tikzpicture}[anchorbase]
	\draw[-to,thin] (0.08,-.4) to (0.08,.4);
    \opendot{0.08,.05};
   \node at (0.3,0) {$\catlabel{\lambda}$};
\end{tikzpicture}
+
u^{-3}
\: \begin{tikzpicture}[anchorbase]
	\draw[-to,thin] (0.08,-.4) to (0.08,.4);
    \opendot{0.08,.15};
    \opendot{0.08,-.05};
   \node at (0.3,0) {$\catlabel{\lambda}$};
\end{tikzpicture}
-\cdots.
\end{align}
Now, unlike in \cref{rels9,rels9b}, we have simply that
\begin{align}
\label{rels9new}
\begin{tikzpicture}[anchorbase,scale=1.1]
	\draw[to-,thin] (0.4,0) to[out=90, in=0] (0.1,0.5);
	\draw[-,thin] (0.1,0.5) to[out = 180, in = 90] (-0.2,0);
\pinO{(0.4,.2)}{a}{(1.2,.2)}{u};
\end{tikzpicture}
&=
\begin{tikzpicture}[anchorbase,scale=1.1]
	\draw[to-,thin] (0.4,0) to[out=90, in=0] (0.1,0.5);
	\draw[-,thin] (0.1,0.5) to[out = 180, in = 90] (-0.2,0);
\pinO{(-0.2,.2)}{a}{(-1,.2)}{u};
\end{tikzpicture}\:.
&
\begin{tikzpicture}[anchorbase,scale=1.1]
	\draw[to-,thin] (0.4,0) to[out=-90, in=0] (0.1,-0.5);
	\draw[-,thin] (0.1,-0.5) to[out = 180, in = -90] (-0.2,0);
\pinO{(0.4,-.2)}{a}{(1.2,-.2)}{u};
\end{tikzpicture}
&=
\begin{tikzpicture}[anchorbase,scale=1.1]
	\draw[to-,thin] (0.4,0) to[out=-90, in=0] (0.1,-0.5);
	\draw[-,thin] (0.1,-0.5) to[out = 180, in = -90] (-0.2,0);
\pinO{(-0.2,-.2)}{a}{(-1,-.2)}{u};
\end{tikzpicture}\:,
\end{align}
and similarly for the other orientation.

We define the {\em fake bubble polynomials}
\begin{align}\label{KMbubblesalt}
\begin{tikzpicture}[anchorbase]
\filledclockwisebubble{(0,0)};
\node at (0,0) {$\kmlabel{u}$};
\node at (.4,0) {$\catlabel{\lambda}$};
\end{tikzpicture}
&:=
\sum_{n=0}^{-\lambda}
\begin{tikzpicture}[anchorbase]
\clockwisebubble{(0,0)};
\node at (0,0) {$\kmlabel{n}$};
\node at (.4,0) {$\catlabel{\lambda}$};
\end{tikzpicture}
\: u^{-\lambda-n},&
\begin{tikzpicture}[anchorbase]
\filledanticlockwisebubble{(0,0)};
\node at (0,0) {$\kmlabel{u}$};
\node at (.4,0) {$\catlabel{\lambda}$};
\end{tikzpicture}
&:=
\sum_{n=0}^\lambda
\begin{tikzpicture}[anchorbase]
\anticlockwisebubble{(0,0)};
\node at (0,0) {$\kmlabel{n}$};
\node at (.4,0) {$\catlabel{\lambda}$};
\end{tikzpicture}\:
u^{\lambda-n},
\end{align}
which are polynomials in $\End_{\fU(\sl_2)}(1_\lambda)[u]$ with
$\begin{tikzpicture}[anchorbase]
\filledclockwisebubble{(0,0)};
\node at (0,0) {$\kmlabel{u}$};
\node at (.4,0) {$\catlabel{\lambda}$};
\end{tikzpicture} = \delta_{\lambda,0} 1_{1_\lambda}$
when $\lambda \geq 0$
and
$\begin{tikzpicture}[anchorbase]
\filledanticlockwisebubble{(0,0)};
\node at (0,0) {$\kmlabel{u}$};
\node at (.4,0) {$\catlabel{\lambda}$};
\end{tikzpicture} = \delta_{\lambda,0} 1_{1_\lambda}$
when $\lambda \leq 0$.
It is often convenient to combine the fake bubble polynomials 
with generating functions for genuinely dotted bubbles by letting
\begin{align}\label{southwalespolice1}
{\catlabel{\lambda}\:}
\textclockwisebubble(u) &:= 
\begin{tikzpicture}[anchorbase]
\filledclockwisebubble{(0,0)};
\node at (0,0) {$\kmlabel{u}$};
\node at (.4,0) {$\catlabel{\lambda}$};
\end{tikzpicture}+\begin{tikzpicture}[anchorbase]
\clockwisebubble{(0,0)};
\pinO{(-.2,0)}{-}{(-.9,0)}{u};
\node at (.4,0) {$\catlabel{\lambda}$};
\end{tikzpicture}
\in u^{-\lambda} 1_{1_\lambda}
+ u^{-\lambda-1} \End_{\fU(\sl_2)}(1_\lambda)\llbracket u^{-1}\rrbracket
,\\\label{southwalespolice2}
{\catlabel{\lambda}\:}
\textanticlockwisebubble(u) &:= \begin{tikzpicture}[anchorbase]
\filledanticlockwisebubble{(0,0)};
\node at (0,0) {$\kmlabel{u}$};
\node at (.4,0) {$\catlabel{\lambda}$};
\end{tikzpicture}+
\begin{tikzpicture}[anchorbase]
\anticlockwisebubble{(0,0)};
\pinO{(.2,0)}{-}{(.9,0)}{u};
\node at (1.25,0) {$\catlabel{\lambda}$};
\end{tikzpicture}\in u^{\lambda} 1_{1_\lambda}
+
u^{\lambda-1} \End_{\fU(\sl_2)}(1_\lambda)\llbracket u^{-1}\rrbracket.
\end{align}
As explained originally in \cite[Prop.~8.2]{Lauda2} (see also \cite[(3.11)--(3.12)]{BD}),
for any $\lambda \in \Z$,
the algebra $\End_{\fU(\sl_2)}(1_\lambda)$ may be identified
with the algebra 
$\Lambda$ of symmetric functions
so that 
${\catlabel{\lambda}\:}
\textclockwisebubble(u)$ and
${\catlabel{\lambda}\:}
\textanticlockwisebubble(u)$ as just defined are identified with the
generating functions $u^{-\lambda} e(-u)$ and $u^\lambda h(u)$ from \cref{genfuncs}.
In particular, we have that
\begin{equation}\label{infgrass}
{\catlabel{\lambda}\:}
\textclockwisebubble(u)\: \textanticlockwisebubble(u) = 1_{1_\lambda}
\end{equation}
as in \cref{grassmannian}.

The following relations are proved in \cite[Cor.~3.5]{Brundan}:
\begin{align}\label{curlrel1}
\begin{tikzpicture}[anchorbase,scale=1]
	\draw[to-,thin] (0,0.6) to (0,0.3);
	\draw[-,thin] (0,0.3) to [out=-90,in=180] (.3,-0.2);
	\draw[-,thin] (0.3,-0.2) to [out=0,in=-90](.5,0);
	\draw[-,thin] (0.5,0) to [out=90,in=0](.3,0.2);
	\draw[-,thin] (0.3,.2) to [out=180,in=90](0,-0.3);
	\draw[-,thin] (0,-0.3) to (0,-0.6);
   \node at (0.75,0) {$\catlabel{\lambda}$};
\end{tikzpicture}
&=
-\sum_{n=0}^{-\lambda} \begin{tikzpicture}[anchorbase,scale=1]
	\draw[to-,thin] (0,0.6) to (0,-0.6);
	\opendot{0,0};\node at (-.6,0) {$\kmlabel{-\lambda-n}$};
\clockwisebubble{(.5,0)};\node at (.5,0) {$\kmlabel{n}$};
\node at (0.95,0) {$\catlabel{\lambda}$};
\end{tikzpicture}=-\left[ \:\begin{tikzpicture}[anchorbase,scale=1]
	\draw[to-,thin] (0,0.6) to (0,-0.6);
	\pinO{(0,0)}{-}{(.8,0)}{u};
	\filledclockwisebubble{(1.4,0)};\node at (1.4,0) {$\kmlabel{u}$};
\node at (1.8,0) {$\catlabel{\lambda}$};
\end{tikzpicture}\right]_{u^{-1}},\\\label{curlrel2}
\begin{tikzpicture}[anchorbase,scale=1]
	\draw[to-,thin] (0,0.6) to (0,0.3);
	\draw[-,thin] (0,0.3) to [out=-90,in=0] (-.3,-0.2);
	\draw[-,thin] (-0.3,-0.2) to [out=180,in=-90](-.5,0);
	\draw[-,thin] (-0.5,0) to [out=90,in=180](-.3,0.2);
	\draw[-,thin] (-0.3,.2) to [out=0,in=90](0,-0.3);
	\draw[-,thin] (0,-0.3) to (0,-0.6);
     \node at (-0.8,0) {$\catlabel{\lambda}$};
\end{tikzpicture}
&=
\sum_{n=0}^{\lambda} \begin{tikzpicture}[anchorbase,scale=1]
	\draw[to-,thin] (0,0.6) to (0,-0.6);
	\opendot{0,0};\node at (.4,0) {$\kmlabel{\lambda-n}$};
\anticlockwisebubble{(-.5,0)};\node at (-.5,0) {$\kmlabel{n}$};
\node at (-.9,0) {$\catlabel{\lambda}$};
\end{tikzpicture}=
\left[ \begin{tikzpicture}[anchorbase,scale=1]
	\draw[to-,thin] (0,0.6) to (0,-0.6);
	\pinO{(0,0)}{-}{(-.8,0)}{u};
	\filledanticlockwisebubble{(-1.4,0)};\node at (-1.4,0) {$\kmlabel{u}$};
\node at (-1.8,0) {$\catlabel{\lambda}$};
\end{tikzpicture}\:\right]_{u^{-1}}\:,\end{align}
with the second equalities being 
easily checked by equating coefficients; here and below we use $[f(u)]_{u^r}$ to denote the coefficient of $u^r$ of a formal Laurent series $f(u)$ in $u^{-1}$. Other defining relations can be written similarly in terms of generating functions. For example, the following are equivalent to \cref{KMrels3}:
\begin{align}\label{KMrels3equiv}
\begin{tikzpicture}[anchorbase,scale=.95]
	\draw[-to,thin] (0.28,0) to[out=90,in=-90] (-0.28,.7);
	\draw[-,thin] (-0.28,0) to[out=90,in=-90] (0.28,.7);
	\draw[to-,thin] (0.28,-.7) to[out=90,in=-90] (-0.28,0);
	\draw[-,thin] (-0.28,-.7) to[out=90,in=-90] (0.28,0);
  \node at (.5,0) {$\catlabel{\lambda}$};
\end{tikzpicture}
&=
-
\begin{tikzpicture}[anchorbase,scale=.95]
	\draw[to-,thin] (0.2,-.7) to (0.2,.7);
	\draw[-to,thin] (-0.3,-.7) to (-0.3,.7);
   \node at (.35,.05) {$\catlabel{\lambda}$};
\end{tikzpicture}
+
\left[\:
\begin{tikzpicture}[anchorbase,scale=.95]
	\draw[-,thin] (0.3,0.7) to[out=-90, in=0] (0,0.1);
	\draw[-to,thin] (0,0.1) to[out = 180, in = -90] (-0.3,0.7);
    \node at (0.4,0) {$\catlabel{\lambda}$};
    \filledanticlockwisebubble{(-.7,0)};
      \node at (-.7,0) {$\kmlabel{u}$};
     \pinO{(-0.28,0.45)}{-}{(-1,.45)}{u};
     	\draw[to-,thin] (0.3,-.7) to[out=90, in=0] (0,-0.1);
	\draw[-,thin] (0,-0.1) to[out = 180, in = 90] (-0.3,-.7);
     \pinO{(-0.28,-0.5)}{-}{(-1,-.5)}{u};
\end{tikzpicture}
\right]_{u^{-1}},&
\begin{tikzpicture}[anchorbase,scale=.95]
	\draw[-,thin] (0.28,0) to[out=90,in=-90] (-0.28,.7);
	\draw[-to,thin] (-0.28,0) to[out=90,in=-90] (0.28,.7);
	\draw[-,thin] (0.28,-.7) to[out=90,in=-90] (-0.28,0);
	\draw[to-,thin] (-0.28,-.7) to[out=90,in=-90] (0.28,0);
  \node at (.5,0) {$\catlabel{\lambda}$};
\end{tikzpicture}
&=
-
\begin{tikzpicture}[anchorbase,scale=.95]
	\draw[-to,thin] (0.2,-.7) to (0.2,.7);
	\draw[to-,thin] (-0.3,-.7) to (-0.3,.7);
   \node at (.35,0) {$\catlabel{\lambda}$};
\end{tikzpicture}
+\left[
\begin{tikzpicture}[anchorbase,scale=.95]
	\draw[-,thin] (0.3,-0.7) to[out=90, in=0] (0,-0.1);
	\draw[-to,thin] (0,-0.1) to[out = 180, in = 90] (-0.3,-.7);
   \pinO{(0.29,-0.5)}{-}{(1,-.5)}{u};
      \filledclockwisebubble{(.53,0)};
      \node at (.53,0) {$\kmlabel{u}$};
   \node at (-0.4,0) {$\catlabel{\lambda}$};
	\draw[to-,thin] (0.3,.7) to[out=-90, in=0] (0,0.1);
	\draw[-,thin] (0,0.1) to[out = -180, in = -90] (-0.3,.7);
   \pinO{(0.29,0.45)}{-}{(1,.45)}{u};
\end{tikzpicture}\:\right]_{u^{-1}}.
\end{align}
Next, we have the {\em bubble slide relations}
\begin{align}\label{bubbleslide1}
\begin{tikzpicture}[anchorbase,scale=1]
	\node at (-.6,0) {$\textclockwisebubble(u)$};
	\draw[to-,thin] (0,0.6) to (0,-0.6);
\node at (.3,0) {$\catlabel{\lambda}$};
\end{tikzpicture}
&=
\begin{tikzpicture}[anchorbase,scale=1]
	\node at (1.4,0) {$\textclockwisebubble(u)$};
	\draw[to-,thin] (0,0.6) to (0,-0.6);
	\pinO{(0,.2)}{-}{(.6,.2)}{u};
	\pinO{(0,-.2)}{-}{(.6,-.2)}{u};
\node at (2,0) {$\catlabel{\lambda}$};
\end{tikzpicture}
\:,&
\begin{tikzpicture}[anchorbase,scale=1]
	\draw[to-,thin] (0,0.6) to (0,-0.6);
	\node at (.7,0) {$\textanticlockwisebubble(u)$};
\node at (1.4,0) {$\catlabel{\lambda}$};
\end{tikzpicture}
&=
\begin{tikzpicture}[anchorbase,scale=1]
\node at (-1.4,0) {$\textanticlockwisebubble(u)$};
\draw[to-,thin] (0,0.6) to (0,-0.6);
	\pinO{(0,.2)}{-}{(-.6,.2)}{u};
	\pinO{(0,-.2)}{-}{(-.6,-.2)}{u};
\node at (.3,0) {$\catlabel{\lambda}$};
\end{tikzpicture},
\end{align}
which follow from \cite[Prop.~3.4]{KL3}. 
They imply similar
relations for the fake bubble polynomials:
\begin{align}
\label{bubbleslide2}
\begin{tikzpicture}[anchorbase,scale=1]
	\draw[to-,thin] (0,0.6) to (0,-0.6);
\filledclockwisebubble{(-.5,0)};
\node at (-.5,0) {$\kmlabel{u}$};
\node at (.3,0) {$\catlabel{\lambda}$};
\end{tikzpicture}
&=
\left[\begin{tikzpicture}[anchorbase,scale=1]
	\draw[to-,thin] (0,0.6) to (0,-0.6);
	\pinO{(0,.2)}{-}{(.6,.2)}{u};
	\pinO{(0,-.2)}{-}{(.6,-.2)}{u};
\filledclockwisebubble{(1.4,0)};
\node at (1.4,0) {$\kmlabel{u}$};
\node at (2,0) {$\catlabel{\lambda}$};
\end{tikzpicture}\right]_{u^{\geq 0}},&
\begin{tikzpicture}[anchorbase,scale=1]
	\draw[to-,thin] (0,0.6) to (0,-0.6);
\filledanticlockwisebubble{(.5,0)};
\node at (.5,0) {$\kmlabel{u}$};
\node at (1,0) {$\catlabel{\lambda}$};
\end{tikzpicture}
&=
\left[\begin{tikzpicture}[anchorbase,scale=1]
	\draw[to-,thin] (0,0.6) to (0,-0.6);
	\pinO{(0,.2)}{-}{(-.6,.2)}{u};
	\pinO{(0,-.2)}{-}{(-.6,-.2)}{u};
\filledanticlockwisebubble{(-1.4,0)};
\node at (-1.4,0) {$\kmlabel{u}$};
\node at (.3,0) {$\catlabel{\lambda}$};
\end{tikzpicture}\right]_{u^{\geq 0}}.
\end{align}
Finally, we have the
{\em alternating braid relation}
\begin{align}\label{altbraid}
\begin{tikzpicture}[anchorbase,scale=1.05]
	\draw[to-,thin] (0.45,.8) to (-0.45,-.4);
	\draw[-to,thin] (0.45,-.4) to (-0.45,.8);
        \draw[to-,thin] (0,-.4) to[out=90,in=-90] (-.45,0.2);
        \draw[-,thin] (-0.45,0.2) to[out=90,in=-90] (0,0.8);
   \node at (.5,.2) {$\catlabel{\lambda}$};
\end{tikzpicture}
-
\begin{tikzpicture}[anchorbase,scale=1.05]
	\draw[to-,thin] (0.45,.8) to (-0.45,-.4);
	\draw[-to,thin] (0.45,-.4) to (-0.45,.8);
        \draw[to-,thin] (0,-.4) to[out=90,in=-90] (.45,0.2);
        \draw[-,thin] (0.45,0.2) to[out=90,in=-90] (0,0.8);
   \node at (.7,.2) {$\catlabel{\lambda}$};
\end{tikzpicture}&=
\left[\:
\begin{tikzpicture}[anchorbase,scale=.95]
	\draw[-,thin] (0.3,0.7) to[out=-90, in=0] (0,0.1);
	\draw[-to,thin] (0,0.1) to[out = 180, in = -90] (-0.3,0.7);
	\draw[-to,thin] (1.3,-0.7) to (1.3,0.7);
   \pinO{(1.3,0)}{-}{(.6,0)}{u};
    \node at (1.6,0) {$\catlabel{\lambda}$};
    \filledanticlockwisebubble{(-.7,0)};
      \node at (-.7,0) {$\kmlabel{u}$};
     \pinO{(-0.28,0.45)}{-}{(-1,.45)}{u};
     	\draw[to-,thin] (0.3,-.7) to[out=90, in=0] (0,-0.1);
	\draw[-,thin] (0,-0.1) to[out = 180, in = 90] (-0.3,-.7);
     \pinO{(-0.28,-0.5)}{-}{(-1,-.5)}{u};
\end{tikzpicture}
+
\begin{tikzpicture}[anchorbase,scale=.95]
	\draw[-,thin] (0.3,-0.7) to[out=90, in=0] (0,-0.1);
	\draw[-to,thin] (0,-0.1) to[out = 180, in = 90] (-0.3,-.7);
	\draw[-to,thin] (-1.2,-0.7) to (-1.2,0.7);
   \pinO{(-1.2,0)}{-}{(-.5,0)}{u};
   \pinO{(0.29,-0.5)}{-}{(1,-.5)}{u};
      \filledclockwisebubble{(.53,0)};
      \node at (.53,0) {$\kmlabel{u}$};
   \node at (1.1,0) {$\catlabel{\lambda}$};
	\draw[to-,thin] (0.3,.7) to[out=-90, in=0] (0,0.1);
	\draw[-,thin] (0,0.1) to[out = -180, in = -90] (-0.3,.7);
   \pinO{(0.29,0.45)}{-}{(1,.45)}{u};
\end{tikzpicture}\:\right]_{u^{-1}}
\end{align}
from \cite[Prop.~3.5]{KL3}.

In order to make a connection between $\fU(\sl_2)$ and the nil-Brauer category, we need to localize at the morphisms
\begin{equation}
\begin{tikzpicture}[anchorbase]
\draw[-to,thin] (-.6,-.3) to (-.6,.3);
\draw[-,thick] (0,-.3) to (0,.3);
\node at (0,-.5) {$\kmlabel{X}$};
\draw[-to,thin] (.6,-.3) to (.6,.3);
\limitbandOO{(-.6,0)}{}{}{(.6,0)};
\node at (0,.5) {$\phantom{\kmlabel{X}}$};
\node at (.9,0) {$\catlabel{\lambda}$};
\end{tikzpicture}
:=
\begin{tikzpicture}[anchorbase]
\draw[-to,thin] (-.6,-.3) to (-.6,.3);
\draw[-,thick] (0,-.3) to (0,.3);
\draw[-to,thin] (.6,-.3) to (.6,.3);
\node at (0,-.5) {$\kmlabel{X}$};
\opendot{-.6,0};
\node at (0,.5) {$\phantom{\kmlabel{X}}$};
\node at (.9,0) {$\catlabel{\lambda}$};
\end{tikzpicture}
+
\begin{tikzpicture}[anchorbase]
\draw[-to,thin] (-.6,-.3) to (-.6,.3);
\draw[-,thick] (0,-.3) to (0,.3);
\node at (0,-.45) {$\kmlabel{X}$};
\draw[-to,thin] (.6,-.3) to (.6,.3);
\opendot{.6,0};
\node at (0,.45) {$\phantom{\kmlabel{X}}$};
\node at (.9,0) {$\catlabel{\lambda}$};
\end{tikzpicture}\label{bands}
\end{equation}
for all $\lambda,\mu \in \Z$ and all 1-morphisms
$X:\lambda+2\rightarrow \mu-2$ in $\fU(\sl_2)$.
This means that we adjoin additional generating 2-morphisms
$\begin{tikzpicture}[anchorbase,scale=.8]
\draw[-to,thin] (-.6,-.3) to (-.6,.3);
\draw[-,thick] (0,-.3) to (0,.3);
\draw[-to,thin] (.6,-.3) to (.6,.3);
\limitteleporterOO{(-.6,0)}{}{}{(.6,0)};
\node at (0,-.45) {$\kmlabel{X}$};
\node at (0,.45) {$\phantom{\kmlabel{X}}$};
\node at (.9,0) {$\catlabel{\lambda}$};
\end{tikzpicture}
:EXE 1_\lambda \Rightarrow EXE 1_\lambda$ subject to the additional relations
\begin{equation}\label{teleporters}
\begin{tikzpicture}[anchorbase]
\draw[-to,thin] (-.6,-.3) to (-.6,.3);
\draw[-,thick] (0,-.3) to (0,.3);
\draw[-to,thin] (.6,-.3) to (.6,.3);
\limitteleporterOO{(-.6,0)}{}{}{(.6,0)};
\node at (0,-.45) {$\kmlabel{X}$};
\node at (0,.45) {$\phantom{\kmlabel{X}}$};
\node at (.9,0) {$\catlabel{\lambda}$};
\end{tikzpicture}
=
\left(
\begin{tikzpicture}[anchorbase]
\draw[-to,thin] (-.6,-.3) to (-.6,.3);
\draw[-,thick] (0,-.3) to (0,.3);
\draw[-to,thin] (.6,-.3) to (.6,.3);
\limitbandOO{(-.6,0)}{}{}{(.6,0)};
\node at (0,.45) {$\phantom{\kmlabel{X}}$};
\node at (0,-.45) {$\kmlabel{X}$};
\node at (.9,0) {$\catlabel{\lambda}$};
\end{tikzpicture}\right)^{-1}.
\end{equation}
By some analogy with \cite[(4.21)]{Foundations}, 
we refer to the 2-morphisms \cref{teleporters} as {\em teleporters}\footnote{They are also closely related to the morphisms called {\em dumbbells} in \cite{K0}.}: the relation
\begin{equation}
\begin{tikzpicture}[anchorbase,scale=1.2]
 	\draw[-to,thin] (.2,-.3) to (.2,.3);
\draw[-,thick] (-.3,-.3) to (-.3,.3);
\node at (-.3,-.45) {$\kmlabel{X}$};
\node at (-.3,.45) {$\phantom{\kmlabel{X}}$};
	\draw[-to,thin] (-.8,-.3) to (-0.8,.3);
	\limitteleporterOO{(-.8,.1)}{}{}{(.2,.1)};
	\opendot{-.8,-.1};
      \node at (.5,0) {$\catlabel{\lambda}$};
\end{tikzpicture}
\:\:+\;\:
\begin{tikzpicture}[anchorbase,scale=1.2]
 	\draw[-to,thin] (.2,-.3) to (.2,.3);
\draw[-,thick] (-.3,-.3) to (-.3,.3);
\node at (-.3,-.45) {$\kmlabel{X}$};
\node at (-.3,.45) {$\phantom{\kmlabel{X}}$};
	\draw[-to,thin] (-.8,-.3) to (-0.8,.3);
	\limitteleporterOO{(-.8,.1)}{}{}{(.2,.1)};
	\opendot{.2,-.1};
      \node at (.5,0) {$\catlabel{\lambda}$};
\end{tikzpicture}
=
\begin{tikzpicture}[anchorbase,scale=1.2]
 	\draw[-to,thin] (.2,-.3) to (.2,.3);
\draw[-,thick] (-.3,-.3) to (-.3,.3);
\node at (-.3,.45) {$\phantom{\kmlabel{X}}$};
\node at (-.3,-.45) {$\kmlabel{X}$};
	\draw[-to,thin] (-.8,-.3) to (-0.8,.3);
      \node at (.45,0) {$\catlabel{\lambda}$};
\end{tikzpicture}\label{dumby0}
\end{equation}
means that dots can ``teleport" across teleporters (hence, the name!).
We denote the strict 2-category 
obtained in this way by
$\fU(\sl_2)_\loc$; it also admits a grading
like in \cref{KMgrading}.

In $\fU(\sl_2)_\loc$, it is easy to see using \cref{KMrels2,otheradjunction} that the 2-morphisms defined similarly to \cref{bands} 
but with one or both of the upward strings changed to downward strings are also invertible; we denote their inverses in the obvious way by modifying the directions of arrows in \cref{teleporters}.
The dotted and solid horizontal lines in all of these diagrams
are present merely to indicate that the open dots at the endpoints have been identified---they are not a part of the string calculus so can be moved freely around larger diagrams as long as the endpoints remain fixed.
It will often be convenient to allow open dots to be connected by dotted or solid lines even when the endpoints are not at the same horizontal level. Such string diagrams
may be interpreted as morphisms by using planar isotopy to redraw the diagrams so that the endpoints are aligned.
For example, we have that
\begin{equation}\label{tusan}

\:.
\end{equation}
It is similarly straightforward to check that
$$
%
\right)^{-1}.
$$
From \cref{tusan}, we deduce that open dots are invertible in $\fU(\sl_2)_\loc$. 
We will denote the inverse of $\textupdot{\scriptstyle\color{catcolor}\lambda}$
simply by labeling the dot by $-1$, and have that
\begin{equation}\label{tak}
%
\:.
\end{equation}
From \cref{dumby0}, we deduce that
\begin{align}
\label{dumby1}
%
\:.\\\intertext{We also have that}
\label{dumby2}
%
\:.
\end{align}
This and the following relations are easily deduced in a similar way to the proof of \cref{rels10}:
\begin{align}
\label{dumby3}
%
\:.
\end{align}
In the literature, the fake bubbles
$%
\:$
are often denoted by
$%
$, respectively.
We have avoided the latter convention because in $\fU(\sl_2)_\loc$
it makes sense to consider
{\em genuine} dotted bubbles with dot labeled by $-1$.

\begin{lemma}\label{leavinghome}
$
\left(
%
\right)
=1.
$
\end{lemma}

\begin{proof}
By \cref{dumby3}, we have that
$$
%
\:.
$$
Then we expand the curls on the left-hand side using \cref{curlrel1,curlrel2} to obtain
$$
\sum_{n=0}^{\lambda} %
\:.
$$
By \cref{KMrels5},
the clockwise dotted bubble on the left-hand side here is 0 unless either $n=0< \lambda$ (when it equals $1_{1_\lambda}$) or $n=\lambda$, and the counterclockwise one is 0 unless either $n=0 < -\lambda$ (when it equals $1_{1_\lambda}$) or $n=-\lambda$.
Using this, and considering the cases $\lambda >0, \lambda = 0$ or $\lambda < 0$ separately,
the equation obtained so far reduces to the identity we are trying to prove.
\end{proof}

\begin{lemma}\label{veryeasy}
$%
$.
\end{lemma}

\begin{proof}
Using the rotated version of the last relation from 
\cref{KMrels1} to commute all the dots to the bottom, one checks easily that
$%
$.
Now compose vertically with teleporters on the top and the bottom.
\end{proof}

We come to what we promise is the last diagrammatical shorthand,
which is a counterpart of \cite[(5.27)--(5.28)]{K0}:
we define the {\em internal bubbles}
\begin{align}
%
\right]_{u^{-1}}\label{internal1}\\
&\text{(degree $2\lambda-2$)}\notag
\end{align}
We also introduce internal bubbles on downward strings by taking the mates of \cref{internal2,internal1},
so that we maintain invariance under planar isotopy.
It can be checked easily that internal bubbles commute with open dots and with 
other internal bubbles on the same string.

\begin{lemma}\label{hazistalk}
The following hold:
\begin{enumerate}
\item
$%
$ if $\lambda \geq -2$.
\end{enumerate}
\end{lemma}

\begin{proof}
We just prove (1), then (2) follows by applying the symmetry $\tR$ (which interchanges the two sorts of internal bubble).
Suppose that $\lambda \leq 2$.
The idea is to compute
\begin{equation}\label{benstalk}
-%
\end{equation} 
in two different ways.
Starting from the form on the left-hand side, we slide the dot on the bubble 
past the crossing above it using \cref{dumby5} and the first relation from \cref{KMrels1} to obtain
$$
-%
\:.
$$
Then we slide the dot on the curl up past the crossing using \cref{dumby5} again to get
\begin{align*}
%
\:.
\end{align*}
On the other hand, we can simplify the right-hand side
of \cref{benstalk} using the second relation from \cref{KMrels3}
to obtain
$$
2%
\:.
$$
The desired equality follows.
\end{proof}

\begin{lemma}\label{stpauls}
The clockwise and counterclockwise internal bubbles are the two-sided inverses of each other, i.e., we have that
$$
%
\:.
$$
\end{lemma}

\begin{proof}
We prove the second equality in the case that $\lambda \geq -1$.
The first equality in this case then follows since internal bubbles commute, and after that the result may be deduced when $\lambda \leq -1$ too by applying the symmetry $\tR$.
So assume that $\lambda \geq -1$.
By \cref{hazistalk}(2) and the definition \cref{internal1} of the counterclockwise internal bubble, we have that
$$
%
\right]_{u^{-1}}.
$$
Now we observe that
$$
%
\right]_{u^{-1}}.
$$
Substituting this into the previous formula gives
$$
%
\:.
$$
To get the last equality here, we used 
that $\left[\:
%
\right]_{u^{\geq 0}} = 1_{1_{\lambda+2}}$,
which follows from \cref{infgrass} on expanding the definitions \cref{southwalespolice1,southwalespolice2}.
\end{proof}

\begin{lemma}\label{sunriver}
$%
= \frac{1-(-1)^\lambda}{2} 1_{1_\lambda}.
$
\end{lemma}

\begin{proof}
Expanding the definitions of the internal bubbles, also multiplying both sides by 2
and using the identity
$%
$ which follows by \cref{dumby1}, this reduces to checking that
$$
%
=  (1-(-1)^\lambda) 1_{1_\lambda}.
$$
This follows in the three cases $\lambda > 0$, $\lambda = 0$
or $\lambda < 0$ using \cref{KMrels5} and \cref{leavinghome}.
\end{proof}

\begin{lemma}\label{lazy}
We have that
$$
%
\right]_{u^{-1}}
$$
for either choice of orientation of the left-hand string.
\end{lemma}

\begin{proof}
We are going to expand
$$
-\:
%
$$
in two different ways, neither of which depends on the chosen orientation of the left-hand string.
Commuting the dot on the bubble past the crossing above it 
using \cref{dumby5} and the first relation from \cref{KMrels1} gives
\begin{align*}
-\:%
\right]_{u^{-1}}.
\end{align*}
On the other hand, we can apply \cref{KMrels3equiv} to obtain
\begin{align*}
%
\right]_{u^{-1}}\!\!.
\end{align*}
The last terms produced in each case cancel with each other, and the lemma is proved.
\end{proof}

\begin{corollary}\label{funeral}
$
%
\:
\right]_{u^{-1}}$.
\end{corollary}

\begin{proof}
This follows from \cref{lazy} 
(taking the left-hand string to be oriented downward)
using also
the definition of
the internal bubble that is the mate of \cref{internal2}.
\end{proof}

\begin{corollary}\label{funeral2}
$
%
\:
$ assuming $\lambda \leq -2$.
\end{corollary}

\begin{proof}
This follows from \cref{lazy} 
(taking the left-hand string to be oriented upward)
using also \cref{KMbubblesalt}
and
\cref{hazistalk}(1).
\end{proof}

\begin{lemma}\label{lazier}
$%
\right]_{u^{-1}}$.
\end{lemma}

\begin{proof}
We have that
\begin{align*}
%
\:
$.
\end{lemma}

\begin{proof}
We prove the first equality assuming that $\lambda \leq -2$.
Then the second equality for $\lambda \leq -2$ can be deduced 
using \cref{stpauls},
and after that both equalities for $\lambda \geq -2$ follow by applying $\tR$.
So assume that $\lambda \leq -2$. By \cref{hazistalk}(1) and \cref{funeral2}, we have that
$$
%
\:.
$$
We need to show that this commutes with the crossing $\textupcrossing{\scriptstyle\color{catcolor}\lambda}$.
It is easy to check that the various dots do so, indeed, 
$x_1 x_2$ and $x_1+x_2$ are central elements in the nil-Hecke algebra
$NH_2$. It is almost as easy to see that the counterclockwise 
bubbles commute with the crossing too:
\begin{align*}
%
\:.
\end{align*}
The lemma is proved. 
\end{proof}

\begin{lemma}\label{florence}
$%
\right]_{u^{-1}}\!\!\!\!\!$.
\end{lemma}

\begin{proof}
We just have to use the rotations of \cref{dumby3} then \cref{dumby4} to 
commute dots downward past the crossing:
\begin{align*}
%
\right]_{u^{-1}}\!\!\!\!\!$.
\end{corollary}

\begin{proof}
In the identity from \cref{florence}, we use
the image of \cref{funeral} under $\tR$ to combine the first and fourth terms on the right-hand side to obtain the identity
$$
%
\right]_{u^{-1}}\!\!\!\!\!.
$$
It remains to rotate clockwise by $90^\circ$ then apply $\tR$.
\end{proof}

\begin{lemma}\label{werehere}
$%
.$
\end{lemma}

\begin{proof}
Closing the identity established in \cref{florence} on the top with a rightward cap
and using \cref{curlrel1} to expand the two curls produced gives that
$%
\right]_{u^{-1}}.
\end{align*}
After rewriting $A$, $B$ and $D$ as indicated,
the first terms from the expansions of $A$ and $B$ combine
according to the definition \cref{internal2}
to produce
the second term $-%
$ from the formula we are trying to prove.
Similarly, the second terms from the expansions of $A$ and $B$ combine to give
$%
$.
Adding this to $C$ and $D$ and applying \cref{sunriver} gives
$\frac{1-(-1)^\lambda+\delta_{\lambda,0}}{2}
%
$.
Comparing with the first term in the formula we are trying to prove, 
we are left with showing that the remaining term from $A$ plus
the term from $E$ equals
${\textstyle\frac{(-1)^\lambda - \delta_{\lambda,0}}{2}}\:%
$.
This holds because these two terms can be simplified as follows:
\begin{align}\label{barre1}
-
{\textstyle\frac{1}{2}}
\left[

\right]_{u^{-1}}
&=
\left\{\begin{array}{ll}
{\textstyle\frac{(-1)^\lambda}{2}}
\:\begin{tikzpicture}[anchorbase,scale=1]
	\draw[-,thin] (0.3,-0.3) to[out=90, in=0] (0,0.4);
	\draw[-to,thin] (0,0.4) to[out = 180, in = 90] (-0.3,-0.3);
\opendot{.27,.1};\node at (.55,.1) {$\kmlabel{-1}$};
\node at (.9,0) {$\catlabel{\lambda}$};
\end{tikzpicture}&\text{if $\lambda < 0$}\\
0&\text{otherwise},
\end{array}\right]_{u^{-1}}&=\left\{
\begin{array}{ll}
{\textstyle\frac{(-1)^\lambda}{2}}
\:\begin{tikzpicture}[anchorbase,scale=1]
	\draw[-,thin] (0.3,-0.3) to[out=90, in=0] (0,0.4);
	\draw[-to,thin] (0,0.4) to[out = 180, in = 90] (-0.3,-0.3);
\opendot{.27,.1};\node at (.55,.1) {$\kmlabel{-1}$};
\node at (.9,0) {$\catlabel{\lambda}$};
\end{tikzpicture}&\text{if $\lambda>  0$}\\
0&\text{otherwise.}
\end{array}\right.
\end{align}
The arguments to prove \cref{barre1,barre2} are similar, so we just explain 
how to see the second one.
The counterclockwise fake bubble polynomial is 0 if $\lambda < 0$,
so the expression on the left-hand side of 
\cref{barre2} is 0 in this case, as required.
The fake bubble polynomial is the identity if $\lambda = 0$, 
and it is easy to see that \cref{barre2} is true in this case too
due to the presence of two pins. Finally, if $\lambda > 0$, the fake bubble polynomial
is $u^\lambda + ($lower terms), and
by \cref{KMrels5} we have that
$%
=(-1)^{\lambda-1} u^{-\lambda}+($lower terms).
It follows that the $u^{-1}$-coefficient on the left-hand side of \cref{barre2} 
comes from the leading terms, and the identity follows easily.
\end{proof}

\begin{corollary}\label{useless}
$%
.$
\end{corollary}

\begin{proof}
By a rotated version of \cref{dumby5}, we have that
$$
%
\:.
$$
It remains to apply \cref{werehere} to the first term on the right-hand side.
\end{proof}

\begin{lemma}\label{fishing}
$
%
$.
\end{lemma}

\begin{proof}
We begin by attaching a rightward crossing to the top of the identity from
\cref{florence} to get that
$%
\right]_{u^{-1}}.
\end{align*}
It just remains to gather the terms together to see that we get exactly the three terms on the right-hand side of the formula claimed
in the statement of the lemma.
The first terms from $A$ and $D$ combine according to \cref{funeral} (rotated through $180^\circ$)
to give the second term in the claimed formula.
The second term from $A$ cancels with the first term from $B$.
The third term from $B$ gives the third term in the claimed formula.
The fourth term from $B$ cancels with the second term from $C$.
The first term from $C$ gives the first term in the claimed formula.
This just leaves the second term in $B$, the second term in $D$
and both terms in $E$. 
The second term in $D$ and the first term in $E$
are easily seen to equal 0 using the definitions \cref{KMbubblesalt}.
The second terms in $B$ and $E$ are 0 too, as follows by considering leading terms like in the proof of \cref{barre1,barre2}.
\end{proof}

\begin{lemma}\label{hideaway}
$\left[
\!
	\right]_{u^{-1}}$.
\end{lemma}

\begin{proof}
We begin by calculating:
\begin{align*}
\quad&\!\!\!\!
%
\!\right]_{u^{-1}}\!\!\!\!\!.
	\end{align*}
	In deriving the last equality here, we have removed two terms that are 0 by the same argument as used to prove \cref{barre1,barre2}.
	Then we use \cref{internal1} to combine the first and third terms to obtain the identity
	$$
%
\!\right]_{u^{-1}}\!\!\!.
	$$
	Similarly, we use \cref{funeral} (or rather its image under $\tT$)
	to combine the term on the left-hand side with the fourth term on the right-hand side to reduce to
		$$
2%
\!\right]_{u^{-1}}\!\!\!.
	$$
	Dividing by 2 and adding an inverse dot to the right-hand string, 
	this rearranges to show that
$$
{\textstyle\frac{1}{2}}
\left[
	%
\!
	\right]_{u^{-1}}\!\!\!\!\!\!,
	$$
	the additional second term on the left-hand side
	obviously being equal to 0.
	It remains to observe that the left-hand side in this is equal to the left-hand
	in the formula we are trying to prove by \cref{internal1}.
\end{proof}

\begin{lemma}\label{lastgasp}
$%
$
for $\lambda \leq -1$.
\end{lemma}

\begin{proof}
We expand the left-hand side:
\begin{align*}
\quad&\!\!\!\!%
.
\end{align*}
Now we use the definition \cref{internal1}
to rewrite
the fourth and sixth terms in this expression:
\begin{multline*}
%
\right]_{u^{-1}}\!\!\!\!\!\!.
\end{multline*}
It just remains to 
expand the fake bubble polynomials 
using the assumption that $\lambda \leq -1$ to see that
$$
\left[
%
.
$$
The lemma is proved. 
\end{proof}

%=====================
% Section 4
%=====================

\section{Monoidal functor from 
\texorpdfstring{$\cNB_t$}{}
to a localized version of
\texorpdfstring{$\fU(\sl_2)$}{}}\label{sec4}

We are now in position to construct the strict monoidal functor
$\Omega_t:\cNB_t \rightarrow \Add\left(\cU(\sl_2;t)_\loc\right)$.
The latter category is the additive envelope of
a monoidal category 
obtained by collapsing the 2-categorical structure on $\fU(\sl_2)_\loc$. Its full definition is as follows:

\begin{definition}\label{hearts}
Let $\cU(\sl_2;t)_\loc$ 
be the 
monoidal category with objects that are words in the free monoid $\langle E, F \rangle$ generated by the letters $E$ and $F$.
For any $X, Y \in \langle E, F \rangle$ and $\lambda \in \Z$,
there are corresponding 1-morphisms $X 1_\lambda, Y 1_\lambda$ in $\fU(\sl_2)_\loc$ obtained by horizontally composing the 1-morphisms $E 1_\mu$ and $F 1_\mu$ corresponding to the letters of $X$ and $Y$
for appropriate weights $\mu$.
Then we define 
\begin{equation}\label{ugly}
\Hom_{\cU(\sl_2;t)_\loc}(Y, X):=
\prod_{\lambda \in t+2\Z} 
\Hom_{\fU(\sl_2)_{\loc}}(Y 1_\lambda, X 1_\lambda)
\end{equation}
for $X, Y \in \langle E, F\rangle$.
Defining the {\em weight} $\wt(X)$ of $X \in \langle E,F\rangle$ to be 
$2\times ($the number of letters $E$ minus the number of letters $F$ in the word $X)$,
the morphism space \cref{ugly} is $0$ unless $\wt(X) = \wt(Y)$.
In general, $f\in \Hom_{\cU(\sl_2;t)_\loc}(Y, X)$
is a tuple $f=(f_\lambda)_{\lambda \in 1+2\Z}$
of morphisms $f_\lambda \in \Hom_{\fU_{q^{-1}}(\sl_2)_{\loc}}(Y 1_\lambda, X 1_\lambda)_\loc$.
The composition law making $\cU(\sl_2;t)_\loc$ into a category is induced by vertical composition in $\fU(\sl_2)_\loc$:
we have that $(g \circ f)_\lambda = g_\lambda \circ f_\lambda$
for morphisms 
$f:X \rightarrow Y$ and $g:Y \rightarrow Z$.
The strict monoidal product $-\star-:\cU(\sl_2;t)_\loc\boxtimes\cU(\sl_2;t)_\loc \rightarrow \cU(\sl_2;t)_\loc$ is induced by horizontal composition in $\fU(\sl_2)_\loc$:
it is defined on objects simply by concatenation of words and
on morphisms by setting 
$(f' \star f)_\lambda := f'_{\lambda+\wt(Y)}\, f_\lambda$
for $f:Y \rightarrow X$, $f':Y' \rightarrow X'$.
\end{definition}

Morphisms in the additive envelope
$\Add\left(\cU(\sl_2;t)_\loc
\right)$ are matrices of morphisms in
$\cU(\sl_2;t)_\loc$.
In the statement of the next theorem, we use some obvious shorthand
to represent such matrices.
For example, the morphism
$\Omega_t\left(\textcrossing\right)$ appearing below
is an endomorphism of 
$$
(E \oplus F)^{\star 2}
\cong E \star E\oplus E \star F \oplus F \star E \oplus F \star F,$$
so it is a $4 \times 4$ matrix with rows and columns indexed
by the words $E E, EF, F E$ and $FF$.
In turn, $\Omega_t\left(\textcrossing\right)_\lambda$
is a matrix 
representing an endomorphism of
 $E E 1_\lambda \oplus \hat EF 1_\lambda\oplus F E 1_\lambda \oplus 
 FF 1_\lambda$.
 The eight morphisms appearing on the right-hand side of the
 equation for $\Omega_t(\textcrossing)_\lambda$ in the statement of the theorem  are matrices of this form with 0 in all but the self-evident entry.

\begin{theorem}\label{psit}
There is a strict monoidal functor
$\Omega_t:\cNB_t \rightarrow \Add\left(\cU(\sl_2;t)_\loc\right)$
taking the generating object $B$ to $E \oplus F$, and defined 
on generating morphisms by letting
\begin{align*}
\Omega_t\left(\textdot\:\right)_\lambda &:=
\begin{tikzpicture}[anchorbase]
\draw[-to,thin] (0,-.4) to (0,.4);
\opendot{0,0};
\node at (0.25,0) {$\catlabel{\lambda}$};
\end{tikzpicture}
-\begin{tikzpicture}[anchorbase]
\draw[to-,thin] (0,-.4) to (0,.4);
\opendot{0,0};
\node at (0.25,0) {$\catlabel{\lambda}$};
\end{tikzpicture}
,\\
\Omega_t\left(\;\textcrossing\;\right)_\lambda &:=
\begin{tikzpicture}[anchorbase,scale=.8]
	\draw[-to,thin] (0.6,-.6) to (-0.6,.6);
	\draw[-to,thin] (-0.6,-.6) to (0.6,.6);
       \node at (0.5,0) {$\catlabel{\lambda}$};
\end{tikzpicture}
+
\begin{tikzpicture}[anchorbase,scale=.8]
	\draw[to-,thin] (0.6,-.6) to (-0.6,.6);
	\draw[to-,thin] (-0.6,-.6) to (0.6,.6);
       \node at (0.5,0) {$\catlabel{\lambda}$};
\end{tikzpicture}
+\begin{tikzpicture}[anchorbase,scale=.8]
	\draw[to-,thin] (0.6,-.6) to (-0.6,.6);
	\draw[-to,thin] (-0.6,-.6) to (0.6,.6);
       \node at (0.5,0) {$\catlabel{\lambda}$};
\end{tikzpicture}
+
\begin{tikzpicture}[anchorbase,scale=.8]
	\draw[-to,thin] (0.6,-.6) to (-0.6,.6);
	\draw[to-,thin] (-0.6,-.6) to (0.6,.6);
\draw[-to,thin,fill=black!10!white] (-.4,.1) arc(-135:225:0.2);
\draw[-to,thin,fill=black!10!white] (.4,-.1) arc(45:-315:0.2);
       \node at (0.7,0) {$\catlabel{\lambda}$};
\end{tikzpicture}
+\begin{tikzpicture}[anchorbase,scale=.8]
	\draw[-to,thin] (-0.7,-.6) to (-0.7,.6);
	\draw[to-,thin] (.1,-.6) to (.1,.6);
\limitteleporterOO{(-.7,0)}{}{}{(.1,0)};
            \node at (.4,0) {$\catlabel{\lambda}$};
\end{tikzpicture}
-\begin{tikzpicture}[anchorbase,scale=.8]
 	\draw[-to,thin] (-0.45,-.6) to[out=90,in=180] (-.1,-.1) to[out=0,in=90] (0.25,-.6);
 	\draw[to-,thin] (-0.45,.6) to[out=-90,in=180] (-.1,.1) to[out=0,in=-90] (0.25,.6);
     \bentlimitteleporterOO{(.17,.3)}{}{}{(.17,-.3)}{(.77,.1) and (.77,-.1)};
\draw[-to,thin,fill=black!10!white] (-.5,.13) arc(-150:210:0.2);
   \node at (.9,0) {$\catlabel{\lambda}$};
\end{tikzpicture}
-
\begin{tikzpicture}[anchorbase,scale=.8]
	\draw[to-,thin] (-0.7,-.6) to (-0.7,.6);
	\draw[-to,thin] (.1,-.6) to (.1,.6);
\limitteleporterOO{(-.7,0)}{}{}{(.1,0)};
            \node at (.4,0) {$\catlabel{\lambda}$};
\end{tikzpicture}
+\begin{tikzpicture}[anchorbase,scale=.8]
 	\draw[to-,thin] (-0.45,-.6) to[out=90,in=180] (-.1,-.1) to[out=0,in=90] (0.25,-.6);
 	\draw[-to,thin] (-0.45,.6) to[out=-90,in=180] (-.1,.1) to[out=0,in=-90] (0.25,.6);
     \bentlimitteleporterOO{(-.42,.3)}{}{}{(-.42,-.3)}{(-1.02,.1) and (-1.02,-.1)};
\draw[-to,thin,fill=black!10!white] (.32,-0.18) arc(30:-330:0.2);
   \node at (.7,0) {$\catlabel{\lambda}$};
\end{tikzpicture}\:,\\
\Omega_t\left(\:\,\txtcap\:\,\right)_\lambda &:=
\begin{tikzpicture}[anchorbase,scale=.8]
	\draw[-,thin] (-0.4,-0.3) to[out=90, in=180] (0,0.3);
	\draw[-to,thin] (-0,0.3) to[out = 0, in = 90] (0.4,-0.3);
\node at (.7,0) {$\catlabel{\lambda}$};
\end{tikzpicture}+
\begin{tikzpicture}[anchorbase,scale=.8]
	\draw[to-,thin] (-0.4,-0.3) to[out=90, in=180] (0,0.3);
	\draw[-,thin] (-0,0.3) to[out = 0, in = 90] (0.4,-0.3);
\draw[-to,thin,fill=black!10!white] (.4,0.2) arc(30:-330:0.2);
\node at (.75,0) {$\catlabel{\lambda}$};
\end{tikzpicture}
,\\
\Omega_t\left(\:\,\txtcup\,\:\right)_\lambda &:=\begin{tikzpicture}[anchorbase,scale=.8] 
	\draw[-,thin] (-0.4,0.3) to[out=-90, in=180] (0,-0.3);
	\draw[-to,thin] (-0,-0.3) to[out = 0, in = -90] (0.4,0.3);
      \node at (.65,0) {$\catlabel{\lambda}$};
\end{tikzpicture}+
\begin{tikzpicture}[anchorbase,scale=.8]
	\draw[to-,thin] (-0.4,0.3) to[out=-90, in=180] (0,-0.3);
	\draw[-,thin] (-0,-0.3) to[out = 0, in = -90] (0.4,0.3);
\draw[-to,thin,fill=black!10!white] (-.4,-.2) arc(-150:210:0.2);
\node at (.65,0) {$\catlabel{\lambda}$};
\end{tikzpicture}
\end{align*}
for $\lambda \in t+2\Z$.
We also have that
% IN THE PUBLISHED VERSION THIS WAS WRONG!
%\begin{equation}\label{magic}
%\Omega_t\left(\O(u)\right)_\lambda
%={\catlabel{\lambda}\:}
%\textanticlockwisebubble(-u) \:\textclockwisebubble(u).
%\end{equation}
% The correct equation was proved in the proof, it is
% just the statement that was messed up.
\begin{equation}\label{magic}
\Omega_t\left(\O(u)\right)_\lambda
=(-1)^t {\catlabel{\lambda}\:}
\textanticlockwisebubble(u) \:\textclockwisebubble(-u).
\end{equation}
\end{theorem}

\begin{proof}
To prove the existence of $\Omega_t$,
we simply need to check the eight defining
relations from \cref{rels1,rels2,rels3,rels4}!
\begin{itemize}
\item
Consider the first relation from \cref{rels2}. We have that
$$
\Omega_t\left(\:\begin{tikzpicture}[baseline=-2.5mm]
\draw (0,-.15) circle (.3);
\end{tikzpicture}
\:\right)_\lambda = 
\begin{tikzpicture}[anchorbase,scale=.8]
\draw[-to,thin] (.3,-.4) arc(0:-360:.4);
\draw[-to,thin,fill=black!10!white] (-.65,-.4) arc(180:540:0.2);
\node at (.6,-.3) {$\catlabel{\lambda}$};
\end{tikzpicture}
+
\begin{tikzpicture}[anchorbase,scale=.8]
\node at (1.55,-.3) {$\catlabel{\lambda}$};
\draw[-to,thin] (.3,-.4) arc(-180:180:.4);
\draw[-to,thin,fill=black!10!white] (1.25,-.4) arc(0:-360:0.2);
\end{tikzpicture}\:.
$$
To check the relation, we must show that this equals
$\Omega_t(t 1_\one)_\lambda = t 1_{1_\lambda}$,
which follows immediately from \cref{sunriver}.
\item
For the second relation from \cref{rels2}, we 
have that
\begin{align*}
\Omega_t\left(\:\begin{tikzpicture}[anchorbase]
  \draw[-] (0.3,0) to (0.3,.4);
	\draw[-] (0.3,0) to[out=-90, in=0] (0.1,-0.4);
	\draw[-] (0.1,-0.4) to[out = 180, in = -90] (-0.1,0);
	\draw[-] (-0.1,0) to[out=90, in=0] (-0.3,0.4);
	\draw[-] (-0.3,0.4) to[out = 180, in =90] (-0.5,0);
  \draw[-] (-0.5,0) to (-0.5,-.4);
\end{tikzpicture}
\:\right)_\lambda&=\begin{tikzpicture}[anchorbase]
  \draw[-to,thin] (0.3,0) to (0.3,.4);
	\draw[-,thin] (0.3,0) to[out=-90, in=0] (0.1,-0.4);
	\draw[-,thin] (0.1,-0.4) to[out = 180, in = -90] (-0.1,0);
	\draw[-,thin] (-0.1,0) to[out=90, in=0] (-0.3,0.4);
	\draw[-,thin] (-0.3,0.4) to[out = 180, in =90] (-0.5,0);
  \draw[-,thin] (-0.5,0) to (-0.5,-.4);
\node at (.5,0) {$\catlabel{\lambda}$};
\end{tikzpicture}
+\begin{tikzpicture}[anchorbase]
  \draw[-,thin] (0.3,0) to (0.3,.4);
	\draw[-,thin] (0.3,0) to[out=-90, in=0] (0.1,-0.4);
	\draw[-,thin] (0.1,-0.4) to[out = 180, in = -90] (-0.1,0);
	\draw[-,thin] (-0.1,0) to[out=90, in=0] (-0.3,0.4);
	\draw[-,thin] (-0.3,0.4) to[out = 180, in =90] (-0.5,0);
  \draw[-to,thin] (-0.5,0) to (-0.5,-.4);
\draw[-to,thin,fill=black!10!white] (0,.2) arc(0:-360:0.15);
\draw[-to,thin,fill=black!10!white] (-.2,-.2) arc(-180:180:0.15);
\node at (.75,0) {$\catlabel{\lambda}$};
\end{tikzpicture}
\:,\\
\Omega_t\left(\:\:
\begin{tikzpicture}[anchorbase]
  \draw[-] (0,-0.4) to (0,.4);
\end{tikzpicture}\:\:\right)_\lambda
&=\:\begin{tikzpicture}[anchorbase]
  \draw[-to,thin] (0,-0.4) to (0,.4);
\node at (.2,0) {$\catlabel{\lambda}$};
\end{tikzpicture}
+\begin{tikzpicture}[anchorbase]
  \draw[to-,thin] (0,-0.4) to (0,.4);
\node at (.2,0) {$\catlabel{\lambda}$};
\end{tikzpicture}
\:,\\
\Omega_t\left(\:\begin{tikzpicture}[anchorbase]
  \draw[-] (0.3,0) to (0.3,-.4);
	\draw[-] (0.3,0) to[out=90, in=0] (0.1,0.4);
	\draw[-] (0.1,0.4) to[out = 180, in = 90] (-0.1,0);
	\draw[-] (-0.1,0) to[out=-90, in=0] (-0.3,-0.4);
	\draw[-] (-0.3,-0.4) to[out = 180, in =-90] (-0.5,0);
  \draw[-] (-0.5,0) to (-0.5,.4);
\end{tikzpicture}\:\right)_\lambda &=\begin{tikzpicture}[anchorbase]
  \draw[-,thin] (0.3,0) to (0.3,-.4);
	\draw[-,thin] (0.3,0) to[out=90, in=0] (0.1,0.4);
	\draw[-,thin] (0.1,0.4) to[out = 180, in = 90] (-0.1,0);
	\draw[-,thin] (-0.1,0) to[out=-90, in=0] (-0.3,-0.4);
	\draw[-,thin] (-0.3,-0.4) to[out = 180, in =-90] (-0.5,0);
  \draw[-to,thin] (-0.5,0) to (-0.5,.4);
\draw[-to,thin,fill=black!10!white] (-.65,0) arc(180:540:0.15);
\draw[-to,thin,fill=black!10!white] (.45,0) arc(0:-360:0.15);
\node at (.75,0) {$\catlabel{\lambda}$};
\end{tikzpicture}+\begin{tikzpicture}[anchorbase]
  \draw[-to,thin] (0.3,0) to (0.3,-.4);
	\draw[-,thin] (0.3,0) to[out=90, in=0] (0.1,0.4);
	\draw[-,thin] (0.1,0.4) to[out = 180, in = 90] (-0.1,0);
	\draw[-,thin] (-0.1,0) to[out=-90, in=0] (-0.3,-0.4);
	\draw[-,thin] (-0.3,-0.4) to[out = 180, in =-90] (-0.5,0);
  \draw[-,thin] (-0.5,0) to (-0.5,.4);
\node at (.5,0) {$\catlabel{\lambda}$};
\end{tikzpicture}\:.
\end{align*}
Thus, to check this relation, we need to show that
\begin{align*}
\begin{tikzpicture}[anchorbase]
  \draw[-to,thin] (0.3,0) to (0.3,.4);
	\draw[-,thin] (0.3,0) to[out=-90, in=0] (0.1,-0.4);
	\draw[-,thin] (0.1,-0.4) to[out = 180, in = -90] (-0.1,0);
	\draw[-,thin] (-0.1,0) to[out=90, in=0] (-0.3,0.4);
	\draw[-,thin] (-0.3,0.4) to[out = 180, in =90] (-0.5,0);
  \draw[-,thin] (-0.5,0) to (-0.5,-.4);
\node at (.5,0) {$\catlabel{\lambda}$};
\end{tikzpicture}
&=
\:\begin{tikzpicture}[anchorbase]
  \draw[-to,thin] (0,-0.4) to (0,.4);
\node at (.2,0) {$\catlabel{\lambda}$};
\end{tikzpicture}
=\begin{tikzpicture}[anchorbase]
  \draw[-,thin] (0.3,0) to (0.3,-.4);
	\draw[-,thin] (0.3,0) to[out=90, in=0] (0.1,0.4);
	\draw[-,thin] (0.1,0.4) to[out = 180, in = 90] (-0.1,0);
	\draw[-,thin] (-0.1,0) to[out=-90, in=0] (-0.3,-0.4);
	\draw[-,thin] (-0.3,-0.4) to[out = 180, in =-90] (-0.5,0);
  \draw[-to,thin] (-0.5,0) to (-0.5,.4);
\draw[-to,thin,fill=black!10!white] (-.65,0) arc(180:540:0.15);
\draw[-to,thin,fill=black!10!white] (.45,0) arc(0:-360:0.15);
\node at (.75,0) {$\catlabel{\lambda}$};
\end{tikzpicture}\:,&
\begin{tikzpicture}[anchorbase]
  \draw[-,thin] (0.3,0) to (0.3,.4);
	\draw[-,thin] (0.3,0) to[out=-90, in=0] (0.1,-0.4);
	\draw[-,thin] (0.1,-0.4) to[out = 180, in = -90] (-0.1,0);
	\draw[-,thin] (-0.1,0) to[out=90, in=0] (-0.3,0.4);
	\draw[-,thin] (-0.3,0.4) to[out = 180, in =90] (-0.5,0);
  \draw[-to,thin] (-0.5,0) to (-0.5,-.4);
\draw[-to,thin,fill=black!10!white] (0,.2) arc(0:-360:0.15);
\draw[-to,thin,fill=black!10!white] (-.2,-.2) arc(-180:180:0.15);
\node at (.75,0) {$\catlabel{\lambda}$};
\end{tikzpicture}
&=\:
\begin{tikzpicture}[anchorbase]
  \draw[to-,thin] (0,-0.4) to (0,.4);
\node at (.2,0) {$\catlabel{\lambda}$};
\end{tikzpicture}
=\begin{tikzpicture}[anchorbase]
  \draw[-to,thin] (0.3,0) to (0.3,-.4);
	\draw[-,thin] (0.3,0) to[out=90, in=0] (0.1,0.4);
	\draw[-,thin] (0.1,0.4) to[out = 180, in = 90] (-0.1,0);
	\draw[-,thin] (-0.1,0) to[out=-90, in=0] (-0.3,-0.4);
	\draw[-,thin] (-0.3,-0.4) to[out = 180, in =-90] (-0.5,0);
  \draw[-,thin] (-0.5,0) to (-0.5,.4);
\node at (.5,0) {$\catlabel{\lambda}$};
\end{tikzpicture}\:.
\end{align*}
These both follow easily using \cref{KMrels2}, \cref{otheradjunction,stpauls}.
\item
In a similar way, the first relation from \cref{rels4} reduces to checking that
\begin{align*}
\begin{tikzpicture}[anchorbase,scale=1.2]
 	\draw[to-,thin] (0.35,.55) to (-0.35,-.15);
	\draw[-to,thin] (0.35,-.15) to (-0.35,.55);
  \node at (.4,.25) {$\catlabel{\lambda}$};
      \opendot{-0.18,0.38};
      \end{tikzpicture}
-
\begin{tikzpicture}[anchorbase,scale=1.2]
	\draw[to-,thin] (0.35,.55) to (-0.35,-.15);
	\draw[-to,thin] (0.35,-.15) to (-0.35,.55);
  \node at (.4,.25) {$\catlabel{\lambda}$};
      \opendot{0.18,0.01};
      \end{tikzpicture}
&=
\begin{tikzpicture}[anchorbase,scale=1.2]
 	\draw[-to,thin] (0.08,-.3) to (0.08,.4);
	\draw[-to,thin] (-0.28,-.3) to (-0.28,.4);
 \node at (.28,.06) {$\catlabel{\lambda}$};
\end{tikzpicture}\:,&
-\begin{tikzpicture}[anchorbase,scale=1.2]
 	\draw[-to,thin] (0.35,.55) to (-0.35,-.15);
	\draw[to-,thin] (0.35,-.15) to (-0.35,.55);
  \node at (.4,.25) {$\catlabel{\lambda}$};
      \opendot{-0.18,0.38};
      \end{tikzpicture}
+
\begin{tikzpicture}[anchorbase,scale=1.2]
	\draw[-to,thin] (0.35,.55) to (-0.35,-.15);
	\draw[to-,thin] (0.35,-.15) to (-0.35,.55);
  \node at (.4,.25) {$\catlabel{\lambda}$};
      \opendot{0.18,0.01};
      \end{tikzpicture}
&=
\begin{tikzpicture}[anchorbase,scale=1.2]
 	\draw[to-,thin] (0.08,-.3) to (0.08,.4);
	\draw[to-,thin] (-0.28,-.3) to (-0.28,.4);
 \node at (.28,.06) {$\catlabel{\lambda}$};
\end{tikzpicture}\:,\\
-\begin{tikzpicture}[anchorbase,scale=1.4]
 	\draw[to-,thin] (0.35,.55) to (-0.35,-.15);
	\draw[to-,thin] (0.35,-.15) to (-0.35,.55);
  \node at (.4,.25) {$\catlabel{\lambda}$};
      \opendot{-0.18,0.38};
      \end{tikzpicture}
+
\begin{tikzpicture}[anchorbase,scale=1.4]
	\draw[to-,thin] (0.35,.55) to (-0.35,-.15);
	\draw[to-,thin] (0.35,-.15) to (-0.35,.55);
  \node at (.4,.25) {$\catlabel{\lambda}$};
      \opendot{0.18,0.01};
      \end{tikzpicture}
&=-\begin{tikzpicture}[anchorbase,scale=1.4]
 	\draw[-to,thin] (-0.25,-.35) to[out=90,in=180] (0,-.05) to[out=0,in=90] (0.25,-.35);
 	\draw[-to,thin] (-0.25,.35) to[out=-90,in=180] (0,.05) to[out=0,in=-90] (0.25,.35);
   \node at (.5,0) {$\catlabel{\lambda}$};
\end{tikzpicture}\:,&
\begin{tikzpicture}[anchorbase,scale=.8]
	\draw[-to,thin] (0.6,.6) to (-0.6,-.6);
	\draw[-to,thin] (0.6,-.6) to (-0.6,.6);
  \node at (.7,0) {$\catlabel{\lambda}$};
      \opendot{-0.45,.45};
\draw[-to,thin,fill=black!10!white] (-.35,.05) arc(-135:225:0.2);
\draw[-to,thin,fill=black!10!white] (.35,-0.05) arc(45:-315:0.2);
      \end{tikzpicture}
-
\begin{tikzpicture}[anchorbase,scale=.8]
	\draw[-to,thin] (0.6,.6) to (-0.6,-.6);
	\draw[-to,thin] (0.6,-.6) to (-0.6,.6);
  \node at (.7,0) {$\catlabel{\lambda}$};
      \opendot{0.45,-.45};
\draw[-to,thin,fill=black!10!white] (-.35,.05) arc(-135:225:0.2);
\draw[-to,thin,fill=black!10!white] (.35,-0.05) arc(45:-315:0.2);
      \end{tikzpicture}
&=-\begin{tikzpicture}[anchorbase,scale=.8]
 	\draw[to-,thin] (-0.45,-.6) to[out=90,in=180] (-.1,-.1) to[out=0,in=90] (0.25,-.6);
 	\draw[to-,thin] (-0.45,.6) to[out=-90,in=180] (-.1,.1) to[out=0,in=-90] (0.25,.6);
\draw[-to,thin,fill=black!10!white] (-.45,.1) arc(-135:225:0.2);
\draw[-to,thin,fill=black!10!white] (.25,-0.1) arc(45:-315:0.2);
   \node at (.9,0) {$\catlabel{\lambda}$};
\end{tikzpicture}\:,
\end{align*}and that\begin{align*}
\begin{tikzpicture}[anchorbase,scale=1]
	\draw[-to,thin] (-0.7,-.6) to (-0.7,.6);
	\draw[to-,thin] (.1,-.6) to (.1,.6);
\limitteleporterOO{(-.7,0)}{}{}{(.1,0)};
\opendot{-.7,.3};
            \node at (.4,0) {$\catlabel{\lambda}$};
\end{tikzpicture}
-\begin{tikzpicture}[anchorbase,scale=1]
 	\draw[-to,thin] (-0.45,-.6) to[out=90,in=180] (-.1,-.1) to[out=0,in=90] (0.25,-.6);
 	\draw[to-,thin] (-0.45,.6) to[out=-90,in=180] (-.1,.1) to[out=0,in=-90] (0.25,.6);
 	     \bentlimitteleporterOO{(.19,.3)}{}{}{(.19,-.3)}{(.77,.1) and (.77,-.1)};
\draw[-to,thin,fill=black!10!white] (-.38,.08) arc(-135:225:0.15);
\opendot{-.42,.42};
   \node at (.85,0) {$\catlabel{\lambda}$};
\end{tikzpicture}
+\begin{tikzpicture}[anchorbase,scale=1]
	\draw[-to,thin] (-0.7,-.6) to (-0.7,.6);
	\draw[to-,thin] (.1,-.6) to (.1,.6);
\limitteleporterOO{(-.7,0)}{}{}{(.1,0)};
            \node at (.4,0) {$\catlabel{\lambda}$};
\opendot{.1,-.3};
\end{tikzpicture}
-\begin{tikzpicture}[anchorbase,scale=1]
 	\draw[-to,thin] (-0.45,-.6) to[out=90,in=180] (-.1,-.1) to[out=0,in=90] (0.25,-.6);
 	\draw[to-,thin] (-0.45,.6) to[out=-90,in=180] (-.1,.1) to[out=0,in=-90] (0.25,.6);
 	     \bentlimitteleporterOO{(.19,.3)}{}{}{(.13,-.22)}{(.77,.1) and (.77,-.1)};
\draw[-to,thin,fill=black!10!white] (-.38,.08) arc(-135:225:0.15);
\opendot{.24,-.4};
   \node at (.85,0) {$\catlabel{\lambda}$};
\end{tikzpicture}
&=
\begin{tikzpicture}[anchorbase,scale=1]
	\draw[-to,thin] (-0.6,-.6) to (-0.6,.6);
	\draw[to-,thin] (.1,-.6) to (.1,.6);
            \node at (.4,0) {$\catlabel{\lambda}$};
\end{tikzpicture}
-\begin{tikzpicture}[anchorbase,scale=1]
 	\draw[-to,thin] (-0.45,-.6) to[out=90,in=180] (-.1,-.1) to[out=0,in=90] (0.25,-.6);
 	\draw[to-,thin] (-0.45,.6) to[out=-90,in=180] (-.1,.1) to[out=0,in=-90] (0.25,.6);
\draw[-to,thin,fill=black!10!white] (-.38,.08) arc(-135:225:0.15);
   \node at (.85,0) {$\catlabel{\lambda}$};
\end{tikzpicture}
\:,\\
\begin{tikzpicture}[anchorbase,scale=1]
	\draw[to-,thin] (-0.7,-.6) to (-0.7,.6);
	\draw[-to,thin] (.1,-.6) to (.1,.6);
\limitteleporterOO{(-.7,0)}{}{}{(.1,0)};
\opendot{-.7,.3};
            \node at (.4,0) {$\catlabel{\lambda}$};
\end{tikzpicture}
-\begin{tikzpicture}[anchorbase,scale=1]
 	\draw[to-,thin] (-0.45,-.6) to[out=90,in=180] (-.1,-.1) to[out=0,in=90] (0.25,-.6);
 	\draw[-to,thin] (-0.45,.6) to[out=-90,in=180] (-.1,.1) to[out=0,in=-90] (0.25,.6);
     \bentlimitteleporterOO{(-.38,.3)}{}{}{(-.39,-.3)}{(-.8,.1) and (-.8,-.1)};
\opendot{-.44,.47};
\draw[-to,thin,fill=black!10!white] (.25,-0.1) arc(45:-315:0.15);
   \node at (.9,0) {$\catlabel{\lambda}$};
\end{tikzpicture}
+
\begin{tikzpicture}[anchorbase,scale=1]
	\draw[to-,thin] (-0.7,-.6) to (-0.7,.6);
	\draw[-to,thin] (.1,-.6) to (.1,.6);
\limitteleporterOO{(-.7,0)}{}{}{(.1,0)};
\opendot{.1,-.3};
            \node at (.4,0) {$\catlabel{\lambda}$};
\end{tikzpicture}
-\begin{tikzpicture}[anchorbase,scale=1]
 	\draw[to-,thin] (-0.45,-.6) to[out=90,in=180] (-.1,-.1) to[out=0,in=90] (0.25,-.6);
 	\draw[-to,thin] (-0.45,.6) to[out=-90,in=180] (-.1,.1) to[out=0,in=-90] (0.25,.6);
     \bentlimitteleporterOO{(-.39,.3)}{}{}{(-.39,-.3)}{(-.8,.1) and (-.8,-.1)};
\opendot{.24,-.48};
\draw[-to,thin,fill=black!10!white] (.25,-0.1) arc(45:-315:0.15);
   \node at (.9,0) {$\catlabel{\lambda}$};
\end{tikzpicture}
&=
\begin{tikzpicture}[anchorbase,scale=1]
	\draw[to-,thin] (-0.6,-.6) to (-0.6,.6);
	\draw[-to,thin] (.1,-.6) to (.1,.6);
            \node at (.4,0) {$\catlabel{\lambda}$};
\end{tikzpicture}
-\begin{tikzpicture}[anchorbase,scale=1]
 	\draw[to-,thin] (-0.45,-.6) to[out=90,in=180] (-.1,-.1) to[out=0,in=90] (0.25,-.6);
 	\draw[-to,thin] (-0.45,.6) to[out=-90,in=180] (-.1,.1) to[out=0,in=-90] (0.25,.6);
\draw[-to,thin,fill=black!10!white] (.25,-0.1) arc(45:-315:0.15);
   \node at (.9,0) {$\catlabel{\lambda}$};
\end{tikzpicture}\:.
\end{align*}
The first four of these follow immediately from the last defining relation in
\cref{KMrels1} plus its variants obtained by rotating through $90^\circ, 180^\circ$ and $270^\circ$. The last two follow using \cref{dumby0}.
\item
Next, we look at the second relation from \cref{rels4}.
For this, we must show that
\begin{align*}
\begin{tikzpicture}[anchorbase,scale=.9]
	\draw[to-,thin] (0.4,0) to[out=90, in=0] (0,0.6);
	\draw[-,thin] (0,0.6) to[out = 180, in = 90] (-0.4,0);
 \opendot{.37,.3};
\node at (.6,0.3) {$\catlabel{\lambda}$};
\end{tikzpicture}
&=
\begin{tikzpicture}[anchorbase,scale=.9]
	\draw[to-,thin] (0.4,0) to[out=90, in=0] (0,0.6);
	\draw[-,thin] (0,0.6) to[out = 180, in = 90] (-0.4,0);
 \opendot{-.36,.3};
\node at (.6,0.3) {$\catlabel{\lambda}$};
\end{tikzpicture}\:,&
\begin{tikzpicture}[anchorbase,scale=.9]
	\draw[-,thin] (0.4,0) to[out=90, in=0] (0,0.6);
	\draw[-to,thin] (0,0.6) to[out = 180, in = 90] (-0.4,0);
\draw[-to,thin,fill=black!10!white] (.35,.63) arc(60:-300:0.2);
 \opendot{.38,.15};
\node at (.6,0.3) {$\catlabel{\lambda}$};
\end{tikzpicture}
&=
\begin{tikzpicture}[anchorbase,scale=.9]
	\draw[-,thin] (0.4,0) to[out=90, in=0] (0,0.6);
	\draw[-to,thin] (0,0.6) to[out = 180, in = 90] (-0.4,0);
\draw[-to,thin,fill=black!10!white] (.35,.63) arc(60:-300:0.2);
 \opendot{-.36,.3};
\node at (.7,0.3) {$\catlabel{\lambda}$};
\end{tikzpicture}\:.
\end{align*}
These both follow using planar isotopy
since open dots commute with internal bubbles.
\item
The first relation in \cref{rels3} follows using \cref{stpauls} and
the relations
\begin{align*}
\begin{tikzpicture}[anchorbase]
	\draw[to-,thin] (-0.35,.3)  to [out=90,in=-90] (.15,.9) to [out=90,in=0] (-0.1,1.2);
 \draw[-,thin] (.15,.3)  to [out=90,in=-90] (-.35,.9) to [out=90,in=180] (-.1,1.2);
\draw[to-,thin,fill=black!10!white] (-.5,1) arc(180:-180:0.16);
\node at (.3,.8) {$\catlabel{\lambda}$};
\end{tikzpicture}
+\begin{tikzpicture}[anchorbase,scale=1.2]
	\draw[-,thin] (0.3,-0.3) to[out=90, in=0] (0,0.4);
	\draw[-to,thin] (0,0.4) to[out = 180, in = 90] (-0.3,-0.3);
\draw[-to,thin] (.87,.26) arc(-225:235:0.2);
\draw[to-,thin,fill=black!10!white] (1.3,.12) arc(0:360:0.13);
\limitteleporterOO{(.81,.12)}{}{}{(.28,.12)};
\node at (1.5,0) {$\catlabel{\lambda}$};
\end{tikzpicture}
-{\textstyle\frac{1}{2}}\:\begin{tikzpicture}[anchorbase,scale=1.2]
	\draw[-,thin] (0.3,-0.3) to[out=90, in=0] (0,0.4);
	\draw[-to,thin] (0,0.4) to[out = 180, in = 90] (-0.3,-0.3);
\opendot{.27,.1};\node at (.5,.1) {$\kmlabel{-1}$};
\node at (.9,0) {$\catlabel{\lambda}$};
\end{tikzpicture}&=0,\\
\begin{tikzpicture}[anchorbase]
	\draw[to-,thin] (0.35,.3)  to [out=90,in=-90] (-.15,.9) to [out=90,in=180] (0.1,1.2);
 \draw[-,thin] (-.15,.3)  to [out=90,in=-90] (.35,.9) to [out=90,in=0] (.1,1.2);
\draw[to-,thin,fill=black!10!white] (.5,1) arc(0:360:0.16);
\node at (.6,.7) {$\catlabel{\lambda}$};
\end{tikzpicture}
-\begin{tikzpicture}[anchorbase,scale=1.2]
	\draw[-,thin] (-0.3,-0.3) to[out=90, in=180] (0,0.4);
	\draw[-to,thin] (0,0.4) to[out = 0, in = 90] (0.3,-0.3);
\draw[to-,thin] (-.87,.26) arc(45:405:0.2);
\draw[to-,thin,fill=black!10!white] (-1.3,.12) arc(180:-180:0.13);
\limitteleporterOO{(-.81,.12)}{}{}{(-.28,.12)};
\node at (.5,0) {$\catlabel{\lambda}$};
\end{tikzpicture}
+{\textstyle\frac{1}{2}}\:\begin{tikzpicture}[anchorbase,scale=1.2]
	\draw[-,thin] (-0.3,-0.3) to[out=90, in=180] (0,0.4);
	\draw[-to,thin] (0,0.4) to[out = 0, in = 90] (0.3,-0.3);
\opendot{-.27,.1};\node at (-.55,.1) {$\kmlabel{-1}$};
\node at (.5,0) {$\catlabel{\lambda}$};
\end{tikzpicture}&=0.
\end{align*}
The first of these is \cref{werehere}, and the second follows
from the first (with $\lambda$ replaced by $-\lambda$)
on applying $\tR$.
\item The second relation from \cref{rels3} requires the
following four identities to be checked:
\begin{align*}
\begin{tikzpicture}[anchorbase,scale=1.5]
	\draw[to-,thin] (0.5,0) to[out=90, in=0] (0.1,0.6);
	\draw[-,thin] (0.1,0.6) to[out = 180, in = 90] (-0.3,0);
 \draw[-to,thin] (0.1,0) to[out=90,in=-90] (-.35,.7);
\node at (.6,.4) {$\catlabel{\lambda}$};
\end{tikzpicture}
&=
\begin{tikzpicture}[anchorbase,scale=1.5]
	\draw[to-,thin] (0.5,0) to[out=90, in=0] (0.1,0.6);
	\draw[-,thin] (0.1,0.6) to[out = 180, in = 90] (-0.3,0);
 \draw[-to,thin] (0.1,0) to[out=90,in=-90] (.55,.7);
\node at (.6,.4) {$\catlabel{\lambda}$};
\end{tikzpicture}\:,&
\begin{tikzpicture}[anchorbase,scale=1.5]
	\draw[to-,thin] (0.5,0) to[out=90, in=0] (0.1,0.5);
	\draw[-,thin] (0.1,0.5) to[out = 180, in = 90] (-0.3,0);
 \draw[to-,thin] (0.1,0) to[out=90,in=-90] (-.3,.7);
\node at (.6,.4) {$\catlabel{\lambda}$};
\end{tikzpicture}
&=
\begin{tikzpicture}[anchorbase,scale=1.5]
	\draw[to-,thin] (0.5,0) to[out=90, in=0] (0.1,0.5);
	\draw[-,thin] (0.1,0.5) to[out = 180, in = 90] (-0.3,0);
 \draw[to-,thin] (0.1,0) to[out=90,in=-90] (.5,.7);
\node at (.6,.4) {$\catlabel{\lambda}$};
\end{tikzpicture}\:,\\
\begin{tikzpicture}[anchorbase,scale=1.5]
	\draw[-,thin] (0.5,0) to[out=90, in=0] (0.1,0.7);
	\draw[-to,thin] (0.1,0.7) to[out = 180, in = 90] (-0.3,0);
 \draw[-to,thin] (0.1,0) to[out=90,in=-90] (-.4,.8);
\draw[-to,thin,fill=black!10!white] (-.43,0.58) arc(-170:190:0.1);
\draw[-to,thin,fill=black!10!white] (0.02,0.35) arc(10:-350:0.1);
\draw[-to,thin,fill=black!10!white] (.55,0.4) arc(0:-360:0.1);
\node at (.8,.4) {$\catlabel{\lambda}$};
\end{tikzpicture}
&=
\begin{tikzpicture}[anchorbase,scale=1.5]
	\draw[-,thin] (0.5,0) to[out=90, in=0] (0.1,0.7);
	\draw[-to,thin] (0.1,0.7) to[out = 180, in = 90] (-0.3,0);
 \draw[-to,thin] (0.1,0) to[out=90,in=-90] (.6,.8);
\draw[-to,thin,fill=black!10!white] (.38,0.66) arc(30:-330:0.1);
\node at (.6,.4) {$\catlabel{\lambda}$};
\end{tikzpicture}\:,&
\begin{tikzpicture}[anchorbase,scale=1.5]
	\draw[-,thin] (0.5,0) to[out=90, in=0] (0.1,0.7);
	\draw[-to,thin] (0.1,0.7) to[out = 180, in = 90] (-0.3,0);
 \draw[to-,thin] (0.1,0) to[out=90,in=-90] (-.4,.8);
\draw[-to,thin,fill=black!10!white] (.55,0.3) arc(0:-360:0.1);
\node at (.6,.6) {$\catlabel{\lambda}$};
\end{tikzpicture}
&=
\begin{tikzpicture}[anchorbase,scale=1.5]
	\draw[-,thin] (0.5,0) to[out=90, in=0] (0.1,0.7);
	\draw[-to,thin] (0.1,0.7) to[out = 180, in = 90] (-0.3,0);
 \draw[to-,thin] (0.1,0) to[out=70,in=-90] (.7,.8);
\draw[-to,thin,fill=black!10!white] (.2,.78) arc(70:-290:0.1);
\draw[-to,thin,fill=black!10!white] (.28,0.5) arc(-170:190:0.1);
\draw[-to,thin,fill=black!10!white] (0.6,0.25) arc(10:-350:0.1);
\node at (.8,.4) {$\catlabel{\lambda}$};
\end{tikzpicture}\:.
\end{align*}
The first two of these follow immediately using planar isotopy.
The last one follows using planar isotopy and \cref{stpauls}.
For the third, one also needs to use \cref{ernie} (rotated counterclockwise by $90^\circ$)
to commute the two clockwise internal bubbles past the crossing to their left.
\item The first relation from \cref{rels1}
involves six non-trivial relations,
coming from the
$(EE 1_\lambda,EE 1_\lambda)$-
$(EF 1_\lambda, FE 1_\lambda)$-,
$(FE 1_\lambda,FE 1_\lambda)$-,
$(FF 1_\lambda,FF 1_\lambda)$-,
$(FE 1_\lambda, EF 1_\lambda)$- and
$(EF 1_\lambda, EF 1_\lambda)$-entries of the corresponding $4 \times 4$ matrices. 
However, after applying \cref{stpauls} to redistribute some internal bubbles, the last three may be deduced from images of the first three under $\tR$ and $\tT$. 
Thus, we really only have to verify three relations, which are as follows:
\begin{align*}
\begin{tikzpicture}[anchorbase,scale=.88]
	\draw[-to,thin] (0.28,0) to[out=90,in=-90] (-0.28,.75);
	\draw[-to,thin] (-0.28,0) to[out=90,in=-90] (0.28,.75);
	\draw[-,thin] (0.28,-.75) to[out=90,in=-90] (-0.28,0);
	\draw[-,thin] (-0.28,-.75) to[out=90,in=-90] (0.28,0);
   \node at (.5,0) {$\catlabel{\lambda}$};
\end{tikzpicture}
=
0,
\qquad\qquad\qquad
\begin{tikzpicture}[anchorbase,scale=.88]
	\draw[-to,thin] (0.4,-.75) to (-0.4,.75);
	\draw[to-,thin] (-0.4,-.75) to (0.4,.75);
	\limitteleporterOO{(-.25,.48)}{}{}{(.25,.48)};
       \node at (0.5,0) {$\catlabel{\lambda}$};
\end{tikzpicture}
-
\begin{tikzpicture}[anchorbase,scale=.88]
	\draw[-to,thin] (0.4,-.75) to (-0.4,.75);
	\draw[to-,thin] (-0.4,-.75) to (0.4,.75);
	\limitteleporterOO{(-.25,-.48)}{}{}{(.25,-.48)};
       \node at (0.5,0) {$\catlabel{\lambda}$};
\end{tikzpicture}
+
\begin{tikzpicture}[anchorbase,scale=.88]
\draw[to-,thin] (-.35,-1.8) to [out=90,in=180] (-.1,-1.4) to [out=0,in=90] (.15,-1.8);
	\draw[to-,thin] (-0.35,-.3)  to [out=-90,in=90] (.15,-.9) to [out=-90,in=0] (-0.1,-1.2);
 \draw[-,thin] (.15,-.3)  to [out=-90,in=90] (-.35,-.9) to [out=-90,in=180] (-.1,-1.2);
\bentlimitteleporterOO{(-.33,-1.56)}{}{}{(-.33,-1.04)}{(-.84,-1.36) and (-.84,-1.214)};
\draw[to-,thin,fill=black!10!white] (.3,-.9) arc(0:360:0.16);
\node at (.6,-1) {$\catlabel{\lambda}$};
\end{tikzpicture}
-
\begin{tikzpicture}[anchorbase,scale=.88]
\draw[to-,thin] (-.35,1.8) to [out=-90,in=180] (-.1,1.4) to [out=0,in=-90] (.15,1.8);
	\draw[to-,thin] (-0.35,.3)  to [out=90,in=-90] (.15,.9) to [out=90,in=0] (-0.1,1.2);
 \draw[-,thin] (.15,.3)  to [out=90,in=-90] (-.35,.9) to [out=90,in=180] (-.1,1.2);
\bentlimitteleporterOO{(.12,1.56)}{}{}{(.12,1.04)}{(.5,1.36) and (.5,1.214)};
\draw[to-,thin,fill=black!10!white] (-.5,1) arc(180:-180:0.16);
\node at (.4,.8) {$\catlabel{\lambda}$};
\end{tikzpicture}&=0,
\\
\begin{tikzpicture}[anchorbase,scale=1]
	\draw[-,thin] (0.28,0) to[out=90,in=-90] (-0.28,.6);
	\draw[-to,thin] (-0.28,0) to[out=90,in=-90] (0.28,.6);
	\draw[-,thin] (0.28,-.6) to[out=90,in=-90] (-0.28,0);
	\draw[to-,thin] (-0.28,-.6) to[out=90,in=-90] (0.28,0);
  \draw[-to,thin,fill=black!10!white] (-.55,0)++(.16,0) arc(-180:180:0.16);
  \node at (.6,0) {$\catlabel{\lambda}$};
\end{tikzpicture}
+\:\begin{tikzpicture}[anchorbase]
\draw[to-,thin] (-.5,-.6) to (-.5,.6);
\draw[-to,thin] (.3,-.6) to (.3,.6);
  \draw[-to,thin,fill=black!10!white] (.28,0)++(-.16,0) arc(-180:180:0.16);
\limitteleporterOO{(-.5,.4)}{}{}{(.3,.4)};
\limitteleporterOO{(-.5,-.4)}{}{}{(.3,-.4)};
\node at (.6,0) {$\catlabel{\lambda}$};
\end{tikzpicture}+
\begin{tikzpicture}[anchorbase,scale=1]
	\draw[-,thin] (0.3,-0.6) to[out=90, in=0] (0,-0.1);
	\draw[-to,thin] (0,-0.1) to[out = 180, in = 90] (-0.3,-.6);
  \draw[to-,thin] (1,0)++(-.28,0) arc(180:-180:0.28);
  \draw[to-,thin,fill=black!10!white] (1.54,0)++(-.16,0) arc(0:360:0.14);
   \node at (1.6,0) {$\catlabel{\lambda}$};
	\draw[to-,thin] (0.3,.6) to[out=-90, in=0] (0,0.1);
	\draw[-,thin] (0,0.1) to[out = -180, in = -90] (-0.3,.6);
\bentlimitteleporterOO{(.3,.4)}{}{}{(.88,.25)}{(.6,.4)};
\bentlimitteleporterOO{(.3,-.4)}{}{}{(.88,-.25)}{(.6,-.4)};
\end{tikzpicture}
-
{\textstyle\frac{1}{2}}\:
\begin{tikzpicture}[anchorbase,scale=1]
	\draw[-to,thin] (0.3,-.6) to[out=90,in=0] (0,-.1) to [out=180,in=90] (-0.3,-.6);
\draw[to-,thin] (0.3,.6) to[out=-90,in=0] (0,.1) to [out=180,in=-90] (-0.3,.6);
\bentlimitteleporterOO{(-.2,-.18)}{}{}{(-.2,.18)}{(-.6,-.1) and (-.6,.1)};
\opendot{.28,-.4};\node at (.58,-.4) {$\kmlabel{-1}$};
       \node at (0.7,0) {$\catlabel{\lambda}$};
\end{tikzpicture}
-
{\textstyle\frac{1}{2}}\:
\begin{tikzpicture}[anchorbase,scale=1]
	\draw[-to,thin] (0.3,-.6) to[out=90,in=0] (0,-.1) to [out=180,in=90] (-0.3,-.6);
\draw[to-,thin] (0.3,.6) to[out=-90,in=0] (0,.1) to [out=180,in=-90] (-0.3,.6);
\bentlimitteleporterOO{(-.2,-.18)}{}{}{(-.2,.18)}{(-.6,-.1) and (-.6,.1)};
\opendot{.26,.3};\node at (.58,.3) {$\kmlabel{-1}$};
       \node at (0.7,0) {$\catlabel{\lambda}$};
\end{tikzpicture}
&=0\:.
\end{align*}
The first of these is the first defining relation from \cref{KMrels1}.
The third one follows from \cref{fishing}
after making an obvious application of \cref{dumby1}.
To prove the second one, we take the equation from \cref{useless}
and add it to the equation obtained from it by applying $\tT$ to both sides to deduce that 
$$
\begin{tikzpicture}[anchorbase,scale=.88]
\draw[to-,thin] (-.35,1.8) to [out=-90,in=180] (-.1,1.4) to [out=0,in=-90] (.15,1.8);
	\draw[to-,thin] (-0.35,.3)  to [out=90,in=-90] (.15,.9) to [out=90,in=0] (-0.1,1.2);
 \draw[-,thin] (.15,.3)  to [out=90,in=-90] (-.35,.9) to [out=90,in=180] (-.1,1.2);
\bentlimitteleporterOO{(.12,1.56)}{}{}{(.12,1.04)}{(.5,1.36) and (.5,1.214)};
\draw[to-,thin,fill=black!10!white] (-.5,1) arc(180:-180:0.16);
\node at (.4,.8) {$\catlabel{\lambda}$};
\end{tikzpicture}
-
\begin{tikzpicture}[anchorbase,scale=.88]
\draw[to-,thin] (-.35,-1.8) to [out=90,in=180] (-.1,-1.4) to [out=0,in=90] (.15,-1.8);
	\draw[to-,thin] (-0.35,-.3)  to [out=-90,in=90] (.15,-.9) to [out=-90,in=0] (-0.1,-1.2);
 \draw[-,thin] (.15,-.3)  to [out=-90,in=90] (-.35,-.9) to [out=-90,in=180] (-.1,-1.2);
\bentlimitteleporterOO{(-.33,-1.56)}{}{}{(-.33,-1.04)}{(-.84,-1.36) and (-.84,-1.214)};
\draw[to-,thin,fill=black!10!white] (.3,-.9) arc(0:360:0.16);
\node at (.6,-1) {$\catlabel{\lambda}$};
\end{tikzpicture}
=
{\textstyle\frac{1}{2}}\:\begin{tikzpicture}[anchorbase,scale=1.2]
	\draw[-,thin] (0.2,-0.4) to[out=90, in=0] (0,0.05);
	\draw[-to,thin] (0,0.05) to[out = 180, in = 90] (-0.2,-0.4);
\draw[to-,thin] (-.2,.7) to [out=-90,in=180] (0,.25) to [out=0,in=-90] (.2,.7);
\bentlimitteleporterOO{(.17,.4)}{}{}{(.16,-.05)}{(.5,.3) and (.5,.05)};
\opendot{.2,-.25};\node at (.42,-.25) {$\kmlabel{-1}$};
\node at (.6,.2) {$\catlabel{\lambda}$};
\end{tikzpicture}+
{\textstyle\frac{1}{2}}\:\begin{tikzpicture}[anchorbase,scale=1.2]
	\draw[-,thin] (0.2,-0.4) to[out=90, in=0] (0,0.05);
	\draw[-to,thin] (0,0.05) to[out = 180, in = 90] (-0.2,-0.4);
\draw[to-,thin] (-.2,.7) to [out=-90,in=180] (0,.25) to [out=0,in=-90] (.2,.7);
\bentlimitteleporterOO{(.17,.4)}{}{}{(.16,-.05)}{(.5,.3) and (.5,.05)};
\opendot{.2,.55};\node at (.42,.55) {$\kmlabel{-1}$};
\node at (.6,.2) {$\catlabel{\lambda}$};
\end{tikzpicture}
\stackrel{\cref{dumby1}}{=}
{\textstyle\frac{1}{2}}\:\begin{tikzpicture}[anchorbase,scale=1.2]
	\draw[-,thin] (0.2,-0.4) to[out=90, in=0] (0,0.05);
	\draw[-to,thin] (0,0.05) to[out = 180, in = 90] (-0.2,-0.4);
\draw[to-,thin] (-.2,.7) to [out=-90,in=180] (0,.25) to [out=0,in=-90] (.2,.7);
\opendot{.2,-.25};\node at (.42,-.25) {$\kmlabel{-1}$};
\opendot{.2,.55};\node at (.42,.55) {$\kmlabel{-1}$};
\node at (.6,.2) {$\catlabel{\lambda}$};
\end{tikzpicture}\:.
$$
Now use \cref{veryeasy}.
\item The second relation from \cref{rels1} is the most complicated to check
since it involves an equality of $8 \times 8$ matrices,
and there are 20 non-zero entries in these matrices. After simplifying with \cref{stpauls,ernie}
and using 
the symmetries $\tR$ and $\tT$, the calculation reduces to checking
five relations, coming from the $(EEE 1_\lambda,EEE 1_\lambda)$-, 
$(FEE 1_\lambda,EEF 1_\lambda)$-,
$(EEF 1_\lambda, EEF 1_\lambda)$-,
$(EFE 1_\lambda,EEF 1_\lambda)$- 
and $(EFE 1_\lambda,EFE 1_\lambda)$-entries.
The first two of these are
\begin{align*}
\:,
\end{align*}
both of which follow from the second defining relation in \cref{KMrels2} (using also some planar isotopy to get the one on the right).
The relations from the $(EEF 1_\lambda,EEF 1_\lambda)$-
and $(EFE 1_\lambda, EEF 1_\lambda)$-entries are
\begin{align*}
%
\:.
\end{align*}
We leave the verification of these as exercises for the reader---simplify using \cref{useless,fishing}.

Finally, we must verify the relation arising from the
$(EFE 1_\lambda,EFE 1_\lambda)$-entry.
We found this to be 
significantly harder than the relations encountered so far!
After rearranging terms 
and simplifying 
a little using \cref{stpauls,ernie}, it requires the following to be proved:
\begin{multline*}
%
\!.
\end{multline*}
Notice that the right-hand side is the image of the left-hand side with $\lambda$ replaced by $-\lambda-2$ under the map $-\tR$.
Thus, we must show that $A_\lambda+\tR(A_{-\lambda-2}) = 0$ where $A_\lambda$ is the expression consisting of the nine terms on the left-hand side of the relation.
We show equivalently that $B_\lambda+\tR(B_{-\lambda-2}) = 0$ where
\begin{align*}
B_\lambda &=\!%
\!\!\:.
\end{align*}
To obtain the second expression for $B_\lambda$, we have used \cref{dumby0}
several times to split the fourth term and the last term into two and to combine the fifth, sixth, seventh and eight terms into one.
We denote the eight terms in this
by $B_{\lambda;1},\dots,B_{\lambda;8}$ in order from left to right,
so $B_\lambda=B_{\lambda;1}+\cdots+B_{\lambda;8}$.
By \cref{florencecor}, we have that $B_{\lambda;1} = C_{\lambda;1}+C_{\lambda;2}+C_{\lambda;3}+C_{\lambda;4}$ where
\begin{align*}
C_{\lambda;1}\!
&
=%
\!\right]_{u^{-1}}\!\!\!\!.
\end{align*}
For the final equality in the expansion of $C_{\lambda;3}$,
the third term on the previous line is 0 by a similar argument to the proof of \cref{barre1,barre2}. 
To see that the second terms on the two lines are equal,
this is clear
when $\lambda < -2$
for the simple reason that the counterclockwise fake bubble polynomial is 0 in that case, and it is almost as easy to see that both terms are 0 when $\lambda=-2$.
So we may assume that $\lambda > -2$.
Then the clockwise bubble can be converted into an internal 
bubble using 
\cref{hazistalk}(2), after which it cancels with the counterclockwise internal bubble by \cref{stpauls}.
In the expansion of $C_{\lambda;4}$, three terms have been removed 
without explicit reference 
since they are 0 by the usual arguments.
Now we collect the pieces to see that $B_\lambda = D_\lambda + E_\lambda + F_\lambda$ where
\begin{align}\notag
D_\lambda &=\begin{tikzpicture}[anchorbase,scale=.85]
	\draw[to-,thin] (0,-.6) to[out=135,in=-90] (-.4,0) to [out=90,in=-135] (0,.6);
   \draw[to-,thin] (-.4,.6) to (.4,-.6);
   \draw[-to,thin] (-.4,-.6) to (.4,.6);
   \node at (.5,0) {$\catlabel{\lambda}$};
\end{tikzpicture}
-\left[\begin{tikzpicture}[anchorbase,scale=.85]
	\draw[-,thin] (0.2,0.6) to[out=-90, in=0] (0,.15);
	\draw[-to,thin] (0,.15) to[out = 180, in = -90] (-0.2,0.6);
	\draw[-to,thin] (1.1,-0.6) to (1.1,0.6);
   \pino{(1.1,0)}{-}{(.5,0)}{u};
    \node at (1.35,0) {$\catlabel{\lambda}$};
    \filledanticlockwisebubble{(-.6,0)};
      \node at (-.6,0) {$\kmlabel{u}$};
     \pino{(-0.18,0.4)}{-}{(-.9,.4)}{u};
     	\draw[to-,thin] (0.2,-.6) to[out=90, in=0] (0,-.15);
	\draw[-,thin] (0,-.15) to[out = 180, in = 90] (-0.2,-.6);
     \pino{(-0.18,-0.45)}{-}{(-.9,-.44)}{u};
\end{tikzpicture}\!\right]_{u^{-1}}\text{(the second term is new, it 
cancels a term on the next line)},\\\notag
E_\lambda&={\textstyle\frac{1}{2}}
\left[\:2\begin{tikzpicture}[anchorbase,scale=.85]
	\draw[-,thin] (0.2,0.6) to[out=-90, in=0] (0,.15);
	\draw[-to,thin] (0,.15) to[out = 180, in = -90] (-0.2,0.6);
	\draw[-to,thin] (1.1,-0.6) to (1.1,0.6);
   \pino{(1.1,0)}{-}{(.5,0)}{u};
    \node at (1.35,0) {$\catlabel{\lambda}$};
    \filledanticlockwisebubble{(-.6,0)};
      \node at (-.6,0) {$\kmlabel{u}$};
     \pino{(-0.18,0.4)}{-}{(-.9,.4)}{u};
     	\draw[to-,thin] (0.2,-.6) to[out=90, in=0] (0,-.15);
	\draw[-,thin] (0,-.15) to[out = 180, in = 90] (-0.2,-.6);
     \pino{(-0.18,-0.45)}{-}{(-.9,-.44)}{u};
\end{tikzpicture}
+2\begin{tikzpicture}[anchorbase,scale=.85]
	\draw[to-,thin] (-.4,.6) to[out=-90,in=180] (-.2,.25) to[out=0,in=-90] (0,.6);
   	\draw[-to,thin] (-.4,-.6) to[out=90,in=180] (-.2,-.25) to[out=0,in=90] (0,-.6);
   \draw[-to,thin] (.4,-.6) to (.4,.6);
   	\filledanticlockwisebubble{(-.7,0.03)};\node at (-.7,0.03) {$\kmlabel{u}$};
   \opendotsmall{.4,0};\node at (.1,0) {$\kmlabel{-1}$};
   \pino{(-.4,.45)}{-}{(-1.1,.45)}{u};
   \pino{(-.4,-.45)}{-}{(-1.1,-.45)}{u};
   \node at (.65,0) {$\catlabel{\lambda}$};
\end{tikzpicture}
-\begin{tikzpicture}[anchorbase,scale=.85]
	\draw[to-,thin] (-.4,.6) to[out=-90,in=180] (-.2,.25) to[out=0,in=-90] (0,.6);
   	\draw[-to,thin] (-.4,-.6) to[out=90,in=180] (-.2,-.25) to[out=0,in=90] (0,-.6);
   \draw[-to,thin] (.4,-.6) to (.4,.6);
   	\filledanticlockwisebubble{(-1,0.33)};\node at (-1,0.33) {$\kmlabel{u}$};
   \opendotsmall{.4,.2};\node at (.1,.23) {$\kmlabel{-1}$};
   \pino{(.4,-.05)}{-}{(-.6,-0.05)}{u};
   \pino{(-.4,-.48)}{-}{(-1.1,-.48)}{u};
   \node at (.65,0) {$\catlabel{\lambda}$};
\end{tikzpicture}
-\begin{tikzpicture}[anchorbase,scale=.85]
	\draw[to-,thin] (-.4,-.6) to[out=90,in=180] (-.2,-.25) to[out=0,in=90] (0,-.6);
   	\draw[-to,thin] (-.4,.6) to[out=-90,in=180] (-.2,.25) to[out=0,in=-90] (0,.6);
   \draw[-to,thin] (.4,-.6) to (.4,.6);
   	\filledanticlockwisebubble{(-1,-0.33)};\node at (-1,-0.33) {$\kmlabel{u}$};
   \opendotsmall{.4,-.2};\node at (.1,-.23) {$\kmlabel{-1}$};
   \pino{(.4,.05)}{-}{(-.6,0.05)}{u};
   \pino{(-.4,.48)}{-}{(-1.1,.48)}{u};
   \node at (.65,0) {$\catlabel{\lambda}$};
\end{tikzpicture}
-\begin{tikzpicture}[anchorbase,scale=.85]
	\draw[-,thin] (0.2,0.6) to[out=-90, in=0] (0,.28);
	\draw[-to,thin] (0,.28) to[out = 180, in = -90] (-0.2,0.6);
	\draw[-to,thin] (1.3,-0.6) to (1.3,0.6);
    \node at (1.64,-.2) {$\catlabel{\lambda}$};
    \filledanticlockwisebubble{(0.2,-.01)};\node at (0.2,0) {$\kmlabel{u}$};
     	\draw[to-,thin] (0.2,-.6) to[out=90, in=0] (0,-.28);
	\draw[-,thin] (0,-.28) to[out = 180, in = 90] (-0.2,-.6);
   \pino{(.2,-.45)}{-}{(.85,-.45)}{u};
   \pino{(.2,.45)}{-}{(.85,.45)}{u};
\pino{(1.3,0)}{-}{(.7,0)}{u};
\opendotsmall{1.3,.3};\node at (1.64,.3) {$\kmlabel{-1}$};
\limitbandoo{(-.16,.4)}{}{}{(-.16,-.4)};
\end{tikzpicture}\!\right]_{u^{-1}}\!\!\!\!,\\\label{flambda}
F_\lambda&=4\:
\begin{tikzpicture}[anchorbase,scale=.85]
	\draw[to-,thin] (0,-.6) to (0,.6);
	   \draw[to-,thin] (-.4,.6) to (-.4,-.6);
   \draw[-to,thin] (.4,-.6) to (.4,.6);
   \node at (.75,0) {$\catlabel{\lambda}$};
   \draw[to-,thin,fill=black!10!white] (-.26,.12) arc(0:360:0.15);
   \draw[-to,thin,fill=black!10!white] (.26,.12) arc(-180:180:0.15);
\opendotsmall{-.4,-.2};  
\opendotsmall{0,0.1};  
   \limitteleporteroo{(-.4,.4)}{}{}{(0,0.4)};
   \limitteleporteroo{(-.4,-.4)}{}{}{(0,-0.4)};
   \limitteleporteroo{(.4,-.18)}{}{}{(0,-0.18)};
\end{tikzpicture}\!\!-\begin{tikzpicture}[anchorbase,scale=.85]
   \draw[-to,thin] (-.4,-.6) to (-.4,.6);
   \draw[to-,thin] (0,-.6) to (0,.6);
   \draw[-to,thin] (.6,-.6) to (.6,.6);
   \opendotsmall{.6,0};\node at (.3,0) {$\kmlabel{-1}$};
   \node at (.85,0) {$\catlabel{\lambda}$};
\end{tikzpicture}\!\!+
2\:\begin{tikzpicture}[anchorbase,scale=.85]
	\draw[-to,thin] (0.2,0.6) to (.2,-.6);
		\draw[-to,thin] (-.2,-.6) to (-0.2,0.6);
	\draw[-to,thin] (1.3,-0.6) to (1.3,0.6);
    \node at (1.6,0) {$\catlabel{\lambda}$};
\clockwisebubble{(0.75,.1)};
\limitteleporteroo{(.95,.1)}{}{}{(1.3,.1)};
\limitteleporteroo{(.55,.1)}{}{}{(-.2,.1)};
	\draw[to-,thin,fill=black!10!white] (1.15,-.3) arc(180:-180:0.15);
     \opendotsmall{-.2,-.2};
\end{tikzpicture}-2
\left[\begin{tikzpicture}[anchorbase,scale=.85]
	\draw[-to,thin] (-0.2,-0.6) to (-.2,.6);
	\draw[-to,thin] (0.2,0.6) to (.2,-.6);
		\draw[-to,thin] (1.2,-0.6) to (1.2,0.6);
    \node at (1.5,0) {$\catlabel{\lambda}$};
 \filledclockwisebubble{(.6,-.2)};\node at (.6,-.2) {$\kmlabel{u}$};
   \pino{(1.2,.3)}{+}{(.5,.3)}{u};
     \pino{(-0.2,0.3)}{+}{(-.9,.3)}{u};
	\draw[to-,thin,fill=black!10!white] (1.05,-0.2) arc(180:-180:0.15);
     \opendotsmall{-.2,-.2};
     \end{tikzpicture}\!\right]_{u^{-1}}
\!\!\!\!.
\end{align}
Next, observe that $D_\lambda + \tR(D_{-\lambda-2}) = 0$ by \cref{altbraid}.
Also $E_\lambda = 0$, indeed, already the expression inside the square brackets vanishes
as follows from the elementary identity
$$\textstyle
\frac{2}{(u-x)(u-y)(u-z)} +
\frac{2}{(u-x)(u-y) z} -
\frac{1}{
(u-y)(u-z)z}-\frac{1}{(u-x)(u-z)z}
-\frac{x+y}{(u-x)(u-y)(u-z)z}=0
$$
where $x$ represents the dot on the top left component, 
$y$ the dot on the bottom left component, and $z$ the dot on the rightmost vertical string.
\noindent
In view of all of this, it just remains to show that
$F_\lambda + \tR(F_{-\lambda-2}) = 0$.
Since $\tT$ is an involution
and we always have that $\lambda \leq -1$ or $-\lambda-2 \leq -1$, 
we may assume without loss of generality that $\lambda \leq -1$.
Hence, we have that
\begin{align}

&= \left\{
\begin{array}{ll}
u\:
\begin{tikzpicture}[anchorbase,scale=.85]
	\draw[-to,thin] (-.4,-.3) to (-.4,.3);
	\draw[to-,thin] (0,-.3) to (0,.3);
	\draw[-to,thin] (.4,-.3) to (.4,.3);
	\node at (.6,0) {$\catlabel{\lambda}$};
\end{tikzpicture}
+\begin{tikzpicture}[anchorbase,scale=.85]
	\draw[-to,thin] (-.4,-.3) to (-.4,.3);
	\draw[to-,thin] (0,-.3) to (0,.3);
	\draw[-to,thin] (.8,-.3) to (.8,.3);
	\anticlockwisebubble{(.4,0)};\node at (.4,0) {$\kmlabel{1}$};
	\node at (1,0) {$\catlabel{\lambda}$};
\end{tikzpicture}
&\text{if $\lambda=-1$}
\\
\begin{tikzpicture}[anchorbase,scale=.85]
	\draw[-to,thin] (-.4,-.3) to (-.4,.3);
	\draw[to-,thin] (0,-.3) to (0,.3);
	\draw[-to,thin] (.4,-.3) to (.4,.3);
	\node at (0,-.35) {$\phantom{.}$};
	\node at (0,.35) {$\phantom{.}$};
	\node at (.6,0) {$\catlabel{\lambda}$};
\end{tikzpicture}
&\text{if $\lambda = -2$}
\\
0&\text{if $\lambda \leq -3$.}
\end{array}\right.\label{useful}
\end{align}
From \cref{flambda} and using the first of the equalities
\cref{useful}, we have that
\begin{align*}
\tR(F_{-\lambda-2})&=4\:
%
,
\end{align*}
where we also used
\cref{hazistalk}(1),
 \cref{stpauls,funeral}
 to get the last three equalities.
 Also by \cref{hideaway}, we have that
 \begin{align*}
 F_\lambda&=4\:
%
\:.
\end{align*}
\end{itemize}
The last equality here needs further explanation.
We used \cref{lazier} to rewrite 
the fourth term on the penultimate line to deduce that the sum of the
fourth, fifth and sixth terms on this line equal 
\begin{equation}
2%
\!
	\right]_{u^{-1}}\!\!\!\!\!.\label{gac}
\end{equation}
Using also \cref{hazistalk}(1) and \cref{stpauls},
the first term in \cref{gac} 
simplifies further to 
$%
$,
hence, it cancels
with the second term on the penultimate line above.
When $\lambda \leq -3$, it is easy to see from \cref{useful}
that the remaining four terms in \cref{gac} are all 0, and the desired
equality follows.
If $\lambda = -2$, the remaining terms are 0 again---this time two of them are non-zero but they cancel with each other.
Finally, when $\lambda =-1$ the remaining four terms simplify
 by a calculation using also \cref{funeral}
 to obtain the final term in the formula we are trying to prove.
We use the simplified formulae for 
$F_{\lambda}$ and $\tR(F_{-\lambda-2})$ now established
to see that $F_{\lambda}+\tR(F_{-\lambda-2})=0$.
Indeed, after cancelling the clockwise internal bubble and a dot from the leftmost string, the sum is equal to
$$
2\:\begin{tikzpicture}[anchorbase,scale=.85]
	\draw[-to,thin] (-0.2,0.6) to (-.2,-.6);
		\draw[-to,thin] (.2,-.6) to (0.2,0.6);
	\draw[-to,thin] (-1.3,-0.6) to (-1.3,0.6);
    \node at (.45,0) {$\catlabel{\lambda}$};
\anticlockwisebubble{(-0.75,0)};
\limitteleporteroo{(-.95,0)}{}{}{(-1.3,0)};
\limitteleporteroo{(-.55,0)}{}{}{(.2,0)};
\end{tikzpicture}-
2\:\begin{tikzpicture}[anchorbase,scale=.85]
	\draw[to-,thin] (0,-.6) to (0,.6);
	   \draw[to-,thin] (-.4,.6) to (-.4,-.6);
   \draw[-to,thin] (1.1,-.6) to (1.1,.6);
   \node at (1.35,0) {$\catlabel{\lambda}$};
\anticlockwisebubble{(0.55,0)};
\limitteleporteroo{(.75,0)}{}{}{(1.1,0)};
\limitteleporteroo{(.35,0)}{}{}{(-.4,0)};
   \limitteleporteroo{(-.4,.4)}{}{}{(0,0.4)};
   \limitteleporteroo{(-.4,-.4)}{}{}{(0,-0.4)};
\end{tikzpicture}
-4\:
\begin{tikzpicture}[anchorbase,scale=.85]
	\draw[to-,thin] (0,-.6) to (0,.6);
	   \draw[to-,thin] (.4,.6) to (.4,-.6);
   \draw[-to,thin] (-.4,-.6) to (-.4,.6);
   \node at (.75,0) {$\catlabel{\lambda}$};
   \draw[-to,thin,fill=black!10!white] (.26,.12) arc(-180:180:0.15);
\opendotsmall{.4,-.2};  
   \limitteleporteroo{(.4,.4)}{}{}{(0,0.4)};
   \limitteleporteroo{(.4,-.4)}{}{}{(0,-0.4)};
   \limitteleporteroo{(-.4,0)}{}{}{(0,0)};
\end{tikzpicture}
-4\:\begin{tikzpicture}[anchorbase,scale=.85]
	\draw[to-,thin] (0,-.6) to (0,.6);
	   \draw[to-,thin] (-.4,.6) to (-.4,-.6);
   \draw[-to,thin] (.4,-.6) to (.4,.6);
   \node at (.75,0) {$\catlabel{\lambda}$};
   \draw[-to,thin,fill=black!10!white] (.26,.12) arc(-180:180:0.15);
\opendotsmall{0.4,-0.4};  
   \limitteleporteroo{(-.4,.3)}{}{}{(0,0.3)};
   \limitteleporteroo{(-.4,-.4)}{}{}{(0,-0.4)};
   \limitteleporteroo{(.4,-.18)}{}{}{(0,-0.18)};
\end{tikzpicture}
+2 \delta_{\lambda,-1}
\:\begin{tikzpicture}[anchorbase,scale=.85]
	\draw[to-,thin] (0,-.6) to (0,.6);
	   \draw[to-,thin] (-.4,.6) to (-.4,-.6);
   \draw[-to,thin] (.4,-.6) to (.4,.6);
   \node at (.55,0) {$\catlabel{\lambda}$};
   \limitteleporteroo{(-.4,.25)}{}{}{(0,0.25)};
   \limitteleporteroo{(-.4,-.25)}{}{}{(0,-0.25)};
\end{tikzpicture},
$$
which is 0 by \cref{lastgasp}.

It just remains to prove \cref{magic}.
By \cref{southwalespolice1,southwalespolice2}, 
it suffices to show that
$$
(-1)^\lambda \left[\Omega_t\left(\O(u)\right)_\lambda\right]_{u^{< 0}}
=
\left[\begin{tikzpicture}[anchorbase,scale=1.2]
\filledclockwisebubble{(0,0)};
\node at (0,0) {$\kmlabel{-u}$};
\end{tikzpicture}\:\:\:
\begin{tikzpicture}[anchorbase]
\anticlockwisebubble{(0,0)};
\pinO{(.2,0)}{-}{(.9,0)}{u};
\node at (1.25,0) {$\catlabel{\lambda}$};
\end{tikzpicture}
\right]_{u^{< 0}}
+
\left[
\begin{tikzpicture}[anchorbase]
\clockwisebubble{(0,0)};
\pinO{(-.2,0)}{-}{(-.9,0)}{-u};
\end{tikzpicture}
\:\:\:
\begin{tikzpicture}[anchorbase]
\filledanticlockwisebubble{(0,0)};
\node at (0,0) {$\kmlabel{u}$};
\node at (.4,0) {$\catlabel{\lambda}$};
\end{tikzpicture}
\right]_{u^{< 0}}+
\begin{tikzpicture}[anchorbase]
\clockwisebubble{(0,0)};
\pinO{(-.2,0)}{-}{(-.9,0)}{-u};
\end{tikzpicture}\:\:\:
\begin{tikzpicture}[anchorbase]
\anticlockwisebubble{(0,0)};
\pinO{(.2,0)}{-}{(.9,0)}{u};
\node at (1.25,0) {$\catlabel{\lambda}$};
\end{tikzpicture}
.
$$
Equivalently, since
$\begin{tikzpicture}[anchorbase]
\clockwisebubble{(0,0)};
\pinO{(-.2,0)}{-}{(-.9,0)}{-u};
\node at (.4,0) {$\catlabel{\lambda}$};
\end{tikzpicture}=-
\begin{tikzpicture}[anchorbase]
\clockwisebubble{(0,0)};
\pinO{(-.2,0)}{+}{(-.9,0)}{u};
\node at (.4,0) {$\catlabel{\lambda}$};
\end{tikzpicture}
$, we show that
$$
(-1)^\lambda \left[\Omega_t\left(\O(u)\right)_\lambda\right]_{u^{< 0}}
=
\left[\begin{tikzpicture}[anchorbase,scale=1.2]
\filledclockwisebubble{(0,0)};
\node at (0,0) {$\kmlabel{-u}$};
\end{tikzpicture}\:\:\:
\begin{tikzpicture}[anchorbase]
\anticlockwisebubble{(0,0)};
\pinO{(.2,0)}{-}{(.9,0)}{u};
\node at (1.25,0) {$\catlabel{\lambda}$};
\end{tikzpicture}
\right]_{u^{< 0}}
-
\left[
\begin{tikzpicture}[anchorbase]
\clockwisebubble{(0,0)};
\pinO{(-.2,0)}{+}{(-.9,0)}{u};
\end{tikzpicture}
\:\:\:
\begin{tikzpicture}[anchorbase]
\filledanticlockwisebubble{(0,0)};
\node at (0,0) {$\kmlabel{u}$};
\node at (.4,0) {$\catlabel{\lambda}$};
\end{tikzpicture}
\right]_{u^{< 0}}
-\begin{tikzpicture}[anchorbase]
\clockwisebubble{(0,0)};
\pinO{(-.2,0)}{+}{(-.9,0)}{u};
\end{tikzpicture}\:\:\:
\begin{tikzpicture}[anchorbase]
\anticlockwisebubble{(0,0)};
\pinO{(.2,0)}{-}{(.9,0)}{u};
\node at (1.25,0) {$\catlabel{\lambda}$};
\end{tikzpicture}.
$$
By \cref{deltadef} and the definition of $\Omega_t$, we have that
\begin{align*}
(-1)^\lambda 
\left[\Omega_t\left(\O(u)\right)\right]_{u^{<0}}&=
-2u \left[\:
\begin{tikzpicture}[anchorbase,scale=.7]
\node at (1.8,-.55) {$\catlabel{\lambda}$};
\draw[-to,thin] (.3,-.4) arc(-180:180:.4);
\pinO{(.95,-.08)}{-}{(1.95,-.08)}{u};
\draw[-to,thin,fill=black!10!white] (1.22,-.5) arc(0:-360:0.2);
\end{tikzpicture}
+
\begin{tikzpicture}[anchorbase,scale=.7]
\draw[-to,thin] (.3,-.4) arc(0:-360:.4);
\draw[-to,thin,fill=black!10!white] (-.62,-.5) arc(180:540:0.2);
\pinO{(-.3,-.08)}{+}{(-1.3,-.08)}{u};
\node at (.65,-.35) {$\catlabel{\lambda}$};
\end{tikzpicture}
\:\right]_{u^{<-1}}.
\end{align*}
Expanding the definitions of the internal bubbles using \cref{internal2,internal1}, this equals $A+B+C$
where
\begin{align*}
A&=
u\left[\sum_{n=0}^{-\lambda}
(-1)^{-\lambda-n}\:\;
\begin{tikzpicture}[anchorbase,scale=1]
\draw[-to,thin] (-.3,.2) arc(-180:180:.2);
\clockwisebubble{(-.75,.2)};\node at (-.75,.2) {$\kmlabel{n}$};
\pinO{(0.1,.2)}{-}{(.75,.2)}{u};
\opendot{0,0.02};\node at (0.5,-.15) {$\kmlabel{-\lambda-n-1}$};
\node at (0.45,.1) {$\phantom{\kmlabel{-\lambda-n-1}}$};
      \node at (1.15,0.2) {$\catlabel{\lambda}$};
\end{tikzpicture}
\right]_{u^{<-1}},\\
B&=
u\left[\sum_{n=0}^\lambda
(-1)^{\lambda-n}\:
\begin{tikzpicture}[anchorbase,scale=1]
\draw[-to,thin] (.3,.2) arc(0:-360:.2);
\anticlockwisebubble{(.75,.2)};\node at (.75,.2) {$\kmlabel{n}$};
\pinO{(-.1,.2)}{+}{(-.75,.2)}{u};
\opendot{0,0.02};
\node at (-0.5,-.15) {$\kmlabel{\lambda-n-1}$};
\node at (-0.45,.1) {$\phantom{\kmlabel{\lambda-n-1}}$};
      \node at (1.15,0.2) {$\catlabel{\lambda}$};
\end{tikzpicture}
\right]_{u^{<-1}},\\
C&=
-u \left[
\begin{tikzpicture}[anchorbase]
\draw[to-,thin] (-.17,0) arc(180:-180:.2);
\draw[to-,thin] (1.17,0) arc(0:360:.2);
\limitteleporterOO{(.22,0)}{}{}{(.78,0)};
\opendot{.08,-.18};\node at (.36,-.23) {$\kmlabel{-1}$};
\pinO{(.08,.18)}{-}{(.78,.4)}{u};
\node at (1.45,0) {$\catlabel{\lambda}$};
\end{tikzpicture}+\begin{tikzpicture}[anchorbase]
\draw[to-,thin] (-.17,0) arc(180:-180:.2);
\draw[to-,thin] (1.17,0) arc(0:360:.2);
\limitteleporterOO{(.22,0)}{}{}{(.78,0)};
\opendot{.88,-.18};\node at (.6,-.23) {$\kmlabel{-1}$};
\pinO{(.98,.18)}{+}{(.28,.4)}{u};
\node at (1.45,0) {$\catlabel{\lambda}$};
\end{tikzpicture}
\right]_{u^{<-1}}.
\end{align*}
We complete the proof by showing that
$A=\left[\begin{tikzpicture}[anchorbase,scale=1.2]
\filledclockwisebubble{(0,0)};
\node at (0,0) {$\kmlabel{-u}$};
\end{tikzpicture}\:\:\:
\begin{tikzpicture}[anchorbase]
\anticlockwisebubble{(0,0)};
\pinO{(.2,0)}{-}{(.9,0)}{u};
\node at (1.25,0) {$\catlabel{\lambda}$};
\end{tikzpicture}
\right]_{u^{< 0}}$,
$B = -\left[
\begin{tikzpicture}[anchorbase]
\clockwisebubble{(0,0)};
\pinO{(-.2,0)}{+}{(-.9,0)}{u};
\end{tikzpicture}
\:\:\:
\begin{tikzpicture}[anchorbase]
\filledanticlockwisebubble{(0,0)};
\node at (0,0) {$\kmlabel{u}$};
\node at (.4,0) {$\catlabel{\lambda}$};
\end{tikzpicture}
\right]_{u^{< 0}}$ and 
$C = -\begin{tikzpicture}[anchorbase]
\clockwisebubble{(0,0)};
\pinO{(-.2,0)}{+}{(-.9,0)}{u};
\end{tikzpicture}\:\:\:
\begin{tikzpicture}[anchorbase]
\anticlockwisebubble{(0,0)};
\pinO{(.2,0)}{-}{(.9,0)}{u};
\node at (1.25,0) {$\catlabel{\lambda}$};
\end{tikzpicture}$.
For $A$ and $B$, this follows simply by expanding the definitions of the pins and fake bubble polynomials, e.g., for $B$ we have that
\begin{align*}
B &= u \sum_{n=0}^\lambda \sum_{r \geq 1}
(-1)^{\lambda-n+r}
\begin{tikzpicture}[anchorbase]
\clockwisebubble{(0,0)};
\opendot{-.2,0};
\node at (-.94,0) {$\kmlabel{\lambda-n-1+r}$};
\anticlockwisebubble{(.6,0)}; \node at (.6,0) {$\kmlabel{n}$};
\node at (1,0) {$\catlabel{\lambda}$};
\end{tikzpicture}
u^{-r-1}\\&=
-\sum_{n=0}^\lambda \sum_{r \geq \lambda-n}
(-1)^{r}
\begin{tikzpicture}[anchorbase]
\clockwisebubble{(0,0)};
\opendot{-.2,0};
\node at (-.4,0) {$\kmlabel{r}$};
\anticlockwisebubble{(.6,0)}; \node at (.6,0) {$\kmlabel{n}$};
\node at (1,0) {$\catlabel{\lambda}$};
\end{tikzpicture}
u^{\lambda-n-r-1}=-\left[
\begin{tikzpicture}[anchorbase]
\clockwisebubble{(0,0)};
\pinO{(-.2,0)}{+}{(-.9,0)}{u};
\end{tikzpicture}
\:\:\:
\begin{tikzpicture}[anchorbase]
\filledanticlockwisebubble{(0,0)};
\node at (0,0) {$\kmlabel{u}$};
\node at (.4,0) {$\catlabel{\lambda}$};
\end{tikzpicture}
\right]_{u^{< 0}}.
\end{align*}
For $C$, we first use the relations \cref{dumby1,dumby2} to see that
$$
\begin{tikzpicture}[anchorbase]
\draw[to-,thin] (-.17,0) arc(180:-180:.2);
\draw[to-,thin] (1.17,0) arc(0:360:.2);
\limitteleporterOO{(.22,0)}{}{}{(.78,0)};
\opendot{.08,-.18};\node at (.36,-.23) {$\kmlabel{-1}$};
\pinO{(.08,.18)}{-}{(.78,.4)}{u};
\node at (1.45,0) {$\catlabel{\lambda}$};
\end{tikzpicture}+\begin{tikzpicture}[anchorbase]
\draw[to-,thin] (-.17,0) arc(180:-180:.2);
\draw[to-,thin] (1.17,0) arc(0:360:.2);
\limitteleporterOO{(.22,0)}{}{}{(.78,0)};
\opendot{.88,-.18};
\node at (.6,-.23) {$\kmlabel{-1}$};
\pinO{(.98,.18)}{+}{(.28,.4)}{u};
\node at (1.45,0) {$\catlabel{\lambda}$};
\end{tikzpicture}
=\begin{tikzpicture}[anchorbase]
\draw[to-,thin] (-.17,0) arc(180:-180:.2);
\draw[to-,thin] (2.4,0) arc(0:360:.2);
\opendot{.08,-.18};
\node at (.36,-.23) {$\kmlabel{-1}$};
\node at (.36,.23) {$\phantom{\kmlabel{-1}}$};
\pinO{(2.0,0)}{+}{(1.3,0)}{u};
\pinO{(.2,0)}{-}{(.8,0)}{u};
\node at (2.6,0) {$\catlabel{\lambda}$};
\end{tikzpicture}
+\begin{tikzpicture}[anchorbase]
\draw[to-,thin] (-.97,0) arc(180:-180:.2);
\draw[to-,thin] (1.37,0) arc(0:360:.2);
\opendot{1.08,-.18};
\node at (.8,-.23) {$\kmlabel{-1}$};
\node at (.8,.23) {$\phantom{\kmlabel{-1}}$};
\opendot{-.57,0};\node at (-.27,0) {$\kmlabel{-1}$};
\pinO{(.97,0)}{+}{(.27,0)}{u};
\node at (1.6,0) {$\catlabel{\lambda}$};
\end{tikzpicture}\:.
$$
Now we can expand the definitions to get that
\begin{align*}
C&=
-u\sum_{r,s \geq 0} (-1)^r
\begin{tikzpicture}[anchorbase]
\clockwisebubble{(0,0)};
\opendot{-.2,0};\node at (-.4,0) {$\kmlabel{r}$};
\anticlockwisebubble{(.6,0)};
\opendot{.8,0};\node at (1.16,0) {$\kmlabel{s-1}$};
\end{tikzpicture}
u^{-r-s-2}
-u\sum_{r \geq 1}(-1)^r
\begin{tikzpicture}[anchorbase]
\clockwisebubble{(0,0)};
\opendot{-.2,0};\node at (-.55,0) {$\kmlabel{r-1}$};
\anticlockwisebubble{(.6,0)};
\opendot{.8,0};\node at (1.08,0) {$\kmlabel{-1}$};
\end{tikzpicture}u^{-r-1}\\
&=-u\sum_{\substack{r\geq 0 \\s \geq 1}} (-1)^r
\begin{tikzpicture}[anchorbase]
\clockwisebubble{(0,0)};
\opendot{-.2,0};\node at (-.4,0) {$\kmlabel{r}$};
\anticlockwisebubble{(.6,0)};
\opendot{.8,0};\node at (1.16,0) {$\kmlabel{s-1}$};
\end{tikzpicture}
u^{-r-s-2}\\&=-\sum_{r,s\geq 0} (-1)^r
\begin{tikzpicture}[anchorbase]
\clockwisebubble{(0,0)};
\opendot{-.2,0};\node at (-.4,0) {$\kmlabel{r}$};
\anticlockwisebubble{(.6,0)};
\opendot{.8,0};\node at (1,0) {$\kmlabel{s}$};
\end{tikzpicture}
u^{-r-s-2}=-\:\begin{tikzpicture}[anchorbase]
\clockwisebubble{(0,0)};
\pinO{(-.2,0)}{+}{(-.9,0)}{u};
\end{tikzpicture}\:\:\:
\begin{tikzpicture}[anchorbase]
\anticlockwisebubble{(0,0)};
\pinO{(.2,0)}{-}{(.9,0)}{u};
\node at (1.25,0) {$\catlabel{\lambda}$};
\end{tikzpicture}\:.
\end{align*}
This completes the proof of the theorem. 
\end{proof}

\begin{remark}\label{morepsit}
The monoidal functor in \cref{psit} is certainly not unique. One way to obtain alternative forms, maintaining the property \cref{magic} 
and preserving the
leading terms
$\textupdot{\scriptstyle\color{catcolor}\lambda}$ and
$\textupcrossing{\scriptstyle\color{catcolor}\lambda}$
in the formulae for $\Omega_t\left(\textdot\:\right)_\lambda$
and $\Omega_t\left(\textcrossing\:\right)_\lambda$,
is as follows.
Suppose that we are given invertible 2-morphisms
$\textdownstar{\scriptstyle\color{catcolor}\lambda}$
in $\fU(\sl_2)$ for each $\lambda \in t + 2\Z$,
such that $\textdownstar{\scriptstyle\color{catcolor}\lambda}$ and $\textdowndot{\scriptstyle\color{catcolor}\lambda}$ commute.
We denote the inverse of $\textdownstar{\scriptstyle\color{catcolor}\lambda}$
by labeling the star with $-1$.
Then there is a strict graded monoidal functor
$\Omega^\star_t:\cNB_t \rightarrow \cU(\sl_2;t)$
taking $B$ to $E \oplus F$
and defined 
on generating morphisms by
\begin{align*}
\Omega^\star_t\left(\textdot\:\right)_\lambda &:=
\begin{tikzpicture}[anchorbase]
\draw[-to,thin] (0,-.4) to (0,.4);
\opendot{0,0};
\node at (0.25,0) {$\catlabel{\lambda}$};
\end{tikzpicture}
-\begin{tikzpicture}[anchorbase]
\draw[to-,thin] (0,-.4) to (0,.4);
\opendot{0,0};
\node at (0.25,0) {$\catlabel{\lambda}$};
\end{tikzpicture}
,\\
\Omega_t^\star\left(\;\textcrossing\;\right)_\lambda &:=
\!\begin{tikzpicture}[anchorbase,scale=.8]
	\draw[-to,thin] (0.6,-.6) to (-0.6,.6);
	\draw[-to,thin] (-0.6,-.6) to (0.6,.6);
       \node at (0.5,0) {$\catlabel{\lambda}$};
\end{tikzpicture}
\!+\!\!\begin{tikzpicture}[anchorbase,scale=.8]
	\draw[to-,thin] (0.6,-.6) to (-0.6,.6);
	\draw[to-,thin] (-0.6,-.6) to (0.6,.6);
	\node at (-.28,-.3) {$\star$};
	\node at (.28,-.3) {$\star$};
\node at (-.3,.28) {$\star$};\node at (-.65,.3) {$\kmlabel{-1}$};
\node at (.3,.28) {$\star$};\node at (.65,.3) {$\kmlabel{-1}$};
       \node at (0.5,0) {$\catlabel{\lambda}$};
\end{tikzpicture}
\!+\!\!\!\begin{tikzpicture}[anchorbase,scale=.8]
	\draw[to-,thin] (0.6,-.6) to (-0.6,.6);
	\draw[-to,thin] (-0.6,-.6) to (0.6,.6);
	\node at (.28,-.3) {$\star$};
\node at (-.3,.28) {$\star$};\node at (-.65,.3) {$\kmlabel{-1}$};
       \node at (0.5,0) {$\catlabel{\lambda}$};
\end{tikzpicture}
\!+
\begin{tikzpicture}[anchorbase,scale=.8]
	\draw[-to,thin] (0.6,-.6) to (-0.6,.6);
	\draw[to-,thin] (-0.6,-.6) to (0.6,.6);
\draw[-to,thin,fill=black!10!white] (-.4,.1) arc(-135:225:0.2);
\draw[-to,thin,fill=black!10!white] (.4,-.1) arc(45:-315:0.2);
\node at (.3,.28) {$\star$};\node at (.65,.3) {$\kmlabel{-1}$};
	\node at (-.28,-.3) {$\star$};
       \node at (0.7,0) {$\catlabel{\lambda}$};
\end{tikzpicture}
\!\!+\begin{tikzpicture}[anchorbase,scale=.8]
	\draw[-to,thin] (-0.7,-.6) to (-0.7,.6);
	\draw[to-,thin] (.1,-.6) to (.1,.6);
\limitteleporterOO{(-.7,0)}{}{}{(.1,0)};
            \node at (.4,0) {$\catlabel{\lambda}$};
\end{tikzpicture}
\!-\begin{tikzpicture}[anchorbase,scale=.8]
 	\draw[-to,thin] (-0.45,-.6) to[out=90,in=180] (-.1,-.1) to[out=0,in=90] (0.25,-.6);
 	\draw[to-,thin] (-0.45,.6) to[out=-90,in=180] (-.1,.1) to[out=0,in=-90] (0.25,.6);
     \bentlimitteleporterOO{(.06,.14)}{}{}{(.05,-.15)}{(.62,.1) and (.62,-.1)};
	\node at (.21,-.35) {$\star$};
\node at (.22,.4) {$\star$};\node at (.55,.4) {$\kmlabel{-1}$};
\draw[-to,thin,fill=black!10!white] (-.5,.13) arc(-150:210:0.2);
   \node at (.8,0) {$\catlabel{\lambda}$};
\end{tikzpicture}
\!-
\begin{tikzpicture}[anchorbase,scale=.8]
	\draw[to-,thin] (-0.7,-.6) to (-0.7,.6);
	\draw[-to,thin] (.1,-.6) to (.1,.6);
\limitteleporterOO{(-.7,0)}{}{}{(.1,0)};
            \node at (.4,0) {$\catlabel{\lambda}$};
\end{tikzpicture}
+\!\!\begin{tikzpicture}[anchorbase,scale=.8]
 	\draw[to-,thin] (-0.45,-.6) to[out=90,in=180] (-.1,-.1) to[out=0,in=90] (0.25,-.6);
 	\draw[-to,thin] (-0.45,.6) to[out=-90,in=180] (-.1,.1) to[out=0,in=-90] (0.25,.6);
     \bentlimitteleporterOO{(-.25,.15)}{}{}{(-.25,-.15)}{(-.8,.1) and (-.8,-.1)};
\node at (-.42,.4) {$\star$};\node at (-.75,.4) {$\kmlabel{-1}$};
	\node at (-.41,-.35) {$\star$};
\draw[-to,thin,fill=black!10!white] (.32,-0.18) arc(30:-330:0.2);
   \node at (.7,0) {$\catlabel{\lambda}$};
\end{tikzpicture}\:,\\
\Omega_t^\star\left(\:\,\txtcap\:\,\right)_\lambda &:=
\begin{tikzpicture}[anchorbase,scale=.8]
	\draw[-,thin] (-0.4,-0.3) to[out=90, in=180] (0,0.3);
	\draw[-to,thin] (-0,0.3) to[out = 0, in = 90] (0.4,-0.3);
\node at (.34,0.05) {$\star$};
\node at (.7,0) {$\catlabel{\lambda}$};
\end{tikzpicture}+
\begin{tikzpicture}[anchorbase,scale=.8]
	\draw[to-,thin] (-0.4,-0.3) to[out=90, in=180] (0,0.3);
	\draw[-,thin] (-0,0.3) to[out = 0, in = 90] (0.4,-0.3);
\node at (-.34,0.05) {$\star$};
\draw[-to,thin,fill=black!10!white] (.4,0.2) arc(30:-330:0.2);
\node at (.75,0) {$\catlabel{\lambda}$};
\end{tikzpicture}
,\\
\Omega_t^\star\left(\:\,\txtcup\,\:\right)_\lambda &:=\begin{tikzpicture}[anchorbase,scale=.8] 
	\draw[-,thin] (-0.4,0.3) to[out=-90, in=180] (0,-0.3);
	\draw[-to,thin] (-0,-0.3) to[out = 0, in = -90] (0.4,0.3);
\node at (-.34,0) {$\star$};\node at (-.66,0) {$\kmlabel{-1}$};
      \node at (.65,0) {$\catlabel{\lambda}$};
\end{tikzpicture}+
\begin{tikzpicture}[anchorbase,scale=.8]
	\draw[to-,thin] (-0.4,0.3) to[out=-90, in=180] (0,-0.3);
	\draw[-,thin] (-0,-0.3) to[out = 0, in = -90] (0.4,0.3);
\draw[-to,thin,fill=black!10!white] (-.4,-.2) arc(-150:210:0.2);
\node at (.34,0) {$\star$};\node at (.66,0) {$\kmlabel{-1}$};
\node at (1.2,0) {$\catlabel{\lambda}$};
\end{tikzpicture}
\end{align*}
for $\lambda \in t+2\Z$.
The proof of the existence of this is almost identical to the proof in \cref{psit}---in all of the calculations stars cancel with their inverses so that the final relations that need to be checked reduce to the same ones as checked before.
For example, 
taking $\textdownstar{\scriptstyle\color{catcolor}\lambda}:=\begin{tikzpicture}[anchorbase,scale=.6]
\draw[-to,thin] (0,.5) to (0,-.5);
\anticlockwiseinternalbubble{(0,0)};
\node at (.4,0) {$\catlabel{\lambda}$};
\end{tikzpicture}$
produces a monoidal functor $\widetilde{\Omega}_t$ with
\begin{align*}
\widetilde{\Omega}_t\left(\textdot\:\right)_\lambda &:=
\begin{tikzpicture}[anchorbase]
\draw[-to,thin] (0,-.4) to (0,.4);
\opendot{0,0};
\node at (0.25,0) {$\catlabel{\lambda}$};
\end{tikzpicture}
-\begin{tikzpicture}[anchorbase]
\draw[to-,thin] (0,-.4) to (0,.4);
\opendot{0,0};
\node at (0.25,0) {$\catlabel{\lambda}$};
\end{tikzpicture}
,\\
\widetilde{\Omega}_t\left(\;\textcrossing\;\right)_\lambda &:=
\!\begin{tikzpicture}[anchorbase,scale=.8]
	\draw[-to,thin] (0.6,-.6) to (-0.6,.6);
	\draw[-to,thin] (-0.6,-.6) to (0.6,.6);
       \node at (0.5,0) {$\catlabel{\lambda}$};
\end{tikzpicture}
\!+\!\!\begin{tikzpicture}[anchorbase,scale=.8]
	\draw[to-,thin] (0.6,-.6) to (-0.6,.6);
	\draw[to-,thin] (-0.6,-.6) to (0.6,.6);
       \node at (0.5,0) {$\catlabel{\lambda}$};
\end{tikzpicture}
\!+\!\!\!\begin{tikzpicture}[anchorbase,scale=.8]
	\draw[to-,thin] (0.6,-.6) to (-0.6,.6);
	\draw[-to,thin] (-0.6,-.6) to (0.6,.6);
\draw[to-,thin,fill=black!10!white] (.4,.1) arc(-45:315:0.2);
\draw[to-,thin,fill=black!10!white] (-.4,-.1) arc(135:-225:0.2);
       \node at (0.7,0) {$\catlabel{\lambda}$};
\end{tikzpicture}
\!+
\begin{tikzpicture}[anchorbase,scale=.8]
	\draw[-to,thin] (0.6,-.6) to (-0.6,.6);
	\draw[to-,thin] (-0.6,-.6) to (0.6,.6);
       \node at (0.7,0) {$\catlabel{\lambda}$};
\end{tikzpicture}
\!\!+\begin{tikzpicture}[anchorbase,scale=.8]
	\draw[-to,thin] (-0.7,-.6) to (-0.7,.6);
	\draw[to-,thin] (.1,-.6) to (.1,.6);
\limitteleporterOO{(-.7,0)}{}{}{(.1,0)};
            \node at (.4,0) {$\catlabel{\lambda}$};
\end{tikzpicture}
\!-\begin{tikzpicture}[anchorbase,scale=.8]
 	\draw[-to,thin] (-0.45,-.6) to[out=90,in=180] (-.1,-.1) to[out=0,in=90] (0.25,-.6);
 	\draw[to-,thin] (-0.45,.6) to[out=-90,in=180] (-.1,.1) to[out=0,in=-90] (0.25,.6);
     \bentlimitteleporterOO{(.13,.2)}{}{}{(.13,-.2)}{(.72,.1) and (.72,-.1)};
\draw[to-,thin,fill=black!10!white] (-.5,-.13) arc(150:-210:0.2);
   \node at (.8,0) {$\catlabel{\lambda}$};
\end{tikzpicture}
\!-
\begin{tikzpicture}[anchorbase,scale=.8]
	\draw[to-,thin] (-0.7,-.6) to (-0.7,.6);
	\draw[-to,thin] (.1,-.6) to (.1,.6);
\limitteleporterOO{(-.7,0)}{}{}{(.1,0)};
            \node at (.4,0) {$\catlabel{\lambda}$};
\end{tikzpicture}
+\!\!\begin{tikzpicture}[anchorbase,scale=.8]
 	\draw[to-,thin] (-0.45,-.6) to[out=90,in=180] (-.1,-.1) to[out=0,in=90] (0.25,-.6);
 	\draw[-to,thin] (-0.45,.6) to[out=-90,in=180] (-.1,.1) to[out=0,in=-90] (0.25,.6);
     \bentlimitteleporterOO{(-.35,.2)}{}{}{(-.35,-.2)}{(-1,.1) and (-1,-.1)};
\draw[to-,thin,fill=black!10!white] (.34,0.15) arc(-20:340:0.2);
   \node at (.7,0) {$\catlabel{\lambda}$};
\end{tikzpicture}\:,\\
\widetilde{\Omega}_t\left(\:\,\txtcap\:\,\right)_\lambda &:=
\begin{tikzpicture}[anchorbase,scale=.8]
	\draw[-,thin] (-0.4,-0.3) to[out=90, in=180] (0,0.3);
	\draw[-to,thin] (-0,0.3) to[out = 0, in = 90] (0.4,-0.3);
\draw[to-,thin,fill=black!10!white] (-.45,0.2) arc(150:-210:0.2);
\node at (.7,0) {$\catlabel{\lambda}$};
\end{tikzpicture}+
\begin{tikzpicture}[anchorbase,scale=.8]
	\draw[to-,thin] (-0.4,-0.3) to[out=90, in=180] (0,0.3);
	\draw[-,thin] (-0,0.3) to[out = 0, in = 90] (0.4,-0.3);
\node at (.65,0) {$\catlabel{\lambda}$};
\end{tikzpicture}
,\\
\widetilde{\Omega}_t\left(\:\,\txtcup\,\:\right)_\lambda &:=\begin{tikzpicture}[anchorbase,scale=.8] 
	\draw[-,thin] (-0.4,0.3) to[out=-90, in=180] (0,-0.3);
	\draw[-to,thin] (-0,-0.3) to[out = 0, in = -90] (0.4,0.3);
\draw[to-,thin,fill=black!10!white] (.42,-.2) arc(-20:340:0.2);
      \node at (.85,0) {$\catlabel{\lambda}$};
\end{tikzpicture}+
\begin{tikzpicture}[anchorbase,scale=.8]
	\draw[to-,thin] (-0.4,0.3) to[out=-90, in=180] (0,-0.3);
	\draw[-,thin] (-0,-0.3) to[out = 0, in = -90] (0.4,0.3);
\node at (.65,0) {$\catlabel{\lambda}$};
\end{tikzpicture}
\end{align*}
for $\lambda \in t+2\Z$.
\end{remark}

\begin{remark}\label{grrem}
Recall from \cref{NBgrading,KMgrading} that both $\cNB_t$ and $\fU(\sl_2)$, hence, $\cU(\sl_2;t)$ 
can be equipped with {gradings}.
However, the functor $\Omega_t$ in \cref{psit} is {\em not} a graded monoidal functor---it does not preserve degrees of the generating morphisms. One way to fix this is to pass to
the $q$-envelope $\fU_q(\sl_2)$
of $\fU(\sl_2)$ defined as in \cite[Sec.~6]{BE}
(ignoring the more complicated
$\Z/2$-gradings present there).
This has $1$-morphisms 
that are formal symbols 
$q^n X 1_\lambda$ for 
1-morphisms $X 1_\lambda$ in $\fU(\sl_2)$ and $n \in \Z$,
and the 2-morphism space
$\Hom_{\fU_q(\sl_2)}(q^m X 1_\lambda, q^n Y 1_\mu)$ is $q^{n-m} \Hom_{\fU(\sl_2)}(X,Y)$,
where the $q$ here is the upward\footnote{If one prefers $q$ to be the downward grading shift functor
then one can instead use the $q^{-1}$-envelope
$\fU_{q^{-1}}(\sl_2)$ in this place and
obtain a graded monoidal functor $\tilde\Omega_t$
taking $B$ to $F \oplus \hat E$ and defined 
on morphisms as in \cref{morepsit}.}
 grading shift functor on the category of graded $\kk$-modules $(q V)_d := V_{d+1}$.
Horizontal and vertical composition making
$\fU_q(\sl_2)$ into a
strict graded 2-category are induced in an obvious way by the ones in $\fU(\sl_2)$.
Then we modify \cref{hearts}, redefining $\cU(\sl_2;t)_\loc$ 
to be the strict graded 
monoidal category with objects that are words in the free monoid $\langle \hat E, F \rangle$.
For any $X \in \langle \hat E, F \rangle$ and $\lambda \in \Z$, the corresponding 1-morphism $X 1_\lambda$ in $\fU_q(\sl_2)$
is defined so that
$\hat E 1_\lambda := q^{-\lambda-1} E 1_\lambda$, and then
\begin{equation*}
\Hom_{\cU(\sl_2;t)_\loc}(Y, X):=
\prod_{\lambda \in t+2\Z} 
\Hom_{\fU_{q}(\sl_2)_{\loc}}(Y 1_\lambda, X 1_\lambda)
\end{equation*}
for $X, Y \in \langle \hat E, F\rangle$.
The graded analog of $\Omega_t$
can now be defined to be the graded monoidal functor given on objects by
$B \mapsto F \oplus \hat E$, and on morphisms as in \cref{psit}; one just needs to check that this does indeed respect degrees.
At the decategorified level, 
this modified definition of $\Omega_t$
is consistent with the standard choice 
of the embedding $\U_q^\imath(\sl_2)
\hookrightarrow \U_q(\sl_2),
B \mapsto F + q K^{-1} E$.
\iffalse
For example, $\,\txtcap\,$ is of degree 0
in $\cNB_t$,
so we need to check that
$\:\begin{tikzpicture}[anchorbase,scale=.8]
	\draw[-,thin] (-0.4,-0.3) to[out=90, in=180] (0,0.3);
	\draw[-to,thin] (-0,0.3) to[out = 0, in = 90] (0.4,-0.3);
\node at (.7,0) {$\catlabel{\lambda}$};
\end{tikzpicture}:\hat EF 1_\lambda
\Rightarrow 1_\lambda$
and
$\:\begin{tikzpicture}[anchorbase,scale=.8]
	\draw[to-,thin] (-0.4,-0.3) to[out=90, in=180] (0,0.3);
	\draw[-,thin] (-0,0.3) to[out = 0, in = 90] (0.4,-0.3);
\draw[-to,thin,fill=black!10!white] (.4,0.2) arc(30:-330:0.2);
\node at (.75,0) {$\catlabel{\lambda}$};
\end{tikzpicture}:\hat E F 1_\lambda \Rightarrow 1_\lambda$ are also both of degree 0 as 2-morphisms in $\fU_{q^{-1}}(\sl_2)_\loc$.
The rightward cap
$\:\begin{tikzpicture}[anchorbase,scale=.8]
	\draw[-,thin] (-0.4,-0.3) to[out=90, in=180] (0,0.3);
	\draw[-to,thin] (-0,0.3) to[out = 0, in = 90] (0.4,-0.3);
\node at (.7,0) {$\catlabel{\lambda}$};
\end{tikzpicture}:EF 1_\lambda\Rightarrow 1_\lambda$ is of degree $1-\lambda$,
so indeed it is of degree 0 when viewed as a morphism $\hat EF 1_\lambda\Rightarrow 1_\lambda$ in the $q$-envelope.
\fi
\end{remark}

%=====================
% Section 5
%=====================

\section{The basis theorem}\label{sec5}

Assume as before that $\kk$ is an integral domain in which 2 is invertible and $t \in \{0,1\}$. Fix also $m, n \geq 0$.
Any morphism $f:B^{\star n} \rightarrow B^{\star m}$
is represented by a linear combination of {\em $m \times n$ string diagrams}, i.e., string diagrams with $m$ boundary points at the top and $n$ boundary points at the bottom
that are obtained by composing the generating
morphisms from \cref{gens}. 
It follows that $\Hom_{\cNB_t}(B^{\star n}, B^{\star m})$ is 0 unless $m \equiv n \pmod{2}$.

The individual strings in 
an $m \times n$ string diagram $s$
are of four basic types:
generalized cups (with two boundary points on the top edge),
generalized caps (with two boundary points on the bottom edge),
propagating strings (with one boundary point at the top and one at the bottom),
and internal bubbles (no boundary points).
We define an equivalence relation $\sim$
on the set of $m \times n$ string diagrams
by declaring that $s \sim s'$ if
their strings define the same matching on the set of $m+n$ boundary points. 
We say that $s$ is
{\em reduced} if the following properties hold:
\begin{itemize}
\item 
There are no internal bubbles.
\item Propagating strings have no critical points (i.e., points of slope 0).
\item Generalized cups/caps (i.e., strings that connect top to top or bottom to bottom) each have exactly one critical point.
\item There are no {\em double crossings} (i.e., two different strings which cross each other at least twice).
\end{itemize}
These assumptions imply in particular that there are no {\em self-intersections} ($=$ crossings of a string with itself).
For example, the first of the following undotted diagrams is not reduced (indeed, it fails all of the above conditions), while the second is reduced in the same $\sim$-equivalence:
$$
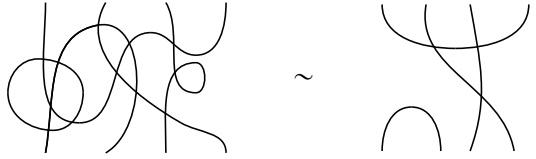
\begin{tikzpicture}[anchorbase]
\draw (0,1) to [out=-90,in=180] (.45,-0.6);
\draw (.45,-.6) to [out=0,in=180] (1.4,0.6);
\draw (2,.3) to [out=0,in=-90] (2.4,1);
\draw (2,0.2) to [out=0,in=0] (2,-.2);
\draw (0,-1) to [out=80,in=190] (.7,.7);
\draw (0,-1) to [out=80,in=190] (.7,.7);
\draw (.7,.7) to [out=10,in=30] (0.8,-1);
\draw (1.8,-0.6) to [out=-30,in=90] (2.4,-1);
\draw (0.8,1) to [out=-120,in=150] (1.8,-0.6);
\draw (1.6,-1) to [out=90,in=-180] (2,0.2);
 \draw (2,-.2) to [out=180,in=-60] (1.6,1);
\draw (-.5,-.2) to [out=90,in=180] (-.1,.25); 
\draw (-.1,.25) to [out=0,in=90] (.5,-.2); 
\draw (.1,-.7) to [out=0,in=-90] (.5,-.2); 
\draw (-.5,-.2) to [out=-90,in=180] (.1,-.7); 
\draw (1.4,.6) to [out=0,in=180] (2,0.3);
\end{tikzpicture}
\qquad\sim\qquad
\begin{tikzpicture}[anchorbase,scale=1.95]
\draw (0.6,-0.5) to [out=70,in=-80] (.6,.5);
\draw (0,-0.5) to [out=90,in=180] (0.2,-0.2);
\draw (.2,-0.2) to [out=0,in=90] (0.4,-0.5);
\draw (0,0.5) to [out=-90,in=180] (0.5,0.2);
\draw (.5,0.2) to [out=0,in=-90] (1,0.5);
\draw (.3,0.5) to[out=-100,in=100] (.9,-.5);
\end{tikzpicture}\:.
$$
Now fix a set $\overline{\RSD}(m,n)$ of representatives for the
$\sim$-equivalence classes
of {\em undotted} reduced $m \times n$ string diagrams;
the total number of such diagrams is $(m+n-1)!!$ if $m \equiv n \pmod{2}$,
and there are none otherwise.
For each of these $\sim$-equivalence
class representatives, we also choose distinguished points in the interior of each of its strings that are away from points of intersection. Then
let $\RSD(m, n)$ be the
set of all morphisms
$f:B^{\star n} \rightarrow B^{\star m}$ obtained
by closed dots
labeled by some non-negative multiplicities
at the elements of $\overline{\RSD}(m,n)$.

Recall the commutative 
algebra $\Gamma$ of Schur $Q$-functions
and the algebra 
homomorphism $\gamma_t:\Gamma \rightarrow \End_{\cNB_t}(\one)$
from \cref{gammacor}.

\begin{theorem}\label{basisthm} 
Viewing
$\Hom_{\cNB_t}(B^{\star n}, B^{\star m})$
as a $\Gamma$-module so that
$p \in \Gamma$ acts on $f:B^{\star n}\rightarrow B^{\star m}$
by $f \cdot p := f \star \gamma_t(p)$,
the morphism space
$\Hom_{\cNB_t}(B^{\star n}, B^{\star m})$
is free as a $\Gamma$-module with basis $\RSD(m, n)$.
\end{theorem}

We split the proof into two parts---spanning and linear independence.

\begin{proof}[Proof of spanning part of \cref{basisthm}]
Consider a morphism
$f \in \Hom_{\cNB_t}(B^{\star n}, B^{\star m})$ 
represented by an $m\times n$ string diagram in which there is either a self-intersection or a double crossing.
We claim that $f$ can be rewritten as a linear combination of string diagrams with no self-intersections or double crossings,
all of which have strictly fewer crossings than the original diagram. To see this, we can essentially ignore closed dots since,
up to a sign,
they can be moved freely along strings 
by the relations \cref{rels4,rels7} 
modulo a linear combination of terms with strictly fewer crossings.
If the diagram involves a curl
$\:\begin{tikzpicture}[anchorbase,scale=.8]
	\draw[-] (0,-.2) to [out=60,in=180] (.4,.3) to [out=0,in=90] (.6,.1) to[out=-90,in=0] (.4,-.1) to [out=180,in=-60] (0,.4);
\end{tikzpicture}\:$
then the morphism is 0 by the defining relations. If there are no curls, the presence of some self-intersection or double crossing implies that 
the diagram can be transformed using the braid relation and planar isotopy
into a diagram containing a bigon
$\:\begin{tikzpicture}[anchorbase,scale=1]
 	\draw[-] (-0.25,-.3) to[out=90,in=180] (0,.1) to[out=0,in=90] (0.25,-.3);
 	\draw[-] (-0.25,.3) to[out=-90,in=180] (0,-.1) to[out=0,in=-90] (0.25,.3);
\end{tikzpicture}\:$, and the resulting morphism is 0 again by the defining relations. 
This last assertion is justified in the proof of
\cite[Th.~8.3]{Lauda} by adapting an argument from \cite[Lem.~2]{Carpentier}, which establishes 
the analogous result for closed 4-valent planar graphs (viewing a crossing in our setup as a 4-valent vertex).

Applying the claim and induction on the number of crossings, 
we are reduced to considering a morphism $f$ represented by an $m\times n$ string diagram with no self-intersections or double crossings.
Using planar isotopy, we can 
assume moreover that
there are no cups or caps on propagating strings, 
and at most one cup or cap on all other strings.
In particular, all of the internal bubbles are simple circles decorated by some number of dots. Using the bubble slide relation \cref{rels13}, these
internal bubbles can then be moved so that they all appear on the
right-hand edge of the diagram and there are no nested bubbles.
Thus, we have produced a $\Gamma$-linear combination of morphisms defined by reduced string diagrams. 
In any reduced string diagram, 
closed dots can be moved to some distinguished point on each string modulo diagrams with fewer crossings. It remains to observe that 
any two undotted string diagrams of the same type are equivalent, i.e., they define the same morphism. This follows using the braid relation and planar isotopy once again.
Hence, we obtain the desired $\Gamma$-linear combination of diagrams
in $\RSD(m, n)$.
\end{proof}

The proof of linear independence
needs 
some additional input about the structure of
the 2-category
$\fU(\sl_2)$. Specifically, we need bases for its 2-morphism spaces. In fact, we are going to work with a completion 
$\fUc(\sl_2)$ of
$\fU(\sl_2)$.  
Recall from \cref{KMgrading} 
that $\fU(\sl_2)$ is naturally 
a strict graded 2-category.
For 1-morphisms $X 1_\lambda, Y1_\lambda :\lambda\rightarrow \mu$ in $\fU(\sl_2)$ (so $X$ and $Y$ are
words in $\langle E, F\rangle$ of weight $\mu-\lambda$),
the graded
2-morphism space 
\begin{equation}
\Hom_{\fU(\sl_2)}(Y1_\lambda,X1_\lambda)
=\bigoplus_{d \in \Z} \Hom_{\fU(\sl_2)}(Y1_\lambda,X1_\lambda)_d
\end{equation}
has finite-dimensional graded pieces and $\Hom_{\fU(\sl_2)}(Y1_\lambda,X1_\lambda)_d = 0$ for $d \ll 0$.
Consequently, it makes sense to pass to the completion with respect to the grading. This is a strict 2-category
with the same objects and 1-morphisms as $\fU(\sl_2)$,
and 2-morphisms that are defined from
\begin{equation}
\Hom_{\fUc(\sl_2)}(Y1_\lambda,X1_\lambda) := \prod_{d \in \Z} \Hom_{\fU(\sl_2)}(Y1_\lambda,X1_\lambda)_d
\end{equation}
with horizontal and vertical composition laws induced by the ones in $\fU(\sl_2)$.

The {\em non-degeneracy} of $\fU(\sl_2)$ 
gives bases for the 2-morphism spaces $\Hom_{\fU(\sl_2)}(Y1_\lambda,X1_\lambda)$
of a similar nature to the bases in \cref{basisthm}. 
The result can be deduced from 
\cite{Lauda,Lauda2}, but it was not formulated explicitly
until \cite[Th.~1.3]{KL3} (which also extended the result from $\sl_2$ to $\sl_n$).
To state the result in a way that is convenient for the present purposes,
take 1-morphisms $X1_\lambda,Y 1_\lambda:\lambda\rightarrow \mu$ in $\fU(\sl_2)$ represented by words $X,Y \in \langle E,F\rangle$ of lengths $m$ and $n$, respectively.
Take $s \in \RSD(n,m)$.
By writing the letters in the words $X$ and $Y$ 
at the ends of the strings at the top and bottom of the diagram $s$, respectively,
$s$ determines a matching between the letters of $X$ and $Y$. 
We say that $s$ is {\em admissible} for $X$ and $Y$ if
the letters matched by each propagating string are equal 
and the letters matched by each generalized cup or cap are different.
In that case, 
there is a corresponding 2-morphism $Y 1_\lambda \rightarrow X 1_\lambda$
in $\fU(\sl_2)$ represented by the string diagram $\vec{s}\;{\scriptstyle\color{catcolor}\lambda}$ obtained from $s$ by replacing the thick strings by thin strings oriented in the way dictated by the letters in $X$ and $Y$ ($E$ indicates upward and $F$ downward), replacing all closed dots with open dots with the same multiplicities, and
labeling the rightmost region by $\lambda$. 
Let $\overrightarrow{\RSD}(Y1_\lambda,X1_\lambda)$ be the set of oriented string diagrams arising in this way from the diagrams in $\RSD(n,m)$
that are admissible for $X$ and $Y$.

As mentioned already, 
the graded algebra $\End_{\fU(\sl_2)}(1_\lambda)$ can be identified with the algebra 
$\Lambda$ of symmetric functions, viewed as a graded algebra  so that $e_r$ and $h_r$ are in degree $2r$,
so that
${\catlabel{\lambda}\:}
\textclockwisebubble(u) = u^{-\lambda} e(-u)$ and
${\catlabel{\lambda}\:}
\textanticlockwisebubble(u) = u^\lambda h(u)$; this assertion is part of the non-degeneracy theorem that we are describing.  
Hence, each 2-morphism
space $\Hom_{\fU(\sl_2)}(Y1_\lambda,X1_\lambda)$ is naturally a graded $\Lambda$-module, 
$p \in \Lambda$ acting on $f$ by $f \cdot p := f \star p$.
The full theorem asserts moreover that each
$\Hom_{\fU(\sl_2)}(Y1_\lambda,X1_\lambda)$ is free as a graded $\Lambda$-module with homogeneous basis given by the set $\overrightarrow{\RSD}(Y1_\lambda,X1_\lambda)$.
The following is an immediate consequence.

\begin{theorem}
\label{nondegeneracy}
For $\lambda \in \Z$, the endomorphism algebra
$\End_{\fUc(\sl_2)}(1_\lambda)$ is identified with
the grading completion $\widehat{\Lambda}$ of the algebra of symmetric functions. Hence, for any 1-morphisms $X1_\lambda, Y1_\lambda:\lambda \rightarrow \mu$,
the 2-morphism space
$\Hom_{\fUc(\sl_2)}(Y1_\lambda,X1_\lambda)$ is naturally a topological $\widehat{\Lambda}$-module.
In fact, it is
free as a topological
$\widehat{\Lambda}$-module with topological basis
$\overrightarrow{\RSD}(Y1_\lambda,X1_\lambda)$.
\end{theorem}

We need one more basic lemma.

\begin{lemma}\label{etaa}
For any $a \in \kk$, there is a strict 2-functor
$\eta_a:\fU(\sl_2) \rightarrow \fU(\sl_2)$ which fixes objects and 1-morphisms, and is defined on generating 2-morphisms by
\begin{align*}
\begin{tikzpicture}[anchorbase,scale=1.4]
	\draw[-to,thin] (0.08,-.15) to (0.08,.3);
      \opendot{0.08,0.07};
      \node at (0.25,0.05) {$\catlabel{\lambda}$};
\end{tikzpicture}
&\mapsto
\begin{tikzpicture}[anchorbase,scale=1.4]
	\draw[-to,thin] (0.08,-.15) to (0.08,.3);
      \opendot{0.08,0.07};
      \node at (0.25,0.05) {$\catlabel{\lambda}$};
\end{tikzpicture}
+a\:\,
\begin{tikzpicture}[anchorbase,scale=1.4]
	\draw[-to,thin] (0.08,-.15) to (0.08,.3);
      \node at (0.2,0.05) {$\catlabel{\lambda}$};
\end{tikzpicture}
,&
\begin{tikzpicture}[anchorbase,scale=1.4]
	\draw[to-,thin] (0.3,-0.1) to[out=90, in=0] (0.1,0.2);
	\draw[-,thin] (0.1,0.2) to[out = 180, in = 90] (-0.1,-0.1);
      \node at (0.5,0.05) {$\catlabel{\lambda}$};
\end{tikzpicture}
&\mapsto
\begin{tikzpicture}[anchorbase,scale=1.4]
	\draw[to-,thin] (0.3,-0.1) to[out=90, in=0] (0.1,0.2);
	\draw[-,thin] (0.1,0.2) to[out = 180, in = 90] (-0.1,-0.1);
      \node at (0.5,0.05) {$\catlabel{\lambda}$};
\end{tikzpicture},
&\begin{tikzpicture}[anchorbase,scale=1.4]
	\draw[to-,thin] (0.3,0.2) to[out=-90, in=0] (0.1,-0.1);
	\draw[-,thin] (0.1,-0.1) to[out = 180, in = -90] (-0.1,0.2);
      \node at (0.5,0.05) {$\catlabel{\lambda}$};
\end{tikzpicture}&
\mapsto
\begin{tikzpicture}[anchorbase,scale=1.4]
	\draw[to-,thin] (0.3,0.2) to[out=-90, in=0] (0.1,-0.1);
	\draw[-,thin] (0.1,-0.1) to[out = 180, in = -90] (-0.1,0.2);
      \node at (0.5,0.05) {$\catlabel{\lambda}$};
\end{tikzpicture},\\
\begin{tikzpicture}[anchorbase,scale=1.4]
	\draw[-to,thin] (0.18,-.15) to (-0.18,.3);
	\draw[-to,thin] (-0.18,-.15) to (0.18,.3);
      \node at (0.3,0.1) {$\catlabel{\lambda}$};
\end{tikzpicture}
&\mapsto
\begin{tikzpicture}[anchorbase,scale=1.4]
	\draw[-to,thin] (0.18,-.15) to (-0.18,.3);
	\draw[-to,thin] (-0.18,-.15) to (0.18,.3);
      \node at (0.3,0.1) {$\catlabel{\lambda}$};
\end{tikzpicture}
,&
\begin{tikzpicture}[anchorbase,scale=1.4]
	\draw[-,thin] (0.3,-0.1) to[out=90, in=0] (0.1,0.2);
	\draw[-to,thin] (0.1,0.2) to[out = 180, in = 90] (-0.1,-0.1);
      \node at (0.5,0.05) {$\catlabel{\lambda}$};
\end{tikzpicture}
&\mapsto
\begin{tikzpicture}[anchorbase,scale=1.4]
	\draw[-,thin] (0.3,-0.1) to[out=90, in=0] (0.1,0.2);
	\draw[-to,thin] (0.1,0.2) to[out = 180, in = 90] (-0.1,-0.1);
      \node at (0.5,0.05) {$\catlabel{\lambda}$};
\end{tikzpicture},&
\begin{tikzpicture}[anchorbase,scale=1.4]
	\draw[-,thin] (0.3,0.2) to[out=-90, in=0] (0.1,-0.1);
	\draw[-to,thin] (0.1,-0.1) to[out = 180, in = -90] (-0.1,0.2);
      \node at (0.5,0.1) {$\catlabel{\lambda}$};
\end{tikzpicture}
&\mapsto
\begin{tikzpicture}[anchorbase,scale=1.4]
	\draw[-,thin] (0.3,0.2) to[out=-90, in=0] (0.1,-0.1);
	\draw[-to,thin] (0.1,-0.1) to[out = 180, in = -90] (-0.1,0.2);
      \node at (0.5,0.1) {$\catlabel{\lambda}$};
\end{tikzpicture}.
\end{align*}
It also maps
\begin{align}\label{coffee1}
\begin{tikzpicture}[anchorbase]
  \draw[to-,thin] (0,0.4) to[out=180,in=90] (-.2,0.2);
  \draw[-,thin] (0.2,0.2) to[out=90,in=0] (0,.4);
 \draw[-,thin] (-.2,0.2) to[out=-90,in=180] (0,0);
  \draw[-,thin] (0,0) to[out=0,in=-90] (0.2,0.2);
     \node at (0.36,0.2) {$\catlabel{\lambda}$};
\opendot{-0.2,.2};
\node at (-.4,0.2) {$\kmlabel{n}$};
\end{tikzpicture}
&\mapsto
\sum_{r=0}^n
\binom{n}{r}
a^r
\begin{tikzpicture}[anchorbase]
  \draw[to-,thin] (0,0.4) to[out=180,in=90] (-.2,0.2);
  \draw[-,thin] (0.2,0.2) to[out=90,in=0] (0,.4);
 \draw[-,thin] (-.2,0.2) to[out=-90,in=180] (0,0);
  \draw[-,thin] (0,0) to[out=0,in=-90] (0.2,0.2);
     \node at (0.36,0.2) {$\catlabel{\lambda}$};
\opendot{-0.2,.2};
\node at (-.4,0.2) {$\kmlabel{r}$};
\end{tikzpicture}
,&
\begin{tikzpicture}[anchorbase]
  \draw[to-,thin] (0,0.4) to[out=180,in=90] (-.2,0.2);
  \draw[-,thin] (0.2,0.2) to[out=90,in=0] (0,.4);
 \draw[-,thin] (-.2,0.2) to[out=-90,in=180] (0,0);
  \draw[-,thin] (0,0) to[out=0,in=-90] (0.2,0.2);
      \node at (0,0.2) {$\kmlabel{n}$};
   \node at (.38,0.2) {$\catlabel{\lambda}$};
\end{tikzpicture}
&\mapsto
\sum_{r=0}^n \binom{-\lambda-r}{n-r}(-a)^{n-r}
\;\begin{tikzpicture}[anchorbase]
  \draw[to-,thin] (0,0.4) to[out=180,in=90] (-.2,0.2);
  \draw[-,thin] (0.2,0.2) to[out=90,in=0] (0,.4);
 \draw[-,thin] (-.2,0.2) to[out=-90,in=180] (0,0);
  \draw[-,thin] (0,0) to[out=0,in=-90] (0.2,0.2);
      \node at (0,0.2) {$\kmlabel{r}$};
   \node at (.38,0.2) {$\catlabel{\lambda}$};
\end{tikzpicture}
,\\
\begin{tikzpicture}[anchorbase]
  \draw[-to,thin] (0.2,0.2) to[out=90,in=0] (0,.4);
  \draw[-,thin] (0,0.4) to[out=180,in=90] (-.2,0.2);
\draw[-,thin] (-.2,0.2) to[out=-90,in=180] (0,0);
  \draw[-,thin] (0,0) to[out=0,in=-90] (0.2,0.2);
\opendot{0.2,.2};
\node at (.4,0.2) {$\kmlabel{n}$};
       \node at (-0.4,0.2) {$\catlabel{\lambda}$};
\end{tikzpicture}
&\mapsto
\sum_{r=0}^n
\binom{n}{r}
a^r
\begin{tikzpicture}[anchorbase]
  \draw[-to,thin] (0.2,0.2) to[out=90,in=0] (0,.4);
  \draw[-,thin] (0,0.4) to[out=180,in=90] (-.2,0.2);
\draw[-,thin] (-.2,0.2) to[out=-90,in=180] (0,0);
  \draw[-,thin] (0,0) to[out=0,in=-90] (0.2,0.2);
\opendot{0.2,.2};
\node at (.4,0.2) {$\kmlabel{r}$};
       \node at (-0.4,0.2) {$\catlabel{\lambda}$};
\end{tikzpicture}
,&
\begin{tikzpicture}[anchorbase]
  \draw[-to,thin] (0.2,0.2) to[out=90,in=0] (0,.4);
  \draw[-,thin] (0,0.4) to[out=180,in=90] (-.2,0.2);
\draw[-,thin] (-.2,0.2) to[out=-90,in=180] (0,0);
  \draw[-,thin] (0,0) to[out=0,in=-90] (0.2,0.2);
     \node at (0,0.2) {$\kmlabel{n}$};
   \node at (-.38,0.2) {$\catlabel{\lambda}$};
\end{tikzpicture}
&\mapsto
\sum_{r=0}^n \binom{\lambda-r}{n-r}(-a)^{n-r}
\begin{tikzpicture}[anchorbase]
  \draw[-to,thin] (0.2,0.2) to[out=90,in=0] (0,.4);
  \draw[-,thin] (0,0.4) to[out=180,in=90] (-.2,0.2);
\draw[-,thin] (-.2,0.2) to[out=-90,in=180] (0,0);
  \draw[-,thin] (0,0) to[out=0,in=-90] (0.2,0.2);
     \node at (0,0.2) {$\kmlabel{r}$};
   \node at (-.38,0.2) {$\catlabel{\lambda}$};
\end{tikzpicture}\label{coffee2}
\end{align}
for $n \geq 0$.
\end{lemma}

\begin{proof}
The existence of $\eta_a$ follows by checking relations. All of these are clear except for \cref{KMrels3}, and this is easy enough to see if one works with the equivalent form \cref{KMrels3equiv}.
In more detail, we note that
$$
\eta_a\left(\begin{tikzpicture}[anchorbase]
	\draw[-to,thin] (0,-.4) to (0,.4);
    \pinO{(0,0)}{-}{(.8,0)}{u};
   \node at (1.15,0) {$\catlabel{\lambda}$};
\end{tikzpicture}\right)=
\begin{tikzpicture}[anchorbase]
	\draw[-to,thin] (0,-.4) to (0,.4);
    \pinO{(0,0)}{-}{(.8,0)}{u-a};
   \node at (1.35,0) {$\catlabel{\lambda}$};
\end{tikzpicture}\:,
$$
because $\eta_a\left(u\:
\begin{tikzpicture}[anchorbase,scale=1]
	\draw[-to,thin] (0.08,-.15) to (0.08,.3);
      \node at (0.2,0.05) {$\catlabel{\lambda}$};
\end{tikzpicture}
 - 
\begin{tikzpicture}[anchorbase,scale=1]
	\draw[-to,thin] (0.08,-.15) to (0.08,.3);
      \opendot{0.08,0.07};
      \node at (0.25,0.05) {$\catlabel{\lambda}$};
\end{tikzpicture}\right)=(u-a)\:\begin{tikzpicture}[anchorbase,scale=1]
	\draw[-to,thin] (0.08,-.15) to (0.08,.3);
      \node at (0.2,0.05) {$\catlabel{\lambda}$};
\end{tikzpicture}
-\begin{tikzpicture}[anchorbase,scale=1]
	\draw[-to,thin] (0.08,-.15) to (0.08,.3);
      \opendot{0.08,0.07};
      \node at (0.25,0.05) {$\catlabel{\lambda}$};
\end{tikzpicture}$.
The relation  \cref{KMrels3equiv} follows using this and the observation that
$[f(u-a)]_{u^{-1}} = [f(u)]_{u^{-1}}$ for $f(u) \in \kk\lround u^{-1}\rround$.
To deduce the formulae describing $\eta_a$
on bubbles, we explain assuming $\lambda \geq 0$.
By the observation already made, we have that
$$
\eta_a\left(\begin{tikzpicture}[anchorbase]
\clockwisebubble{(0,0)};
\pinO{(-.2,0)}{-}{(-.9,0)}{u};
\node at (.4,0) {$\catlabel{\lambda}$};
\end{tikzpicture}
\right)
=
\begin{tikzpicture}[anchorbase]
\clockwisebubble{(0,0)};
\pinO{(-.2,0)}{-}{(-1,0)}{u-a};
\node at (.4,0) {$\catlabel{\lambda}$};
\end{tikzpicture}.
$$
Since $\lambda \geq 0$, we have that
$
\begin{tikzpicture}[anchorbase]
\filledclockwisebubble{(0,0)};
\node at (0,0) {$\kmlabel{u}$};
\node at (.4,0) {$\catlabel{\lambda}$};
\end{tikzpicture} = \delta_{\lambda,0} 1_{1_\lambda}$, and it follows easily that $\eta_a$ maps the generating function ${\catlabel{\lambda}\:}
\textclockwisebubble(u)$ from \cref{southwalespolice1} to
 ${\catlabel{\lambda}\:}
\textclockwisebubble(u-a)$.
Inverting, we deduce that $\eta_a$ maps
${\catlabel{\lambda}\:}
\textanticlockwisebubble(u)$ to
${\catlabel{\lambda}\:}
\textanticlockwisebubble(u-a)$.
The various formulae now follow by equating coefficients.
\end{proof}

Let 
$\iota:\fU(\sl_2) \rightarrow \fUc(\sl_2)$ be the canonical inclusion.
The 2-functor $\eta_a:\fU(\sl_2)\rightarrow \fU(\sl_2)$ 
maps
$$
\begin{tikzpicture}[anchorbase]
\draw[-to,thin] (-.6,-.3) to (-.6,.3);
\draw[-,thick] (0,-.3) to (0,.3);
\node at (0,-.5) {$\kmlabel{X}$};
\draw[-to,thin] (.6,-.3) to (.6,.3);
\limitbandOO{(-.6,0)}{}{}{(.6,0)};
\node at (0,.5) {$\phantom{\kmlabel{X}}$};
\node at (.9,0) {$\catlabel{\lambda}$};
\end{tikzpicture}
\mapsto
2a\:
\begin{tikzpicture}[anchorbase]
\draw[-to,thin] (-.6,-.3) to (-.6,.3);
\draw[-,thick] (0,-.3) to (0,.3);
\draw[-to,thin] (.6,-.3) to (.6,.3);
\node at (0,-.5) {$\kmlabel{X}$};
\node at (0,.5) {$\phantom{\kmlabel{X}}$};
\node at (.9,0) {$\catlabel{\lambda}$};
\end{tikzpicture}
+\begin{tikzpicture}[anchorbase]
\draw[-to,thin] (-.6,-.3) to (-.6,.3);
\draw[-,thick] (0,-.3) to (0,.3);
\draw[-to,thin] (.6,-.3) to (.6,.3);
\node at (0,-.5) {$\kmlabel{X}$};
\opendot{-.6,0};
\node at (0,.5) {$\phantom{\kmlabel{X}}$};
\node at (.9,0) {$\catlabel{\lambda}$};
\end{tikzpicture}
+
\begin{tikzpicture}[anchorbase]
\draw[-to,thin] (-.6,-.3) to (-.6,.3);
\draw[-,thick] (0,-.3) to (0,.3);
\node at (0,-.45) {$\kmlabel{X}$};
\draw[-to,thin] (.6,-.3) to (.6,.3);
\opendot{.6,0};
\node at (0,.45) {$\phantom{\kmlabel{X}}$};
\node at (.9,0) {$\catlabel{\lambda}$};
\end{tikzpicture}.
$$
Assuming that $a \in \kk^\times$, 
this 2-morphism is invertible 
in the completion $\fUc(\sl_2)$, hence, 
the composition $\iota\circ\eta_a$ 
extends uniquely to a strict  2-functor
\begin{equation}\label{zeta}
\overline{\iota\circ\eta}_a:\fU(\sl_2)_\loc \rightarrow \fUc(\sl_2).
\end{equation}
Finally, for $t \in \{0,1\}$,
we let $\cUc(\sl_2;t)$ be the strict monoidal category defined by collapsing $\fUc(\sl_2)$ in exactly the same way as in \cref{hearts},
replacing the localization $\fU(\sl_2)_\loc$ with the completion $\fUc(\sl_2)$. The 2-functor \cref{zeta} induces
a strict monoidal functor
\begin{equation}\label{zeta2}
\zeta_a:\cU(\sl_2;t)_\loc \rightarrow \cUc(\sl_2;t).
\end{equation}
We denote its natural extension to the additive envelopes of these categories by
 \begin{equation}\label{zeta3}
 \zeta_a^+:\Add\left(\cU(\sl_2;t)_{\loc}\right)
\rightarrow \Add\big(\cUc(\sl_2;t)\big).
\end{equation}
This functor will be useful as it is easier to work with $\cUc(\sl_2;t)$ than with $\cU(\sl_2;t)_\loc$, since we can exploit the topological basis arising from \cref{nondegeneracy}.

\begin{proof}[Proof of linear independence part of \cref{basisthm}]
We first use the pivotal structure to make a standard reduction:
Closing diagrams on the left
$$
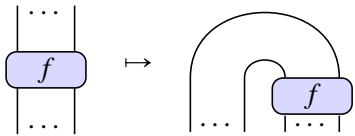
\begin{tikzpicture}[baseline=0mm,scale=.95]
\draw (-.4,-.9) to (-.4,-.15);
\draw (.4,-.9) to (.4,-.15);
\draw (-.4,.9) to (-.4,.15);
\draw (.4,.9) to (.4,.15);
\node at (0,-.8) {$\dots$};
\node at (0,.8) {$\dots$};
\node[rectangle,rounded corners,draw,fill=blue!15!white,inner sep=2pt] at (0,0) {$\hspace{3.5mm}f\hspace{3.5mm}$};
\end{tikzpicture}
\quad\mapsto\quad
\begin{tikzpicture}[anchorbase,scale=.9]
%\node at (0,.3) {$\dots$};
\node at (0,-.3) {$\dots$};
\node at (-1.4,-.3) {$\dots$};
\draw (-.4,-.4) to (-.4,-.16);
\draw (.4,-.4) to (.4,-.16);
\draw (-.4,.4) to (-.4,.16);
\draw (.4,.4) to (.4,.16);
\draw (.4,.4) to[out=90,in=90,looseness=1.5] (-1.8,.4) to (-1.8,-.4);
\draw (-.4,.4) to[out=90,in=90,looseness=1.5] (-1,.4) to (-1,-.4);
\node[rectangle,rounded corners,draw,fill=blue!15!white,inner sep=2pt] at (0,0.13) {$\hspace{3.5mm}f\hspace{3.5mm}$};
\end{tikzpicture}
$$
defines a $\Gamma$-module isomorphism 
$\Hom_{\cNB_t}(B^{\star n}, B^{\star m})
\stackrel{\sim}{\rightarrow} \Hom_{\cNB_t}(B^{\star (m+n)}, \one)$. Thus, the proof is  reduced to the special case that $m=0$, which we assume from now on.

Take a linear relation
\begin{equation}\label{math230}
\sum_{s \in \RSD(n,0)} s \cdot p_s = 0
\end{equation}
in $\Hom_{\cNB_t}(B^{\star n}, \one)$
for $p_s \in \Gamma$.
We must show that $p_s = 0$ for all $s$.
Suppose not and choose $u \in \RSD(n,0)$ with a maximal number of crossings
such that $p_u \neq 0$.
Let $Y$ be the word in $\langle E, F \rangle$ obtained
by orienting the generalized caps in $u$ from left to right then reading
the orientations of the boundary points using the usual dictionary $E = $ upward and $F = $ downward.
Also take any $\lambda \in \Z$.
We apply the monoidal functor
$\zeta_a^+ \circ \Omega_t$
to \cref{math230} to obtain a
morphism in $\Add\big(\cUc(\sl_2;t)\big)$, then restrict this to
$Y 1_\lambda$ to obtain a linear relation
\begin{equation*}
\sum_{s \in \RSD(n,0)}
\zeta_a^+(\Omega_t(s\cdot p_s))|_{Y 1_\lambda}
= 0
\end{equation*}
in $\Hom_{\fUc(\sl_2)}(Y 1_\lambda,1_\lambda)$.
By \cref{magic}, we deduce that
\begin{equation}\label{math231}
\sum_{s \in \RSD(n,0)}
\zeta_a^+(\Omega_t(s))|_{Y 1_\lambda}
\cdot \zeta_a(p_s)= 0,
\end{equation}
where $\zeta_a(p_s)$ denotes the image of $p_s \in \Gamma \subseteq \Lambda$ under the automorphism $\zeta_a$ of $\Lambda\equiv\End_{\fU(\sl_2)}(1_\lambda)$ described by \cref{coffee1,coffee2}.
In particular, this completes the proof in the special case $n=0$.

Now we need to think more carefully about the 2-morphisms
$\zeta_a^+(\Omega_t(s))|_{Y 1_\lambda}$ arising in \cref{math231}.
It is simply 0 if $s$ is not admissible for $Y$.
Assuming $s$ is admissible,
the definition of $\Omega_t$ given in \cref{psit} implies that
$\zeta_a^+(\Omega_t(s))|_{Y 1_\lambda}$
is a topological
sum of morphisms defined by oriented string diagrams 
with the same number or with fewer crossings compared to $s$, with all the ones
with the same number of crossings being of the same underlying type as $s$; this sum here may be infinite since
the images under $\zeta_a$ of
internal bubbles and teleporters are infinite linear combinations of diagrams with extra dots and internal dotted bubbles.
Using the straightening algorithm sketched in the spanning part of the proof, diagrams with the same number of crossings but a different type
to $s$ and diagrams with strictly fewer crossings than $s$
can be rewritten as a $\Lambda$-linear combination of
basis vectors from $\overrightarrow{\RSD}(Y 1_\lambda,1_\lambda)$, 
 all of which either have the same number of crossings but a different type to $s$ as before or have fewer crossings than $s$.
Now let 
$$
\X
:= \{s \in \RSD(n,0)\:|\:s \sim u\}.
$$
For $s \in \X$, all of the generalized caps in
 $\vec{s}\;{\scriptstyle\color{catcolor}\lambda}$ 
 are oriented from left to right, hence, this diagram
 only involves upward or rightward crossings. Using the definition of $\Omega_t$,
it follows that
$\zeta_a^+(\Omega_t(s))|_{Y 1_\lambda} = 
\zeta_a(\vec{s}\;{\scriptstyle\color{catcolor}\lambda})+($a topological linear combination of basis vectors with strictly fewer crossings$)$.
Putting these two points together, it follows that
$$
\sum_{s \in \RSD(n,0)}
\zeta_a^+(\Omega_t(s))|_{Y 1_\lambda}\cdot \zeta_a(p_s)
= \sum_{s \in \X} \zeta_a(\vec{s}\;{\scriptstyle\color{catcolor}\lambda})\cdot \zeta_a(p_s) + (*) = 0,
$$
where $(*)$ is a topological linear combination of basis vectors
with the same number of crossings as $u$ but a different type or with strictly fewer crossings.
In view of \cref{nondegeneracy}, it follows that
 $$
 \sum_{s \in \X} \zeta_a(\vec{s}\;{\scriptstyle\color{catcolor}\lambda})\cdot \zeta_a(p_s) = 0$$ 
 in $\Hom_{\fU(\sl_2)}(Y 1_\lambda,1_\lambda)$, 
 hence,
$\sum_{s \in \X} (\vec{s}\;{\scriptstyle\color{catcolor}\lambda})\cdot p_s = 0$.
We deduce that $p_s = 0$
for all $s \in \X$, in particular, $p_u= 0$. This contradiction completes the proof
of \cref{basisthm}.
\end{proof}

\begin{corollary}\label{bubblealgebra}
The algebra homomorphism $\gamma_t:\Gamma \rightarrow \End_{\cNB_t}(\one)$ is an isomorphism.
\end{corollary}

\begin{proof}
This follows from the $m=n=0$ case of \cref{basisthm}, since
$\RSD(0 , 0)$ is a singleton.
\end{proof}

%===========
% References
%===========

\bibliographystyle{alpha}
\bibliography{nilbrauer-corrected}

\begin{thebibliography}{BWW23}

\bibitem[AMR06]{AMR}
Susumu Ariki, Andrew Mathas, and Hebing Rui.
\newblock Cyclotomic {N}azarov-{W}enzl algebras.
\newblock {\em Nagoya Math. J.}, 182:47--134, 2006.

\bibitem[BD17]{BD}
Jonathan Brundan and Nicholas Davidson.
\newblock Categorical actions and crystals.
\newblock In {\em Categorification and higher representation theory}, volume
  683 of {\em Contemp. Math.}, pages 105--147. Amer. Math. Soc., Providence,
  RI, 2017.

\bibitem[BE17]{BE}
Jonathan Brundan and Alexander~P. Ellis.
\newblock Monoidal supercategories.
\newblock {\em Comm. Math. Phys.}, 351(3):1045--1089, 2017.

\bibitem[Bru16]{Brundan}
Jonathan Brundan.
\newblock On the definition of {K}ac-{M}oody 2-category.
\newblock {\em Math. Ann.}, 364:353--372, 2016.

\bibitem[BSW20]{HKM}
Jonathan Brundan, Alistair Savage, and Ben Webster.
\newblock Heisenberg and {K}ac-{M}oody categorification.
\newblock {\em Selecta Math. (N.S.)}, 26:Paper No. 74, 62, 2020.

\bibitem[BSW21]{Foundations}
Jonathan Brundan, Alistair Savage, and Ben Webster.
\newblock Foundations of {F}robenius {H}eisenberg categories.
\newblock {\em J. Algebra}, 578:115--185, 2021.

\bibitem[BSW23]{K0}
Jonathan Brundan, Alistair Savage, and Ben Webster.
\newblock The degenerate {H}eisenberg category and its {G}rothendieck ring.
\newblock {\em Ann. Sci. Éc. Norm. Supér. (4)}, 56(5):1517--1563, 2023.

\bibitem[BW18a]{BW18QSP}
Huanchen Bao and Weiqiang Wang.
\newblock Canonical bases arising from quantum symmetric pairs.
\newblock {\em Invent. Math.}, 213(3):1099--1177, 2018.

\bibitem[BW18b]{BW18KL}
Huanchen Bao and Weiqiang Wang.
\newblock A new approach to {K}azhdan-{L}usztig theory of type {$B$} via
  quantum symmetric pairs.
\newblock {\em Ast\'{e}risque}, 402:vii+134, 2018.

\bibitem[BW18c]{BeW18}
Collin Berman and Weiqiang Wang.
\newblock Formulae of {$\imath$}-divided powers in {$U_q(\mathfrak{sl}_2)$}.
\newblock {\em J. Pure Appl. Algebra}, 222(9):2667--2702, 2018.

\bibitem[BWW23]{BWW}
Jonathan Brundan, Ben Webster, and Weiqiang Wang.
\newblock Nil-{B}rauer categorifies the split iquantum group of rank one.
\newblock \arxiv{2305.05877}, 2023.

\bibitem[Car00]{Carpentier}
Rui~Pedro Carpentier.
\newblock From planar graphs to embedded graphs---a new approach to {K}auffman
  and {V}ogel's polynomial.
\newblock {\em J. Knot Theory Ramifications}, 9(8):975--986, 2000.

\bibitem[DRV13]{DRV}
Zajj Daugherty, Arun Ram, and Rahbar Virk.
\newblock Affine and degenerate affine {BMW} algebras: actions on tensor space.
\newblock {\em Selecta Math. (N.S.)}, 19(2):611--653, 2013.

\bibitem[ES18]{ES}
Michael Ehrig and Catharina Stroppel.
\newblock Nazarov-{W}enzl algebras, coideal subalgebras and categorified skew
  {H}owe duality.
\newblock {\em Adv. Math.}, 331:58--142, 2018.

\bibitem[KL10]{KL3}
Mikhail Khovanov and Aaron~D. Lauda.
\newblock A categorification of quantum {${\rm sl}(n)$}.
\newblock {\em Quantum Topol.}, 1:1--92, 2010.

\bibitem[Lau10]{Lauda}
Aaron~D. Lauda.
\newblock A categorification of quantum {${\rm sl}(2)$}.
\newblock {\em Adv. Math.}, 225(6):3327--3424, 2010.

\bibitem[Lau11]{Lauda2}
Aaron~D. Lauda.
\newblock Categorified quantum {$\rm sl(2)$} and equivariant cohomology of
  iterated flag varieties.
\newblock {\em Algebr. Represent. Theory}, 14(2):253--282, 2011.

\bibitem[Mac15]{Mac}
I.~G. Macdonald.
\newblock {\em Symmetric {F}unctions and {H}all {P}olynomials}.
\newblock Oxford Classic Texts in the Physical Sciences. The Clarendon Press,
  Oxford University Press, New York, second edition, 2015.

\bibitem[Naz96]{Nazarov}
Maxim Nazarov.
\newblock Young's orthogonal form for {B}rauer's centralizer algebra.
\newblock {\em J. Algebra}, 182(3):664--693, 1996.

\bibitem[Rou08]{Rou}
Rapha\"el Rouquier.
\newblock $2$-{K}ac-{M}oody algebras.
\newblock \arxiv{0812.5023}, 2008.

\bibitem[RS19]{RS}
Hebing Rui and Linliang Song.
\newblock Affine {B}rauer category and parabolic category {$\mathcal O$} in
  types {$B$}, {$C$}, {$D$}.
\newblock {\em Math. Z.}, 293:503--550, 2019.

\end{thebibliography}
\end{document}